\documentclass{amsart}
\usepackage[utf8]{inputenc}

\usepackage{fullpage}

	\usepackage{amsmath}
 \usepackage{amsfonts}
	\usepackage{amssymb}
	\usepackage{amsthm}
	\usepackage{braket}
	\usepackage{xcolor}
\usepackage[shortlabels]{enumitem}

 \usepackage{float}
	\usepackage{graphicx}
	\usepackage{hyperref}
    \usepackage{cleveref}
	\usepackage{ifsym}
	\usepackage{indentfirst}
	\usepackage{mathrsfs}
	\usepackage{mathtools}
	\usepackage{proof}
	\usepackage{qtree}
	\usepackage{setspace}
	\usepackage{tensor}
 \usepackage{tabularx}
	\usepackage{tikz}
	\usepackage{tikz-cd}
	\usepackage{tabstackengine}
        \usepackage{cancel}
\makeatletter
\newsavebox{\@abr}
\newcommand{\llangle}[1][]{\savebox{\@abr}{\(\m@th{#1\langle}\)}%
  \mathopen{\copy\@abr\mkern2mu\kern-0.9\wd\@abr\usebox{\@abr}}}
\newcommand{\rrangle}[1][]{\savebox{\@abr}{\(\m@th{#1\rangle}\)}%
  \mathclose{\copy\@abr\mkern2mu\kern-0.9\wd\@abr\usebox{\@abr}}}

  \newsavebox{\@sbr}
\newcommand{\llsquare}[1][]{\savebox{\@sbr}{\(\m@th{#1{[}}\)}%
  \mathopen{\copy\@sbr\mkern3mu\kern-0.9\wd\@sbr\usebox{\@sbr}}}
\newcommand{\rrsquare}[1][]{\savebox{\@sbr}{\(\m@th{#1{]}}\)}%
  \mathclose{\copy\@sbr\mkern3mu\kern-0.9\wd\@sbr\usebox{\@sbr}}}
\makeatother

\newcommand{\medmatrix}[2][.8]{%
  \scalebox{#1}{%
    \renewcommand{\arraystretch}{.8}%
    $\begin{pmatrix}#2\end{pmatrix}$%
  }
}

	\providecommand{\corollaryname}{Corollary}
	\providecommand{\definitionname}{Definition}
	\providecommand{\examplename}{Example}
	\providecommand{\lemmaname}{Lemma}
	\providecommand{\propositionname}{Proposition}
	\providecommand{\remarkname}{Remark}
	\providecommand{\theoremname}{Theorem}
	\providecommand{\setupname}{Setup}
	\providecommand{\conjecturename}{Conjecture}
	\providecommand{\questionname}{Question}
	\providecommand{\objectivename}{Objective}
	\providecommand{\aimname}{Aim}
	\providecommand{\notationname}{Notation}

	\theoremstyle{plain}
		\newtheorem{thm}{\protect\theoremname}[section] % reset thm numbering for each section
		\newtheorem{prop}[thm]{\protect\propositionname}
		\newtheorem{lem}[thm]{\protect\lemmaname}
		\newtheorem{cor}[thm]{\protect\corollaryname}

	\theoremstyle{definition}
		\newtheorem{defn}[thm]{\protect\definitionname}
		\newtheorem{example}[thm]{\protect\examplename}
		\newtheorem{setup}[thm]{\protect\setupname}

		\newtheorem{notn}[thm]{\protect\notationname}

	\theoremstyle{remark}
		\newtheorem{rem}[thm]{\protect\remarkname}
		
	\numberwithin{figure}{section}
	\numberwithin{equation}{section}

	\usetikzlibrary{matrix,arrows,decorations.pathmorphing,positioning,decorations.pathreplacing}
	\tikzset{commutative diagrams/.cd, 
		mysymbol/.style = {start anchor=center, end anchor = center, draw = none}}
    \tikzset{
    labl/.style={anchor=north, rotate=90, inner sep=2mm}
    }

	\tikzcdset{every label/.append style = {font = \footnotesize}}
	\tikzcdset{scale cd/.style={every label/.append style={scale=#1},
    cells={nodes={scale=#1}}}}

		\newcommand{\rad}{\operatorname{rad}\nolimits}
		\newcommand{\soc}{\operatorname{soc}\nolimits}

		\newcommand{\im}{\operatorname{im}\nolimits}
		
\newcommand{\rMod}[1]{\operatorname{\mathsf{Mod}\hspace{0.2mm}-\hspace{-0.4mm}}\nolimits{#1}}
\newcommand{\lMod}[1]{{#1}\operatorname{\hspace{-0.4mm}-\hspace{0.2mm}\mathsf{Mod}}\nolimits}

\newcommand{\lmod}[1]{{#1}\operatorname{\hspace{-0.4mm}-\hspace{0.2mm}\mathsf{mod}}\nolimits}
\newcommand{\rfl}[2]{\tensor*[_{#1}]{\operatorname{\mathsf{fl}\hspace{0.2mm}-\hspace{-0.4mm}}\nolimits{#2}}{}}
\newcommand{\lfl}[2]{\tensor*[]{{#1}\operatorname{\hspace{-0.4mm}-\hspace{0.2mm}\mathsf{fl}}\nolimits}{_{#2}}}

\newcommand{\leftmult}[1]{\tensor*[_{#1}]{\mathsf{m}\hspace{-0.3mm}}{}}

\newcommand{\rightmult}[1]{\tensor*[]{\hspace{-0.3mm}\mathsf{m}}{_{#1}}}

  \newcommand{\Hom}[1]{\tensor*[]{\operatorname{Hom}}{_{#1}}}
  \newcommand{\End}[1]{\tensor*[]{\operatorname{End}}{_{#1}}}
  \newcommand{\Ext}[2]{\tensor*[]{\operatorname{Ext}}{^{#1}_{#2}}}

	\newcommand{\maxideal}{\mathfrak{m}}
 	\newcommand{\jacrad}{\mathcal{J}}
    \newcommand{\env}{\mathfrak{a}}
    \newcommand{\homdual}[2]{\tensor*[]{\operatorname{\mathscr{C}}}{_{\hspace{-0.45mm}#1}^{#2}}}
    \newcommand{\matdual}[2]{\tensor*[]{\operatorname{\mathscr{D}}}{_{\hspace{-0.45mm}#1}^{#2}}}

\newenvironment{acknowledgements}{
		\begin{abstract}} {\end{abstract}}

	\makeatletter
\DeclareRobustCommand{\rvdots}{%
  \vbox{
    \baselineskip4\p@\lineskiplimit\z@
    \kern-\p@
    \hbox{.}\hbox{.}\hbox{.}
  }}

\newcommand{\compactdhat}{\skew{1}\widehat{\rule{0ex}{1.4ex}\smash{\mathbb{Z}}}}

\makeatother

\title{String algebras over local rings: 
admissibility  and biseriality}
\author{Raphael Bennett-Tennenhaus}
\address{Raphael Bennett-Tennenhaus\newline
		Department of Mathematics,\newline 
		Aarhus University, 
		Ny Munkegade 118,  
		8000 Aarhus C,  
		Denmark}
\email{raphaelbennetttennenhaus@gmail.com}
\date{}


\begin{document}

\maketitle

\begin{abstract} 
For a path algebra over a noetherian local ground ring, the notion of an  admissible ideal was defined  by Raggi-C{\'a}rdenas and  Salmer{\'o}n.  
We provide sufficient conditions for admissibility and use them to study  semiperfect module-finite  algebras over local rings   whose quotient by the radical is a product of copies of the residue field. 
We define string algebras over local ground rings and recover the notion  introduced by Butler and Ringel when the ground ring is a field. 
We prove they are biserial in a sense of  Kiri\v{c}enko and Kostyukevich. 
We describe the syzygies of uniserial summands of the radical. 
We give examples of B\"{a}ckstr\"{o}m orders that are string algebras over discrete valuation rings. 
% We observe  properties of the corresponding uniserial modules.
\end{abstract}
% LEEDS TALK ABSTRACT String algebras are usually quotients of path algebras over fields. Path algebras have also been considered over any local ground ring. Raggi-Cardenas and Salmeron gave a definition of admissible ideals in this context, that specialises to the usual notion when the ground ring is a field.  A generalisation of string algebras from my PhD thesis likewise replaced the ground field with a local ring. In this talk I will show how to produce examples of string algebras over regular local rings. These examples require the Krull dimension of the ground ring to coincide with a combinatorial invariant that is defined by cycles in the quiver that avoid relations. This construction recovers a plethora of examples from the literature. I will then explain how string algebras over local rings are quotients by ideals that are admissible in the aforementioned sense of Raggi-Cardenas and Salmeron. Time permitted, I will show how these quotients are biserial rings in a sense used by Kiri\v{c}enko and Yaremenko.   

\section{Introduction}
\label{section-intro}
\emph{String algebras} were defined by Butler and Ringel \cite{ButRin1987} and have the form $kQ/I$ where $kQ$ is a path algebra of a quiver $Q$ over a ground field $k$, $I$ is  admissible  and other compatibility conditions hold. 
% We define string algebras over noetherian local ground rings, recovering the notion from \cite{ButRin1987} when the ground ring is a field\footnote{Here we assume $Q$ is finite. See \Cref{prop:butler-ringel-string-for-a-field}. }. 
% Along the way we study path algebras over such ground rings. 
% Noting that $Q$ can be infinite in \cite{ButRin1987}, we only deal with unital rings. 
% So we must and we do assume $Q$ is finite to recover the prototypical examples of string algebras over a field. 
In \cite{ButRin1987}  an explicit description is given for the finite-dimensional indecomposable modules, the irreducible maps between them and  the Auslander--Reiten sequences. 
See also work of Gel'fand and Ponomarev \cite{GelPon1968} and Wald and Waschb\"{u}sch  \cite{WalWas1985}. 
These results have been prolific in representation theory. 

On the one hand, already \cite{WalWas1985} used the classification of modules over string algebras to obtain a similar classification\footnote{But for the finite-dimensional projective-injective indecomposable modules.} for the special-biserial algebras introduced by Skowro\'{n}ski and Waschb\"{u}sch \cite{SkoWas1983}. 
Representations for  other interesting classes of algebras have also been studied in terms of string algebras; see for example \cite{Calderon-Giraldo-Universal-deformation-rings-of-string-modules-over-certain-class-of-self-injective-special-biserial-algebras, Cra2018, Erdmann-Skowronski-algebras-of-generalised-dihedral-type, Geiss-Labardini-Fragoso-Schroer-schemes-of-modules-over-gentle-algebras-and-laminations-of-surfaces, Leszczynski-skowronski-tame-generalised-canonical-algebras}. 
On the other hand, the gentle algebras introduced by Assem and Skowro\'{n}ski \cite{AssSko1986} form a particular class of better-behaved string algebras. 
The aforementioned results from \cite{ButRin1987} have been used extensively to study combinatorial, geometric and homological properties of gentle algebras; see for example \cite{Baur-Coelho-Simoes-A-geometric-model-for-the-module-category-of-a-gentle-algebra, BekMer2003, Canakci-Pauksztello-Schroll-on-extensions-for-gentle-algebras, Kalck-Singularity-categories-of-gentle-algebras, Schroer-Modules-without-self-extensions-over-gentle-algebras}. 
Hence string algebras over fields are both important and well-understood. 

\emph{Module}-\emph{finite}\footnote{An $R$-algebra $\Gamma$ is \emph{module}-\emph{finite} (over $R$) provided it is finitely generated when considered as an $R$-module. } algebras over local rings have been looked at from several representation-theoretic perspectives. 
Gordan and Green \cite{GorGre1976} gave a representation theory for noether algebras to mimic that of artin algebras. 
Ringel and Roggenkamp \cite{RinRog1979} focused on representations of orders using valued graphs. 
Auslander \cite{Aus1986} looked at almost-split sequences for isolated singularities. 
Pave\v{s}i\'{c} \cite{Pav2010} studied module-finite algebras over complete local rings as a prototypical class of semiperfect rings. 
Niu \cite{Niu2014} considered almost-split sequences and triangles for such algebras.  
Ueyama \cite{Uey2013} studied graded Gorenstein isolated singularities. 
Gorenstein algebras appear also in more recent work of Iyengar and  Krause \cite{IyeKra2020}, Gnedin \cite{Gnedi2022}, and Esentepe \cite{esentepe2022note}. 

\emph{Path algebras} over a  commutative ring $R$ are denoted $RQ$ and defined by the $R$-span of the paths in a quiver  $Q$. 
% , or equivalently modules over the path algebra $RQ$, 
%\footnote{Modules over $RQ$ are the same as representations of $Q$ over $R$, meaning diagrams of shape $Q$ in the category of $R$-modules. }
Modules over $RQ$ 
% correspond to representations of $Q$ over $R$, meaning $Q$-shaped diagrams in the category of $R$-modules. 
% These objects 
have been studied before using a variety of approaches.  
For example Schofield \cite{Scho1992} considered questions about dimension vectors, Crawley-Boevey \cite{Cra1996,crawley-boevey-rigid-integral-revisited}  studied rigid lattices and 
Ortega \cite{Orte2006} looked at poset representations and incidence algebras. 
Sometimes additional restrictions on $R$ are imposed; see for example work of  
Miyamoto \cite{Miyamoto-On-the-non-periodic-stable-Auslander--Reiten-Heller-component-for-the-Kronecker-algebra-over-a-complete-discrete-valuation-ring}, Lee and Voll \cite{LeeVol2020}, 
Stangle \cite{Sta2017}, 
Hausel, Letellier and Rodriguez Villegas \cite{HauLetRod2018},  
Ringel and Zhang \cite{RinZha2011} and Gei{\ss},  Leclerc and Schr{\"o}er \cite{GeiLecSchro2014,geiss-leclerc-schroer-rigid-modules-and-schur-roots}. 

In this article we take $R$ to be a noetherian local ring, and study instead  quotients rings of the form $RQ/I$ that are finitely generated as $R$-modules. 
Here the ideal $I$ is \emph{admissible} in the sense discussed below.

\subsection{Admissibility}
\label{subsec-admissibility}

\emph{Admissible ideals} have been defined for path algebras over local rings in work of Raggi-C{\'a}rdenas and  Salmer{\'o}n  \cite{CarSal1987}. We recall and follow  \cite{CarSal1987} closely. 
See also work of Pe{\~n}a and Raggi-C{\'a}rdenas  \cite{PenRag1988}. 
We generalise\footnote{See \Cref{comparing-example} for a comparison of the definitions. } these notions and provide sufficient conditions for admissibility. 
In these terms one may characterise semiperfect basic module-finite elementary algebras over local rings; see \Cref{thm-characterising-semiperfect-mod-fin-alg-with-brick-simples} and \Cref{thm-admissible-presentations}. 
In \cite{PenRag1988, CarSal1987} both sets of authors also assume that $R$ is complete in the $\maxideal$-adic topology, and take the residue field $k$ to be algebraically closed.

Let $Q$ be a finite quiver. 
Let $R$ be a commutative noetherian local ring  with unique maximal ideal $\maxideal$ and residue field  $k=R/\maxideal$. 
Let $RQ$ be the path algebra. 
Fix a (two-sided) ideal $I\triangleleft RQ$ and  let  $\Lambda =RQ/I$. 
    % These conditions are written in terms of the ideal $A+I/I$ in $\Lambda$.     
    % So, if $I$ is permissible for a set $Z$ of paths then the paths defining non-zero elements of $\Lambda$ are those that do not have a subpath in $Z$. 
    \begin{itemize}
    \item $I$ is said to be \emph{permissible} if $a\notin I$, or equivalently $\Lambda a\neq 0\neq a\Lambda$, for any arrow $a$. 
    \item $I$ is said to be \emph{arrow}-\emph{radical} if (i) and (ii) below hold for any vertex $v$  with trivial path $e_{v}$. 
        \begin{enumerate}[label={\upshape(\roman*)}]
        \item $\Lambda e_{v}$ is local\footnote{Recall a module $M$ is local if it has a unique maximal submodule, that must be equal to its jacobson radical $\rad(M)$. 
        In (i) and (iii) the submodules involved are left $\Lambda$-submodules, and in (ii) and (iv) they are right $\Lambda$-submodules. 
        } with 
        $\rad(\Lambda e_{v})=\sum \Lambda a$ where the sum runs through arrows $a$ with tail $t(a)=v$. 
        \item $e_{v}\Lambda $ is local  with 
        $\rad(e_{v}\Lambda )=\sum a\Lambda $ where the sum runs through arrows $a$ with head $h(a)=v$. 
        \end{enumerate}
    \item $I$ is said to be  \emph{arrow}-\emph{distinct} if (iii) and (iv) below hold for any arrow $a$. 
        \begin{enumerate}[label={\upshape(\roman*)}]
        \setcounter{enumi}{2}
        \item $\Lambda a\cap \sum \Lambda b\subseteq \rad(\Lambda a)$ where the sum runs through the arrows $b\neq a$ with $t(b)=t(a)$. 
    \item $a\Lambda \cap \sum b\Lambda \subseteq \rad(a\Lambda )$ where the sum runs through the arrows $b\neq a$ with $h(b)=h(a)$. 
        \end{enumerate}
    % where $Z$ is some given subset of the paths in $Q$, if the set of paths that lie in $I$ is equal to the set of paths with a subpath in $Z$. 
\end{itemize}
A ring $\Gamma$ with jacobson radical $\jacrad$ is \emph{semiperfect} provided the quotient ring $\overline{\Gamma}=\Gamma/\jacrad$ is semisimple artinian and provided that any element of $\Gamma$ that defines an idempotent in $\overline{\Gamma}$ 
 lifts to an idempotent in $\Gamma$. 
A semiperfect ring $\Gamma$ is \emph{basic} if $\overline{\Gamma}$ is a product of division rings. 
An \emph{module}-\emph{finite} $R$-algebra is an $R$-algebra that is  finitely generated as an $R$-module. 
A semiperfect and basic module-finite $R$-algebra $\Gamma$ is \emph{elementary} provided $\overline{\Gamma}$ is a product of copies of $k$. 
Such algebras are characterised as follows. 
\begin{thm}
    \label{thm-characterising-semiperfect-mod-fin-alg-with-brick-simples} 
    Any semiperfect, module-finite and elementary  $R$-algebra has the form $RQ/I$ for some quiver $Q$ and some permissible, arrow-radical, and  arrow-distinct ideal $I\triangleleft RQ$. 
\end{thm}

% Our notion of an \emph{admissible} ideal is based on and generalises notions from work of Raggi-C{\'a}rdenas and Salmer{\'o}n  \cite{CarSal1987} and Pe{\~n}a and Raggi-C{\'a}rdenas \cite{PenRag1988}.  

Let $\maxideal Q$ be the $R$-submodule of $RQ$ consisting of $R$-linear combinations of paths with coefficients in $\maxideal$. 
For each integer $\ell\geq 0$ let $A_{\ell}\subseteq RQ$ be the $R$-span of the paths of length $\ell$. 
Let $A=\bigoplus_{\ell\geq 1}A_{\ell}$.

\begin{itemize}
    \item $I$ is said to be \emph{admissible}  if it is \emph{bounded above} and \emph{bounded below}  in the following sense. 
        \begin{enumerate}[label={\upshape(\alph*)}]
        % \item $re_{v}\notin I$ whenever $0\neq r\in R$ and $v$ is a vertex. 
        \item $I$ is said to be \emph{bounded above} if $I\subseteq A+ \maxideal Q$ and $I\cap A \subseteq A^{2}+ \maxideal Q\cap A$.
        \item $I$ is said to be \emph{bounded below} if $A^{m}\subseteq I+ \maxideal Q$ and $\maxideal Q\subseteq I+\sum_{\ell=1}^{n} A_{\ell}$ for some integers $m,n>0$. 
        \end{enumerate} 
    \end{itemize}
If  $I$ is admissible then $RQ/I\cong SQ/J$ where $S$ is a quotient of $R$ and  $J\triangleleft SQ$ is admissible in the stricter sense 
 used in \cite{CarSal1987}; see \Cref{prop-admissible-presentations-extra}.   
% \Cref{thm-admissible-presentations} gives a way of relating the terms we are using. 
In \S\ref{quiver-noetherian-semiperfect} and \S\ref{sec-quiv-complete-case} we recall the \emph{Gabriel quiver} of a basic module-finite algebra over a complete local ring. 
See \Cref{def-gabriel-quiver}. 
It turns out that the notions of being arrow-radical, arrow-distinct and permissible combine to give sufficient conditions for an ideal to be admissible.  

\begin{thm}
\label{thm-admissible-presentations}
Let $I\triangleleft RQ$ be arrow-radical,  $\Lambda=RQ/I$ and  $\jacrad=\rad(\Lambda)$. 
The following statements hold. 
\begin{enumerate}
        \item $\Lambda$ is semiperfect and basic with pair-wise non-isomorphic local  idempotents $e_{v}+I$ \emph{(}$v\in Q_{0}$\emph{)}.  
    \item  $\Lambda$ is finitely generated as an $R$-module if and only if the ideal $I$ is bounded below. 
    \item If $I$ is arrow-distinct then  $I$ is permissible if and only if $I$ is bounded above. 
    \item If $I$ is bounded above  then $\Lambda/\jacrad\cong k^{\vert Q_{0}\vert}$ as rings and $\jacrad/\jacrad^{2}\cong k^{\vert Q_{1}\vert}$ as $k$-vector spaces. 
    \item If $I$ is admissible and $R$ is $\maxideal$-adically complete then $Q$ is the Gabriel quiver of $\Lambda$. 
\end{enumerate}
\end{thm}

Combining  \Cref{thm-characterising-semiperfect-mod-fin-alg-with-brick-simples} with \Cref{thm-admissible-presentations}, if $R$ is $\maxideal$-adicaly complete and $\Gamma$ is a basic module-finite elementary $R$-algebra  then $\Gamma$ has the form $RQ/I$ where $Q$ is the Gabriel quiver of $\Gamma$ and $I$ is admissible.  
Together with \Cref{prop-admissible-presentations-extra} this recovers \cite[Proposition, p.2619]{CarSal1987}; see \Cref{cor-raggi-card-salm}. 
 % For any module-finite algebra $\Gamma$ over a complete local noetherian ring with algebraically closed residue field, if $\Gamma$ is basic (as a semiperfect ring) then it is isomorphic to a quotient of a path algebra by an admissible ideal; see \cite[Proposition, p.2619]{CarSal1987}.  

\subsection{Biseriality} 

\emph{Uniserial} modules are those whose submodules form a chain. 
\emph{Biserial} modules are those with a pair of uniserial submodules whose sum is the entire module or the unique maximal submodule, and whose intersection is trivial or the unique simple submodule.   
Following\footnote{The notion from \cite{kirichenko-kostyukevich} generalised the notion from artinian rings introduced by Fuller \cite{Ful1978}. }  Kiri\v{c}enko and Kostyukevich \cite{kirichenko-kostyukevich} a noetherian semiperfect ring is \emph{biserial} if every indecomposable projective module is biserial.  
See \S\ref{subsec-biserial}.

We keep using $R$, $Q$ and $I$ to denote a  local ring, a quiver and an ideal in the path algebra $RQ$. 
A \emph{subpath} of a non-trivial path $p=a_{n}\dots a_{1}$ is a path of the form $a_{j}\dots a_{i}$ with $n\geq j\geq i\geq 1$. 
In case $i=1$ they are \emph{right subpaths}\footnote{The terminology of left and right subpaths was used by Huisgen-Zimmermann and Smal\o{} in  \cite{HuiSma2005}. } and in case $j=n$ they are \emph{left subpaths}. 
The \emph{right arrow} of $p$ is $a_{1}$ and the \emph{left arrow} is $a_{n}$.   
An \emph{inadmissible} path is one that lies in $I$. 
The paths lying outside $I$ are  \emph{admissible}.  
% The algebras we define in this article are module-finite $R$-algebras of the form $RQ/I$ where $I$ is permissible, arrow-radical and arrow-direct. 

% Under mild restrictions on $Q$ and $I$  we show $I$ is monomial when $R$ is a field; see \Cref{lem-special-conditions-and-arrow-radical-over-field-means-monomial}. 
% Let $\Lambda =RQ/I$ where $I$ is $Z$-permissible.  
% \Cref{thm-admissible-presentations} motivated our terminology, and conditions (i)--(iv) stated above. 

% For convenience we collect the observations made so far. 
% \begin{cor}
% \label{cor-admissibility-with-completeness-and-algebraic-closure}
% Let $I\triangleleft RQ$ be arrow-radical, arrow-direct and permissible for a set $Z$ of paths of length at least $2$. 
% If $R$ is $\maxideal$-adically complete and $k$ is algebraically closed then the quotient ring $\Lambda=RQ/I$ is basic and semiperfect.
% \end{cor}

\begin{itemize}
     \item $I$ is said to be \emph{special}\footnote{This terminology is motivated by work of  Skowro{\'n}ski and Waschb{\"u}sch \cite{SkoWas1983}.} if for any arrow $b$ there exists at most $1$ arrow $a$ such that $t(a)=h(b)$ and  $ab$ is admissible, and there exists at most $1$ arrow $c$ such that  $h(c)=t(b)$ and  $bc$ is admissible.
\end{itemize}

% Hence an ideal $I\triangleleft RQ$ is permissible over $Z$ if and only if the set of paths that lie in $I$ is equal to the set of inadmissible paths. 

\begin{thm}
\label{thm-main-string-biserial-multi}
Let  $I$ be arrow-radical, bounded below and special. 
Let $p$ be an admissible non-trivial path with left arrow $b$ and right arrow $a$. 
The following statements hold for the ring $\Lambda=RQ/I$. 
\begin{enumerate}
     \item Any non-trivial left submodule of $\Lambda p$ has the form $\Lambda q$ for an admissible path $q$ with right subpath $p$.  
     \item Any non-trivial right submodule of $p\Lambda $ has the form $q\Lambda$ for an admissible path $q$ with left subpath $p$.  
     \item If $q$ is an admissible path with right arrow $a$ and $\Lambda p=\Lambda q$  then $p=q$. 
     \item If $q$ is an admissible path with left arrow $b$ and $p\Lambda =q\Lambda $  then $p=q$.
     % \item If $a$ is the right arrow of $p$ then we have the inclusion $p\rad(\Lambda)\cap \Lambda a\subseteq \rad(\Lambda)p$.
     % \item If $b$ is the left arrow of $p$ then we have the inclusion $\rad(\Lambda)p\cap b\Lambda\subseteq p\rad(\Lambda)$.
\end{enumerate} 
\end{thm}
Hence the modules of the form $\Lambda p$ and $p\Lambda$ appearing in \Cref{thm-main-string-biserial-multi} are uniserial.

\begin{itemize}
    \item $I$ is said to be   \emph{arrow}-\emph{direct} if (v) and (vi) below hold for any arrow $a$. 
        \begin{enumerate}[label={\upshape(\roman*)}]
        \setcounter{enumi}{4}
        \item $\Lambda a\cap \sum \Lambda b=0$ where the sum runs through the arrows $b\neq a$ with $t(b)=t(a)$. 
        \item $a\Lambda \cap\sum b\Lambda =0$ where the sum runs through the arrows $b\neq a$ with $h(b)=h(a)$.
        \end{enumerate}
\end{itemize} 

% So arrow-direct ideals are arrow-distinct. 
% We now state our definition of a string algebra over a local ring $(R,\maxideal,k)$. 

\begin{defn}
\label{defn-string-pairs-and-algebras}
A module-finite $R$-algebra   $RQ/I$ is a \emph{string algebra over} $R$ if $I$ is arrow-radical, arrow-direct, permissible and  special and if each vertex is the head (respectively, tail) of at most $2$ arrows.  
% following conditions hold. 
% \begin{itemize}
%  \item For any vertex $v$ there exist at most $2$ arrows $a,b$ with head $h(a)=v=h(b)$. 
%  \item For any vertex $v$ there exist at most $2$ arrows $a,b$ with tail $t(a)=v=t(b)$. 
%      \item For any arrow $b$ there exists at most $1$ arrow $a$ with $t(a)=h(b)$ and  $ab\notin I$.
%     \item For any arrow $b$ there exists at most $1$ arrow $c$ with $h(c)=t(b)$ and  $bc\notin I$.
% \end{itemize} 
\end{defn}
If $\Lambda=RQ/I$  is a string algebra over $R$ then $I$ is admissible. 
If also $R$ is complete then $Q$ is the Gabriel quiver of $\Lambda$. 
See \Cref{thm-admissible-presentations}. 
A string algebra over a field (in the sense of \Cref{defn-string-pairs-and-algebras}) is the same thing as a \emph{string algebra}\footnote{Here we assume quivers are finite and hence we only consider the  rings from  \cite{ButRin1987} that are unital.} in the sense of  \cite{ButRin1987}. 
See \Cref{prop:butler-ringel-string-for-a-field}. 
In \S\ref{section-examples} we look at other examples.

\begin{thm}
\label{thm-main-string-biserial}
If $\Lambda$ is a string algebra  over $R$ then $\Lambda$ is biserial as a semiperfect noetherian ring. 
Furthermore the following statements hold for any admissible non-trivial path $p$. 
\begin{enumerate}
     % \item Any left submodule of $\Lambda p$ has the form $\Lambda q$ for some admissible path $q$ with right subpath $p$.  
     % \item Any right submodule of $p\Lambda $ has the form $q\Lambda$ for some admissible path $q$ with left subpath $p$.  
     % \item If $q$ is an admissible path such that either $\Lambda p=\Lambda q$ or $p\Lambda=q\Lambda $ then $p=q$.
     \item If $p$ has left arrow $b$ and right arrow $a$ then $\jacrad p\cap b\Lambda\subseteq p\jacrad$ and $p\jacrad\cap \Lambda a\subseteq \jacrad p$  where  $\jacrad=\rad(\Lambda)$. 
     \item The kernel of the projective cover of $\Lambda p$ is a direct sum of at most $2$ modules each of the form $\Lambda q$ where $q$ is an admissible path of  minimal length such that $qp$ is inadmissible. 
     \item The kernel of the projective cover of $p\Lambda $ is a direct sum of at most $2$ modules each of the form $q\Lambda $ where $q$ is an admissible path of  minimal length  such that $pq$ is inadmissible.
\end{enumerate} 
\end{thm}
Note that Roggenkamp and Wiedemann \cite[Theorem 3]{RogWie1983} related the global dimension of an order with properties of the kernel of the projective cover of its jacobson radical. 

In \S\ref{section-background} and \S\ref{sec-module-finite-algebras} we survey properties of noetherian semiperfect rings,  module-finite algebras over local rings and their quivers. 
In \S\ref{section-definition} we look at admissibility and prove \Cref{thm-characterising-semiperfect-mod-fin-alg-with-brick-simples} and \Cref{thm-admissible-presentations}. 
In \S\ref{sec-stringalgebras-over-local} we look at biseriality and prove \Cref{thm-main-string-biserial-multi} and \Cref{thm-main-string-biserial}. 
In \S\ref{section-examples} we discuss examples of string algebras over local rings that arise as orders  considered previously in the literature \cite{BurDro2006,Dro2001,Fields-Examples-1969,Gab1987,GelPon1968,Rin1975}. 

\section{Noetherian semiperfect basic rings}
\label{section-background}

In  \S\ref{sec-module-finite-algebras} we collect  properties of module-finite algebras over local noetherian ground rings. 
Such algebras are noetherian and semilocal, and semiperfect in case $R$ is $\maxideal$-adically complete. 
Thus in \S\ref{subsection-semiperfect-background} we recall properties of semiperfect rings, and in \S\ref{subsection-noetherian-semiperfect} we focus on semiperfect rings that are also noetherian. 
We use the books by Gubareni, Hazewinkel and Kiri\v{c}enko \cite{Haze2004} and by Lam \cite{Lam1991}. 

\subsection{Background}
\label{subsection-semiperfect-background} 
We begin with a set up for \S\ref{section-background} that we gradually  add restrictions to. 

\begin{setup}
\label{ring-setup}

Fix a (unital, associative) ring $\Lambda$. 
For any integer $n>0$ write  $\boldsymbol{\mathrm{M}}_{n}(\Lambda)$ for the ring of 
 $n$-by-$n$ square matrices $\left(x_{ij}\right)$ with entries $x_{ij}\in \Lambda$. 
Let $\jacrad=\rad(\Lambda)$, the jacobson radical of $\Lambda$. 
% We write $\overline{x}=x+\jacrad $ for any $x\in \Lambda$. 
Let $\lMod{\Lambda}$ be the category of left $\Lambda$-modules and $\rMod{\Lambda}$ the category of right $\Lambda$-modules. 
\end{setup}

\subsubsection{Idempotents} 
\label{subsubsec-idempotents}
% Let $e,f$ be non-zero idempotents in $\Lambda$. 
% % ,  and so $0\neq e=e^{2}\in \Lambda$ and similarly for $f$. 
% Then $e$ is:  \emph{primitive} if the ring $e\Lambda e$ has no non-trivial idempotents;  \emph{local} if $e\Lambda e$ is local; \emph{isomorphic} to $f$, denoted $e\simeq f$, if $e=ab$ and $f=ba$ where $a\in e\Lambda f$ and $b\in f\Lambda e$; and \emph{orthogonal} 
%  to $f$ if $ef=0=fe$. 
%  Recall that there are isomorphisms
%  \[
%  \Hom{\lMod{\Lambda}}(\Lambda f,\Lambda e)\cong f\Lambda e \cong \Hom{\rMod{\Lambda}}(e\Lambda , f\Lambda )
%  \]
% of $f\Lambda f$-$e\Lambda e$-bimodules. 
% Furthermore when $e=f$ these isomorphisms are ring isomorphisms. 
% PROOF: Let $H\Hom{\lMod{\Lambda}}(\Lambda f,\Lambda e)$. 
% Let $\theta\in H$. 
% It is sent to $\theta(f)=\theta(f^{2})=f\theta(f)\in f\Lambda e$. 
% Let $\lambda\in f\Lambda f$ and $\mu\in e\Lambda e$. 
% The action of $\lambda$ on $\theta$ is defined by $\lambda\theta (\alpha)=\theta(\alpha \lambda)$ for all $\alpha\in \Lambda f$. 
% Note that if $\theta'=\lambda\theta$ then for any $\beta\in \Lambda$ we have $\beta(\theta' (\alpha))=\beta\theta(\alpha \lambda)=\theta(\beta \alpha\lambda)=\theta'(\beta\alpha)$ and so $\theta'\in H$. 
% Also, $\theta'(f)=\theta(\lambda)=\lambda\theta(f)$ and so the map is left $f\Lambda f$ linear. 
% Dually and similarly it is $f\Lambda f$-$e\Lambda e$-bilinear and it is well-known to be an isomorphism. 

An idempotent $e\in \Lambda$ is: \emph{primitive} if $e\neq f+g$ for all non-zero orthogonal idempotents $f$ and $g$, or equivalently if $\Lambda e$ is indecomposable; \emph{local}  
% if and only if $\Lambda e$ is strongly indecomposable 
if $\Lambda e$ has unique maximal submodule $\jacrad e$, or equivalently if $\Lambda e/\jacrad e$ is simple; and \emph{isomorphic} to an idempotent $f$ if $\Lambda e\cong \Lambda f$, or equivalently if $e=ab$ and $f=ba$ for some $a\in e\Lambda f$ and some $b\in f\Lambda e$. 
So if $e$ is local then the quotient $\Lambda e\to \Lambda e/\jacrad e$ is a projective cover of a simple. 
% if and only if $ e\Lambda\cong f \Lambda $. 
Dual claims hold for right modules. 
See \cite[Propositions 21.8, 21.9, 21.18, 24.10]{Lam1991}. 
% In case idempotents lift modulo the radical we have that $e$ is primitive in $\Lambda$
%  if and only if $\overline{e}$ is primitive in $\Lambda/\jacrad$; see \cite[Propositions 21.22]{Lam1991}. 

A set $\boldsymbol{e}=\{e_{1},\dots,e_{n}\}$ of idempotents $e_{i}\in \Lambda$ is said to be: \emph{complete} if $1=e_{1}+\dots +e_{n}$; and   \emph{orthogonal} if $e_{i}e_{j}=0=e_{j}e_{i}$  whenever $i\neq j$. 
If such a set $\boldsymbol{e}$ consists of isomorphic idempotents then $e=e_{1}+\dots+e_{n}$ satisfies  $e\Lambda e\cong \boldsymbol{\mathrm{M}}_{n}(e_{i}\Lambda e_{i})$ for any $i$.  
Fix complete and orthogonal sets $\boldsymbol{e}=\{e_{1},\dots,e_{n}\}$ and $\boldsymbol{f}=\{f_{1},\dots,f_{m}\}$ of idempotents. 
% one has $\rad(\Lambda e)=\rad(\Lambda)e$ as modules; one has $\End{\lMod{\Lambda}}(\Lambda e/\rad(\Lambda)e)\cong e\Lambda e/e\rad(\Lambda)e$ as rings; 
If $n=m$ and if $e_{i}$ and $f_{i}$ are isomorphic for each $i$ then there exists a unit $x\in \Lambda$ with $f_{i}=x^{-1}e_{i}x$ for each $i$. 
If $e_{i},f_{i}$ are local for each $i$ then $n=m$ and conjugating by some unit $x$  gives a bijection $\boldsymbol{e}\to \boldsymbol{f}$.  
See  \cite[p. 334, Exercises 15--18]{Lam1991} and also \cite[p. 58]{jacobson-structure-of-rings}.

\subsubsection{Semiperfect rings}
\label{sec-projective-modules-over-semiperfect}

One says that \emph{idempotents lift modulo the radical} provided that for any $e\in \Lambda$ such that $e+\jacrad$ is idempotent (so  $e^{2}-e\in \jacrad $)  then there is an idempotent  $f\in \Lambda$ such that $e-f\in \jacrad $. 
Recall that $\Lambda$ is a \emph{semilocal} ring provided  $\Lambda/\jacrad$ is an artinian ring, and hence a semisimple artinian ring. 
In this case  $\jacrad M$ is the radical of any left  $\Lambda$-module $M$, and dually for right modules; see for example \cite[Proposition 24.4]{Lam1991}. 
A semilocal ring $\Lambda$ is  \emph{semiperfect} provided   idempotents lift modulo the radical. 

Note that any local ring is semiperfect, and any artinian ring is semiperfect. 
Of particular interest later is that any module-finite algebra over a complete local noetherian ground ring is semiperfect; see \S\ref{sec-module-finite-algebras}.  
%\begin{example}\label{example-s-of-semiperfect-rings}
% % The definition of gentle we use refers to a class of such algebras, and so, for our intents and purposes, these algebras serve as the most important examples of semiperfect rings in what follows.
% \end{example}
By a famous result of Bass \cite[Theorem 2.1]{Bas1960}, saying that $\Lambda$ is semiperfect is equivalent to saying that any finitely generated $\Lambda$-module\footnote{Note that we ignore specifying left or right modules here. 
The definition of semiperfect rings we use is symmetric. 
The result of Bass says that such a specification is unnecessary if  defined in terms of covers of finitely generated modules.} $M$ has a projective cover $P\to M$. In this case $P$ must also be finitely generated by the theory of projective covers. 
% , that itself is equivalent to saying that any cyclic $\Lambda$-module has a projective cover. 
% \begin{rem}
% \end{rem}
By a result of M{\"u}ller  \cite[Theorem 1]{Mueller-on-semiperfect-rings} the ring $\Lambda$ is semiperfect if and only if there exists a complete and orthogonal set  $\boldsymbol{e}=\{e_{1},\dots,e_{n}\}$ of local idempotents.

\subsubsection{Basic rings}
\label{basic rings}

% We continue to assume $\Lambda$ is a semiperfect ring with a complete and orthogonal set $\boldsymbol{e}=\{e_{1},\dots,e_{n}\}$ of local idempotents. 

Let $\Lambda$ be semiperfect with a complete and orthogonal set $\boldsymbol{e}=\{e_{1},\dots,e_{n}\}$ of local idempotents. 
Partition $ \boldsymbol{e}=\boldsymbol{c}(1)\sqcup \dots\sqcup \boldsymbol{c}(m)$ into isomorphism classes $\boldsymbol{c}(j)$. 
Hence if $h\neq i$, $e_{h}\in \boldsymbol{c}(l)$ and $e_{i}\in\boldsymbol{c}(j)$ one has  $l\neq j$ if and only if  $\Lambda e_{h}\ncong \Lambda e_{i}$ if and only if   $e_{h}\Lambda \ncong e_{i}\Lambda$. 
 
For each $j$  the sum of the elements of $\boldsymbol{c}(j)$ defines a new idempotent $f_{j}$.  
 Running over $j$ defines a complete and orthogonal set $\boldsymbol{f}=\{f_{1},\dots,f_{m}\}$ of idempotents. 
% The idempotents in $\boldsymbol{f}$ are no longer primitive, but distinct idempotents in $\boldsymbol{f}$  cannot have isomorphic direct summands. 
This gives  \emph{Peirce decompositions}
% \[
% \begin{array}{cc}
% \Lambda=\begin{pmatrix}
% f_{1}\Lambda f_{1} & f_{1}\Lambda f_{2} & \cdots & f_{1}\Lambda f_{m} \\
% f_{2}\Lambda f_{1} & f_{2}\Lambda f_{2} & \cdots & f_{2}\Lambda f_{m}
% \\
% \vdots & \vdots & \ddots & \vdots
% \\
% f_{m}\Lambda f_{1} & f_{m}\Lambda f_{2} & \cdots & f_{m}\Lambda f_{m}
% \end{pmatrix}, & 
% \jacrad  =\begin{pmatrix}
% f_{1}\jacrad  f_{1} & f_{1}\Lambda f_{2} & \cdots & f_{1}\Lambda f_{m} \\
% f_{2}\Lambda f_{1} & f_{2}\jacrad  f_{2} & \cdots & f_{2}\Lambda f_{m}
% \\
% \vdots & \vdots & \ddots & \vdots
% \\
% f_{m}\Lambda f_{1} & f_{m}\Lambda f_{2} & \cdots & f_{m}\jacrad f_{m}
% \end{pmatrix}.
% \end{array}
% \]
% (depicted for $m>1$). 
\[
\begin{array}{cc}
\Lambda=\begin{pmatrix}
f_{1}\Lambda f_{1}  & \cdots & f_{1}\Lambda f_{m} 
\\
\vdots & \ddots & \vdots
\\
f_{m}\Lambda f_{1} & \cdots & f_{m}\Lambda f_{m}
\end{pmatrix}, & 
\jacrad  =\begin{pmatrix}
f_{1}\jacrad  f_{1} & \cdots & f_{1}\Lambda f_{m} \\
\vdots & \ddots & \vdots
\\
f_{m}\Lambda f_{1} & \cdots & f_{m}\jacrad f_{m}
\end{pmatrix}.
\end{array}
\]
See \cite[Proposition 11.1.1]{Haze2004}. 
% For fixed $j$ any $i$ with $e_{i}\in \boldsymbol{c}(j)$ gives $\Lambda f_{j}\cong \Lambda e_{i}^{n(j)}$ where $n(j)=\vert \boldsymbol{c}(j)\vert $. 
% Hence
% \[
% \begin{array}{cc}
% f_{j}\Lambda f_{j}\cong \End{\lMod{\Lambda}}\left(\Lambda e_{i}^{n(j)}\right)\cong \boldsymbol{\mathrm{M}}_{n(j)}(e_{i}\Lambda e_{i}),
% &
% \rad(\boldsymbol{\mathrm{M}}_{n(j)}(e_{i}\Lambda e_{i}))=\boldsymbol{\mathrm{M}}_{n(j)}(e_{i}\jacrad e_{i}).
% \end{array}
% \]
% One then obtains a Wedderburn--Artin decomposition of the semisimple artinian ring  $\Lambda/\jacrad \cong \prod_{j=1}^{m} \boldsymbol{\mathrm{M}}_{n(j)}(D_{j})$ where for each $j$ we have $D_{j}=e_{i}\Lambda e_{i}/e_{i}\jacrad e_{i}$. 
The semiperfect ring $\Lambda$ is \emph{basic} if $\Lambda/\jacrad $ is isomorphic to a product of division rings, or equivalently, if  
% $n(j)=1$ for each $j$, or that 
no pair of distinct elements of $\boldsymbol{e}$  are isomorphic. 
In this case $n=m$ and there is a bijection  $\boldsymbol{e}\to \boldsymbol{f}$ defined by conjugating with some unit; see \cite[Proposition 25.10]{Lam1991}. 

For each $j=1,\dots,m$ make a choice of $i(j)=1,\dots,n$ such that $e_{i(j)}\in \boldsymbol{c}(j)$, and then let $e=\sum_{j=1}^{m}e_{i(j)}$. 
 Then $ e\Lambda e\cong\End{\lMod{\Lambda}}(P)$ where $P=\bigoplus_{j=1}^{m}\Lambda e_{i(j)}$. Note that $e\Lambda e$ is: Morita equivalent to $\Lambda$, see for example \cite[\S3]{Morita-duality-modules-rings-with-min-condtions} or \cite[Chapters II, III]{Bass-algebraic-K-theory}; hence   semiperfect by \cite[Theorem 2.1]{Bas1960}; and basic by construction.

\subsubsection{Projectives}
\label{sec-proj-mods}
Let $\Lambda$ be semiperfect with a complete and orthogonal set $\{e_{1},\dots,e_{n}\}$ of local idempotents. 

By a \emph{principal local} left module we mean one of the form $\Lambda e_{i}$ with $i=1,\dots,n$. 
Hence principal local left modules have local endomorphism rings. 
Likewise \emph{principal local} right modules have the form $e_{i}\Lambda$. 
Either for left modules or for right modules, in the notation of \S\ref{basic rings} for each $j=1,\dots,m$ let $P_{j}$ be the principal local module corresponding to the chosen idempotent  $e_{i(j)}\in \boldsymbol{c}(j)$, and let $S_{j}=P_{j}/\rad(P_{j})$ be the simple quotient.  
Hence the quotient map $P_{j}\to S_{j}$ is the projective cover. 

Any projective indecomposable is isomorphic to some $P_{j}$ where $j$ is unique since the idempotents $f_{1},\dots,f_{m}$ are pair-wise non-isomorphic. 
Furthermore $P_{j}\mapsto S_{j}$  defines a bijection between isoclasses of projective indecomposables and simples. 
Moreover, 
% any finitely generated projective module is isomorphic to a finite direct sum of copies of $P_{1},\dots,P_{m}$ occurring with unique multiplicity. 
% See for example \cite[Theorem 25.3]{Lam1991}. 
% In fact, 
say following a result of P\v{r}\'{\i}hoda \cite[Corollary 2.4(iii)]{Pri2007}, every projective $\Lambda$-module is isomorphic to a (possibly infinite) direct sum of principal local modules. 
By the Krull--Remak--Schmidt--Azumaya decomposition theorem \cite[Theorem 1]{Azumaya-KRS}, since principal local modules are indecomposable with local endomorphism rings, any such direct sum decomposition of a projective module is unique up to an isoclass preserving bijection between principal local modules.

\subsubsection{Covers}
\label{sec-covers}
\Cref{lem-proj-covers-ext-between-simples} is a well-known result that we use in \S\ref{quiver-noetherian-semiperfect}. 
Let $M$, $P$ and $S$ be left $\Lambda$-modules where $S$ is simple and where $\mathsf{p}\colon P\to M$ is a projective cover. 
    Then any $\mathsf{h}\in \Hom{\lMod{\Lambda}}(P,S)$ has the form $\mathsf{h}=\mathsf{m}\mathsf{p}$ for some $\mathsf{m}\in \Hom{\lMod{\Lambda}}(M,S)$. 
    We  use this idea repeatedly in what follows.  

% \begin{lem}
% \label{lem-tech-1}  
% \end{lem}
% \begin{proof}
% We can assume $\mathsf{h}\neq 0$. 
% Since $S$ is simple the map $\mathsf{h}$ must be surjective and the kernel $K'$ of $\mathsf{h}$ is a maximal submodule of $P$. 
% Let $\mathsf{k}\colon K\to P$ be the kernel of $\mathsf{p}$. 
% Since $K$ is a superfluous submodule of $P$ and $\mathsf{h}\neq 0$ we must have that $K+K'\neq P$. 
% By maximality this gives $K\subseteq K'$ and so $\mathsf{h}\mathsf{k}=0$. 
% The proof follows by the exactness of $0\to K\to P\to M\to 0$ and the left-exactness of $\Hom{\lMod{\Lambda}}(-,S)$. 
% \end{proof}

\begin{lem}
    \label{lem-proj-covers-ext-between-simples}
    \emph{(c.f \cite[\S2]{he-ye-yoneda-ext}). } 
    Let $\Lambda $ be a semiperfect ring with a complete and orthogonal set $\boldsymbol{e}$ of local idempotents. 
    Let $e,f\in\boldsymbol{e}$, $S=\Lambda e/\jacrad e$,   $T=\Lambda f/\jacrad f$ and  $U=\jacrad f$. 
    If $U$ is finitely generated and $P\to U$ is a projective cover with kernel $K$ then there are left $\End{\lMod{\Lambda}}(S)$-module isomorphisms 
    \[
\Hom{\lmod{\Lambda}}(P,S) \cong \Hom{\lmod{\Lambda}}(U,S)\cong \Ext{1}{\lmod{\Lambda}}(T,S),
\,
\Hom{\lmod{\Lambda}}(K,S)\cong \Ext{1}{\lmod{\Lambda}}(U,S) \cong \Ext{2}{\lmod{\Lambda}}(T,S).
    \]
\end{lem}
\begin{proof}
Let $\mathcal{A}=\lmod{\Lambda}$, the category  of finitely generated left $\Lambda$-modules.  
Let $\mathcal{B}=\lMod{\End{\mathcal{A}}(S)}$, the category  of  left $\End{\mathcal{A}}(S)$-modules and consider\footnote{Any functor in this paper is assumed to be covariant. 
To consider contravariant functors we appropriately take opposite categories. } the hom-functor\footnote{Here we using that for any preadditive categories $\mathcal{C},\mathcal{D}$ and any biadditive functor $\mathscr{E}$ from $\mathcal{C}^{\mathrm{op}}\times \mathcal{D}$ to the category of abelian groups, each object $X$ in $\mathcal{D}$ defines an additive functor $\mathscr{E}(-,X)$ from $ \mathcal{C}^{\mathrm{op}}$ to the category of left $\End{\mathcal{D}}(X)$-modules. }    $\Hom{\mathcal{A}}(-,S)\colon \mathcal{A}^{\mathrm{op}}\to \mathcal{B}$. 
%
% given by $\mathscr{F}(X)=\Hom{\mathcal{A}}(X,T)$ for any object $X$ in $\mathcal{A}$  and $\left(\mathscr{F}(\mathsf{z})\right)(\mathsf{h})=\mathsf{h}\mathsf{z}$ for any morphisms $\mathsf{h}\colon Y\to T$ and $\mathsf{z}\colon X\to Y$ in $\mathcal{A}$. 
% One can  check $\mathscr{F}$ defines a functor of the described form. 

Since $\Lambda$ is noetherian $\mathcal{A}$ is abelian.  
Since $\Lambda$ is semiperfect $\mathcal{A}$ has enough projectives (objects have projective covers). 
Consider the functors $\Ext{d}{\mathcal{A}}(-,S)\colon \mathcal{A}^{\mathrm{op}}\to \mathcal{B}$ ($d\geq0$); see for example \cite[Theorem 2.4.5]{Wie1986}. 

Let $\mu\colon 0\to L\to M\to N\to 0$ be a conflation in $\mathcal{A}$. 
The long exact sequence of $\Hom{\mathcal{A}}(-,S)$ for $\mu$ is 
\[
\begin{tikzpicture}[descr/.style={fill=white,inner sep=2pt}]
        \matrix (m) [
            matrix of math nodes,
            row sep=1.5em,
            column sep=3.5em,
            text height=1.5ex, text depth=0.25ex
        ]
        { 0 & \Hom{\mathcal{A}}(N,S) & \Hom{\mathcal{A}}(M,S) & \Hom{\mathcal{A}}(L,S) &\\
            &  \Ext{1}{\mathcal{A}}(N,S) &  \Ext{1}{\mathcal{A}}(M,S) &  \Ext{1}{\mathcal{A}}(L,S) & \\
            & \Ext{2}{\mathcal{A}}(N,S) & \Ext{2}{\mathcal{A}}(M,S) & \Ext{2}{\mathcal{A}}(L,S) & \cdots \\
        };

        \path[overlay,->, font=\scriptsize,>=latex]
        (m-1-1) edge (m-1-2)
        (m-1-2) edge["$\alpha_{0}^{\mu}$"]  (m-1-3)
        (m-1-3) edge["$\beta_{0}^{\mu}$"] (m-1-4)
        (m-1-4) edge[out=195,in=15] node[descr,yshift=0.3ex] {$\gamma_{0}^{\mu}$} (m-2-2)
        (m-2-2) edge["$\alpha_{1}^{\mu}$"] (m-2-3)
        (m-2-3) edge["$\beta_{1}^{\mu}$"] (m-2-4)
        (m-2-4) edge[out=195,in=15] node[descr,yshift=0.3ex] {$\gamma_{1}^{\mu}$} (m-3-2)
        (m-3-2) edge["$\alpha_{2}^{\mu}$"] (m-3-3)
        (m-3-3) edge["$\beta_{2}^{\mu}$"] (m-3-4)
        (m-3-4) edge (m-3-5);
\end{tikzpicture}
\]
Now assume $M\to N$ is a projective cover. 
Since $M$ is projective  $\gamma_{0}^{\mu}$ is surjective and  $\gamma_{d}^{\mu}$ is an isomorphism for all $d>0$. 
The discussion at the start of \S\ref{sec-covers} gives   $\beta_{0}^{\mu}=0$ and so $\alpha_{0}^{\mu}$ and $\gamma_{0}^{\mu}$ are isomorphisms. 
Taking $\mu\colon 0\to K\to P\to U\to 0$ gives  $\Hom{\mathcal{A}}(P,S)\cong  \Hom{\mathcal{A}}(U,S)$ and $\Hom{\mathcal{A}}(K,S)\cong  \Ext{1}{\mathcal{A}}(U,S)$. 
Taking  $\mu\colon 0\to U\to \Lambda f\to T\to 0$ gives $\Hom{\mathcal{A}}(U,S) \cong \Ext{1}{\mathcal{A}}(T,S)$ and $\Ext{1}{\mathcal{A}}(U,S) \cong \Ext{2}{\mathcal{A}}(T,S)$. 
\end{proof}

% By the result  \cite[Corollary 2.4(iii)]{Pri2007} discussed in \S\ref{sec-projective-modules-over-semiperfect},  for any projective (left or right) $\Lambda$-module  $P$ we have that $P/\rad(P)$ is a direct sum of 

\subsection{Noetherian semiperfect basic rings}
\label{subsection-noetherian-semiperfect}

We begin to restrict \Cref{ring-setup}. 

\begin{setup}
\label{setup-noetherian-semiperfect} 
% We restrict \Cref{semiperfect-ring-setup}. 
In \S\ref{subsection-noetherian-semiperfect} assume $\Lambda$ is a noetherian semiperfect basic ring with a complete and orthogonal set $\boldsymbol{e}=\{e_{1},\dots,e_{n}\}$ of pair-wise non-isomorphic local idempotents. 
Recall $\jacrad =\rad(\Lambda)$. 
\end{setup}

Since the ring $\Gamma=\Lambda/\jacrad^{2}$ satisfies $\rad(\Gamma)=\jacrad /\jacrad^{2}$ and hence $\rad(\Gamma)^{2}=0$, it follows that $\Gamma$ is \emph{semiprimary}. 
Given that $\Lambda$ is noetherian $\Gamma$ is also noetherian. 
Hence $\Gamma$ is artinian by the Hopkins-Levitzki theorem. 
% We collect properties of $\Lambda$ following the books: by Hazewinkel, Gubareni and Kiri\v{c}enko \cite{Haze2004}; and by Lam \cite{Lam1991}. 
% \begin{rem}\label{facts-about-noetherian-semiperfect-rings}
% \begin{enumerate}
%     \item   
%     \item   \cite[Theorem 3.7.1 (Hopkins-Levitzki)]{Haze2004}.
    % \item By \cite[Theorem 10.3.8, (M{\"u}ller)]{Haze2004}, saying that $\Lambda$ is semiperfect is equivalent to saying that there exist pairwise-orthogonal idempotents $e_{1},\dots,e_{n}$ in $\Lambda$ such that $1=e_{1}+\dots +e_{n}$ and such that each $e_{i}\Lambda e_{i}$ is a local ring.
    % \item By \cite[Theorem 10.4.8 (Bass)]{Haze2004}, saying that $\Lambda$ is semiperfect is equivalent to saying that any finitely generated $\Lambda$-module has a projective cover. 
% \end{enumerate}
% \end{rem}

\subsubsection{Quivers of noetherian semiperfect basic rings}
\label{quiver-noetherian-semiperfect}

We saw above that $\Lambda/\jacrad^{2}$ is artinian since $\Lambda$ is noetherian, and hence we can and do follow work of Kiri\v{c}enko \cite{kirichenko-gen-uniserial} and recall the \emph{left quiver} $Q_{l}(\Lambda)$ and the \emph{right  quiver} $Q_{r}(\Lambda)$; see also \cite{Haze2004,MR957886}. 
In this article a quiver $Q=(Q_{0},Q_{1},h,t)$ consists of vertices $Q_{0}$, arrows $Q_{1}$ and head and tail functions $Q_{1}\to Q_{0}$ denoted $h$ and $t$ respectively.

Since $\Lambda$ is noetherian, for each $i$ the module $\rad(\Lambda e_{i})=\jacrad e_{i}$ is finitely generated, and 
% by part (5) of Remark \Cref{facts-about-noetherian-semiperfect-rings}
we can take a projective cover $P(\jacrad e_{i})\to \jacrad  e_{i}$ by \cite[Theorem 2.1]{Bas1960}. 
By the discussion at the end of \S\ref{sec-projective-modules-over-semiperfect} we have 
\[
P(\jacrad e_{i})\cong \Lambda(a_{l}(i,1),\dots,a_{l}(i,n))= (\Lambda e_{1})^{(a_{l}(i,1))} \oplus \dots \oplus   (\Lambda e_{n})^{(a_{l}(i,n))}
\]
for integers $a_{l}(i,j)\geq 0$. 
Since we are assuming $\Lambda$ is basic, by  the discussion at the end of \S\ref{basic rings} the integers $a_{l}(i,j)$ are uniquely determined. 
% by part (3) of Remark \Cref{facts-about-noetherian-semiperfect-rings-2} and by \cite[Theorem 10.4.11 (Krull-Schmidt)]{Haze2004}. 
%By part (3) of Remark \Cref{facts-about-noetherian-semiperfect-rings} and by \cite[Definition, p. 263]{Haze2004}, t
The left quiver $Q_{l}(\Lambda)$ is defined by taking $\{v_{1},\dots,v_{n}\}$ as the set of vertices, and declaring that there are $a_{l}(i,j)$ distinct arrows with tail $v_{i}$ and head $v_{j}$. 
Said briefly, $a_{l}(i,j)$ is the unique multiplicity of $\Lambda e_{j}$ as a summand of the cover of $\jacrad e_{i}$. 

For the right quiver $Q_{r}(\Lambda)$ one takes vertices $\{u_{1},\dots,u_{n}\}$ and declares that there are $a_{r}(i,j)$ distinct arrows with tail $u_{i}$ and head $u_{j}$ provided  $e_{j}\Lambda$ occurs with multiplicity $a_{r}(i,j)$ in the projective cover of $e_{i}\jacrad$. 
We are interested in when the left and right quivers are opposites. 
Later, in \S\ref{section-definition}, we consider path algebras. 

\begin{rem}
\label{arrow-direction-convention-remark}
    Fix arrows $x$ and $y$ in a quiver $Q$ where the head of $y$ coincides with the tail of $x$. 
% Concatenation of these arrows over $v$ defines a length-$2$ path in $Q$ that has head equal to the head of $x$ and tail equal to the head of $y$. 
% One must choose a convention by denoting this path either by $xy$ or by $yx$. 
We write  $xy$ for the composition of $x$ and $y$. 
This means representations\footnote{Representations of $Q$ over a commutative ring $R$ are given by diagrams of shape $Q$ in $\lMod{R}$. } of $Q$ over a ring $R$ correspond to left modules over the path algebra $RQ$. 
% In \S\ref{sec-ext-sym-simples-matlis} we use Matlis duality to show the left and right quivers are opposite to each other . 
If $R$ is a field and $Q$ is finite and acyclic then $Q$ is the left quiver of $RQ$. 
\end{rem}

\begin{lem}
    \label{lem-counting-arrows-in-quiver}
    \emph{(c.f \cite[\S III.1, Proposition 1.15]{AusReiSma1995}). }
The following statements hold for any $i,j=1,\dots,n$. 
\begin{enumerate}
    \item $a_{l}(i,j)$ is equal to the rank of $\Ext{1}{\lMod{\Lambda}}(\Lambda e_{i}/\jacrad e_{i},\Lambda e_{j}/\jacrad e_{j})$ as a left $\End{\lMod{\Lambda}}(\Lambda e_{j}/\jacrad e_{j})$-module. 
    \item $a_{r}(i,j)$ is equal to the rank of $\Ext{1}{\rMod{\Lambda}}(e_{i}\Lambda /e_{i}\jacrad ,e_{j}\Lambda /e_{j}\jacrad )$ as a left $\End{\rMod{\Lambda}}(e_{j}\Lambda /e_{j}\jacrad )$-module. 
\end{enumerate}
\end{lem}
\begin{proof}
Consider that $\Lambda$ is noetherian and basic. 
Let $S_{h}=\Lambda e_{h}/\jacrad e_{h}$ for all $h=1,\dots,n$. 
Hence each $\jacrad e_{i}$ is finitely generated. 
By \Cref{lem-proj-covers-ext-between-simples} we have $\Hom{\lMod{\Lambda}}(S_{h},S_{j})\cong \Hom{\lMod{\Lambda}}(\Lambda e_{h},S_{j})$ for any $h$. 

By Schur's lemma $\Hom{\lMod{\Lambda}}(S_{h},S_{j})$ is non-trivial if and only if $h=i$. 
By \Cref{lem-proj-covers-ext-between-simples} we have 
\[
\begin{array}{c}
\Ext{1}{\lMod{\Lambda}}(S_{i},S_{j})\cong \Hom{\lMod{\Lambda}}(P(\jacrad e_{i}),S_{j})\cong \bigoplus_{h=1}^{n}\left(\Hom{\lMod{\Lambda}}(\Lambda e_{h},S_{j})\right)^{a_{l}(i,h)}\cong \left(\End{\lMod{\Lambda}}(S_{j})\right)^{a_{l}(i,j)}
\end{array}
\]
as $\End{\lMod{\Lambda}}(S_{j})$-modules. 
Hence (1) holds, and (2) is dual. 
\end{proof}

% Dually we have the following result about the right quiver of a noetherian semiperfect ring.

% \begin{lem}
%     \label{lem-counting-arrows-in-quiver-dual}
% For any $i,j=1,\dots,n$ .   
% \end{lem}

In \S\ref{sec-quiv-complete-case} we see that the left and right quivers of basic module-finite algebras over complete local noetherian rings are opposites. 
Instead of a quiver one may alternatively define a \emph{weighted graph} using the values $a_{l}(i,j)$ and $a_{r}(i,j)$. 
This construction appeared in work on artin algebras going back to Auslander, Reiten and Smal{\o} \cite{AusReiSma1995}. More recently, it appeard in work of G\'{e}linas \cite[\S 3]{Gelinas-fintistic-dimension}.

% Hence $a_{l}(i,j)$ is determined by the rank of $\Ext{\lMod{\Lambda}}{1}(S_{i},S_{j})$. 

% The right quiver $Q_{r}(\Lambda)$ with vertices $w_{1},\dots, w_{n}$ is defined dually, taking a cover
% \[
% ( e_{1}\Lambda)^{r(i,1)}\oplus \dots \oplus ( e_{n}\Lambda)^{r(i,n)}\to e_{i} \jacrad 
% \] of right modules, 
% and defining $r(i,j)$ distinct arrows with tail $w_{i}$ and head $w_{j}$.  

\subsubsection{Biserial noetherian semiperfect rings}
\label{subsec-biserial}

Fuller \cite{Ful1978} introduced the notion of an artinian biserial ring. 
We consider the adaptation of this notion according to Kiri\v{c}enko and Kostyukevich \cite{kirichenko-kostyukevich}. 

Recall that a $\Lambda$-module $M$ is said to be \emph{uniserial} provided $L\subseteq N$ or $N\subseteq L$ for any pair $L,N$ of submodules of $M$. 
Hence $0$ is uniserial, as well as any simple module. 
More generally $M$ is said to be \emph{biserial} provided it contains uniserial modules $L,N$ such that the following two conditions hold. 
\begin{itemize}
    \item Either $L+N=M$, or $L+N=\rad(M)$ is the unique maximal submodule of $M$. 
    \item Either $L\cap N=0$, or $L\cap N=\soc(M)$ is the unique simple submodule of $M$. 
\end{itemize}
We say that $\Lambda$ is \emph{biserial} provided all of its principal local modules are biserial. 
In \Cref{thm-main-string-biserial} we prove that if $\Lambda$ is a string algebra over a local ring then $\Lambda$ is biserial. 
In particular, we prove that any principal local module $P$ contains uniserial modules $L,N$ such that $L\oplus N=\rad(P)$.

% see also work of Kiri\v{c}enko and Yaremenko \cite{kirichenko-yaremenko-noetherian-biserial-rings}

\begin{prop}
\label{noetherian-semiperfect-lemma}
% Let $\boldsymbol{e}=\{e_{1},\dots,e_{n}\}$ be a complete orthogonal set of local idempotents for the semiperfect noetherian ring $\Lambda$. 
If $0\neq L$ is a uniserial left submodule of $\Lambda e_{i} $ for some $i$, then there exists some $j$ and some $a\in e_{j}\Lambda e_{i}$ such that $L=\Lambda a$. 
In this case  the map $\Lambda e_{j} \to L$ given by $\lambda \mapsto \lambda a$ defines a projective cover. 
\end{prop}
\begin{proof}
Since $L$ is a submodule of a finitely generated $\Lambda$-module, $L$ is finitely generated (since $\Lambda$ is noetherian). 
Since finitely generated uniserial modules are always cyclic, we have $L=\Lambda a$ for some $a\in L\subseteq   \Lambda e_{i}$.  
Since $\Lambda$ is semiperfect, and since non-zero uniserial modules are local, by the dual of \cite[Lemma 1.5]{Mul1991} we have $a=e_{j}a$ for some $j$. 
Hence right multiplication defines a surjective map $\Lambda e_{j}\to L$, necessarily a projective cover  since $\Lambda e_{j}$ is a projective indecomposable module. 
\end{proof}

\begin{lem}
    \label{quivers-of-biserials}\emph{(c.f \cite[Theorem 1]{MR957886})}
    Let $Q$ be the left (respectively, right) quiver for $\Lambda$. 
    If $\Lambda$ is biserial then any vertex in $Q$ is the tail  of at most $2$ arrows. 
\end{lem}

\begin{proof}
Suppose $Q$ is the left quiver of $\Lambda$. 
Recall $\boldsymbol{e}=\{e_{1},\dots,e_{n}\}$ is a complete and orthogonal set of pair-wise non-isomorphic local idempotents. 
Fix $i=1,\dots,n$. 
We consider the most complicated situation, where $\jacrad e=L+N$ for non-zero uniserial modules $L,N$. 
By \Cref{noetherian-semiperfect-lemma} one can choose $h,j=1,\dots,n$ such that there are projective covers of the form $\Lambda e_{j}\to L$ and $\Lambda e_{h}\to N$. 
Hence taking direct sums defines a projective cover $\Lambda e_{j}\oplus \Lambda e_{h}\to L\oplus N$.  
Considering the 
% canonical exact sequence $0\to L\cap N\to L\oplus N\to L+N\to 0$ defines an 
epimorphism $ L\oplus N\to \jacrad e$  and 
% . 
% By the theory of projective covers this means 
that $P\to \jacrad e$ is a projective cover there exists  a split epimorphism $\Lambda e_{j}\oplus \Lambda e_{h}\to P$. 
Hence if $P$ is non-trivial it is isomorphic to $\Lambda e_{j}$, $\Lambda e_{h}$ or their direct sum. 
Thus there cannot be $3$ distinct arrows in $Q$ with head $i$ since $\Lambda e_{j}$ and $\Lambda e_{h}$ are indecomposable; see \S\ref{sec-proj-mods}. 
The respective claims follow by symmetry. 
\end{proof}

% \begin{cor}
% Let $i=1,\dots,n$ and suppose that  $\jacrad e_{i}=\bigoplus_{\ell=1}^{d}U_{\ell}$ for non-zero uniserial submodules   
%  $U_{1},\dots,U_{d}$.  
%  Then there are $d$ arrows in $Q_{l}(\Lambda)$ with head $v_{i}$. 
% \end{cor}

% RELATED AND RECENT WORK on representation theoretic methods in the context of noetherian semiperfect rings, see https://arxiv.org/pdf/2004.04828.pdf

\section{Module-finite algebras over noetherian local rings}
\label{sec-module-finite-algebras} 
In \S\ref{sec-module-finite-algebras}  we forget the assumptions from \S\ref{section-background}. 
We provide well-known remarks about module-finite algebras over local noetherian ground rings. 
In \S\ref{sec-mod-finite-over-complete} we note more that can be said when $R$ is complete in its $\maxideal$-adic topology. 
In the meantime, in \S\ref{subsec-krull-intersect} we recall what holds even without this completeness assumption.   
% For consistency we mostly refer to the book by Lam \cite{Lam1991}. 

\begin{setup}
\label{setup-mod-finite-over-R}
In \S\ref{sec-module-finite-algebras} we let $(R,\maxideal,k)$ be a (commutative) noetherian local ring and we let $\Lambda$ be an $R$-algebra. 
% So there is a ring homomorphism $\mathsf{a}\colon R\to \Lambda$ such that the image $\im(\mathsf{a})$ of $\mathsf{a}$ is central, meaning $\im(\mathsf{a})\subseteq \mathcal{Z}(\Lambda)$ where $\mathcal{Z}(\Lambda)$ is the centre of $\Lambda$. 
Assume additionally that  $\Lambda$ is \emph{module-finite over} $R$, meaning that $\Lambda$ finitely generated as an $R$-module. 
\end{setup}

\subsection{Krull-intersection and Matlis-duality}
\label{subsec-krull-intersect} 
In \S\ref{subsec-krull-intersect} we note properties of module-finite $R$-algebras.

\subsubsection{Radical topologies and Krull-intersection}
\label{subsec-rad-topol-and-krull-intersect}

Since $R$ is a noetherian ring, the ring $\Lambda$ is left and right noetherian, and hence the centre of $\Lambda$   is a finite ring extension of $R$. 
Since $\Lambda$ is module-finite over $R$ and since $R$ is (semi)local the ring $\Lambda$ is semilocal. 
In particular, by \cite[Proposition 24.4]{Lam1991} we have $\rad(M)=\jacrad M$ 
% (respectively, $\rad(M)=M\jacrad$) 
for any left 
% (respectively, right)
$\Lambda$-module $M$. 
Furthermore by \cite[Proposition 20.6]{Lam1991} we have
\[
\begin{array}{cc}
\Lambda\maxideal\subseteq \jacrad,
&
\jacrad^{n}\subseteq \Lambda\maxideal
\end{array}
\]
for some integer $n>0$, where $\Lambda\maxideal$ is the two-sided ideal in $\Lambda$ generated by the image of $\maxideal$ under the algebra map  $\mathsf{a}\colon R\to \Lambda$.  
Since $\maxideal$ is maximal this means we have  $\ker(\mathsf{b}\mathsf{a})=\maxideal$ where $\mathsf{b}\colon \Lambda\to \Lambda/\jacrad$ is the canonical projection. 
Hence there is a ring homomorphism $\mathsf{c}\colon k\to \Lambda/\jacrad $ that is injective since $k$ is a field, and so $\mathsf{c}$ equips  $\Lambda /\jacrad $ with a $k$-algebra structure. 
Moreover $\Lambda$ is module-finite over $R$ and so $\dim_{k}\left(\Lambda/\jacrad \right)<\infty$. 
% is finite-dimensional as a vector space over $k$. 

Given that we are assuming $\Lambda$ is module-finite over $R$, the conclusion of \cite[Theorem 1]{Schelter-intersection-1976} holds\footnote{This can be seen by observing that $\Lambda$ is a \emph{polynomial identity ring} (see the first sentence in the introduction of  \cite{Schelter-intersection-1976}) since it must be module-finite as an algebra over its centre of $\Lambda$  given that it is module-finite as an $R$-algebra.},
and so we can recall a \emph{Krull intersection} property. 
Here we follow a general result by Schelter \cite{Schelter-intersection-1976}. 

Let $X$ be a (two-sided) ideal in $\Lambda$ such that for any maximal ideal $Y$ in $\Lambda$ we have that $X\cap \mathcal{Z}(\Lambda)\subseteq Y$ implies $X\subseteq Y$. 
By \cite[Theorem 2]{Schelter-intersection-1976} we have $XN=N$ for any finitely generated $\Lambda$-module $M$ where $N=\bigcap_{n>0}X^{n}M$. 
Now specify  $X\jacrad =\mathrm{rad}(\Lambda)$. 
In this case  $N=\bigcap_{n>0}\jacrad^{n}M$ and since $\Lambda$ is noetherian the submodule $N$ of $M$ is (also) finitely generated: and hence  $\jacrad N=N$ gives $N=0$ by Nakayamas lemma. 
In this article we refer to the equality $\bigcap_{n>0}\jacrad^{n}M=0$ as the \emph{Krull-intersection property} for $M$.

\subsubsection{Completeness and Matlis duality}  
\label{sec-completeness-closure}

We recall well-known work of Matlis \cite{Matlis-Modules-with-descending-chain-condition}. 

Since $R$ is noetherian an $R$-module $V$ is noetherian if and only if it there is an exact sequence of $R$-modules of the form $R^{n}\to R^{m}\to V\to 0$. 
Let $\env$ be the injective envelope of $k$ in $\lMod{R}$. 
By \cite[Theorem 1]{Matlis-Modules-with-descending-chain-condition}  $V$ is artinian if and only if there is an exact sequence $0\to V\to \env^{n}\to \env^{m}$. 
This means $\env$ is artinian, equal to the union of its  $\maxideal^{n}$-torsion submodules, and so $\End{\lMod{R}}(\env)\cong R_{\maxideal}$,  the $\maxideal$-adic completion of $R$. 

Let $\homdual{R}{}$ be the $R$-linear  functor $\Hom{\lMod{R}}(-,\env)\colon \lMod{R}\to(\lMod{R})^{\mathrm{op}}$ that is exact since $\env$ is injective. 
Note that $\homdual{R}{}$ sends noetherian modules to artinian modules. 
For $V$  in $\lMod{R}$ consider the morphism $\Delta_{V}\colon V\to \homdual{R}{\mathrm{op}}(\homdual{R}{}(V))$ that maps $v\in V$ to  $\mathsf{f}\mapsto \mathsf{f}(v)$ for each $\mathsf{f}\in\homdual{R}{}(V)$. 
This gives a natural transformation $\Delta\colon \boldsymbol{1}_{\lMod{R}}\Rightarrow \homdual{R}{\mathrm{op}}\homdual{R}{}$  
% Since $R$ is local any  $v\neq 0$ gives a map $Rv\to k$ that lifts to some $\mathsf{f}\in\homdual{R}{}(V)$ by the injectivity of $\env$. 
where each of the maps $\Delta_{V}$ are injective.

For any $M$ in $\lMod{\Lambda}$ the image $\homdual{R}{}(M)=\Hom{\lMod{R}}(M,\env)$ of the underlying $R$-module has the inherited structure of a right $\Lambda$-module, and hence we have a functor $\homdual{\Lambda}{}\colon \lMod{\Lambda}\to (\rMod{\Lambda})^{\mathrm{op}}$ given by composing  $\homdual{R}{}$ with the  appropriate forgetful functors.  
Thus $\Delta$ defines a natural transformation $\boldsymbol{1}_{\lMod{\Lambda}}\Rightarrow \homdual{\Lambda}{\mathrm{op}}\homdual{\Lambda}{}$ of functors  $\lMod{\Lambda}\to \lMod{\Lambda}$ given by injective maps $M\to \homdual{\Lambda}{\mathrm{op}}(\homdual{\Lambda}{}(M))$. 
% $(\rMod{\Lambda})^{\mathrm{op}}\to (\rMod{R})^{\mathrm{op}}$ and  $\lMod{\Lambda}\to\lMod{R}$. 

Now assume $R$ is complete in the $\maxideal$-adic topology. 
Hence $\End{\lMod{R}}(\env)\cong R$ from the discussion above, and so $\homdual{R}{}(k)\cong k$ by applying $\homdual{R}{}$ to the injective envelope $k\to \env$. 
This means that $\homdual{R}{}(V)$ is noetherian for any artinian $R$-module $V$, and if $V$ has finite length then $\homdual{R}{}(V)$ has the same length. 
% Thus by assuming $R$ is complete as above, we are ready to recall Matlis duality. 

Let $\lfl{\Lambda}{R}$  and $\rfl{R}{\Lambda}$ be the full subcategories of $\lMod{\Lambda}$ and $\rMod{\Lambda}$ consisting of finite-length $\Lambda$-modules, or equivalently\footnote{Since $\Lambda$ is module-finite over $R$ we have $\Lambda \maxideal\subseteq \jacrad$ and so $\maxideal$ annihilates simple $\Lambda$-modules, and so such simples are finite-dimensional vector spaces over $k$.  
It follows that finite-length $\Lambda$-modules are finite-length over $R$. 
The converse is trivial. }, $\Lambda$-modules that have finite length over $R$.  
Hence as $R$ is complete the functor $\homdual{\Lambda}{}=\Hom{\lMod{R}}(-,\env)$ restricts to an equivalence  $\matdual{\Lambda}{}\colon \lfl{\Lambda}{R}\to (\hspace{-0.5mm}\rfl{R}{\Lambda})^{\mathrm{op}}$ called the \emph{Matlis dual}. 
Namely, from the discussion so far we have $\matdual{\Lambda}{\mathrm{op}}\matdual{\Lambda}{}\cong \boldsymbol{1}_{\lfl{\Lambda}{R}}$.  
% Next we apply $\matdual{\Lambda}{}$ to short exact sequences. 

\subsection{Module-finite algebras over complete local rings} 
\label{sec-mod-finite-over-complete}

Recall $\Lambda$ is a module-finite algebra over a noetherian local ring $(R,\maxideal,k)$. 
Hence $\Lambda$ is noetherian and semilocal. 
If we assume $R$ is $\maxideal$-adically complete one can show the ring $\Lambda$ is   also semiperfect; see for example \cite[Example 23.3]{Lam1991}.

\begin{setup}
\label{setup-completeness-of-R-assumtpion}
We restrict \Cref{setup-mod-finite-over-R}, where we assumed already that $\Lambda$ is a module-finite algebra over a local ring $(R,\maxideal,k)$. 
In \S\ref{sec-mod-finite-over-complete} we additionally assume firstly that $R$ is $\maxideal$-adically complete, and so $\Lambda$ is semiperfect with a complete and orthogonal set $\boldsymbol{e}=\{e_{1},\dots,e_{n}\}$ of local idempotents. 
\end{setup}

From the completeness of $R$ one yields useful information about $\Lambda$. 
We already saw that $\Lambda$ is semiperfect. 
Furthermore the centre of $\Lambda$  is  isomorphic to a product of complete local noetherian rings; see for example  \cite[Proposition 4.3.2]{HunSwa2006}. 
By \S\ref{sec-completeness-closure} we can also consider the Matlis dual $\matdual{\Lambda^{\text{op}}}{}\colon \lfl{\Lambda}{R}\to (\hspace{-0.5mm}\rfl{R}{\Lambda})^{\mathrm{op}}$, an equivalence with quasi-inverse $\matdual{\Lambda}{\mathrm{op}}$. 
In \S\ref{sec-mod-finite-over-complete}
 we use $\matdual{\Lambda}{}$ to recover additional properties of $\Lambda$.

\subsubsection{Extension symmetry between simples}
\label{sec-ext-sym-simples-matlis} 
We note how  $\matdual{\Lambda}{}$ interacts with short exact sequences in  $\lfl{\Lambda}{R}$. 

 Fix a short exact sequence $\mu\colon 0\to L\to M\to N\to 0$ in  $\lMod{\Lambda}$. 
Assuming $L,N$ lie in $ \lfl{\Lambda}{R}$ so does $M$,  and applying $\matdual{\Lambda}{}$ to $\mu$ defines a short exact sequence $\matdual{\Lambda}{}(\mu)\colon 0\to \matdual{\Lambda}{}(N)\to \matdual{\Lambda}{}(M)\to \matdual{\Lambda}{}(L)\to 0$ in $\rfl{R}{\Lambda}$ by the injectivity of $\env$; see for example \cite[Exercise 3.2.4]{Wie1986}. 
From here it is straightforward to see this gives an $R$-module isomorphism  $\Ext{1}{\lMod{\Lambda}}(L,N)\cong \Ext{1}{\rMod{\Lambda}}(\matdual{\Lambda}{}(N),\matdual{\Lambda}{}(L))$ as $R$-modules. 

% \begin{lem}
%     \label{lem-duality-between-finlen-ext} 
   
% \end{lem}

% \begin{proof}
  
%    Since $L$ and $N$ lie in $ \lfl{\Lambda}{R}$ . 
%     A
%     If $\mu$ is equivalent to a short exact sequence $\eta$ then $\matdual{\Lambda}{}(\mu)$ and $\matdual{\Lambda}{}(\eta)$ are equivalent. 
    
%     Hence we have a map $\mathsf{x}\colon \Ext{1}{\lMod{\Lambda}}(L,N)\to \Ext{1}{\rMod{\Lambda}}(\matdual{\Lambda}{}(N),\matdual{\Lambda}{}(L))$. 
%     Now the action of $r\in R$ on $\mu$ can be found by the pull-back in $\lfl{\Lambda}{R}$ of $\mu$ along the endomorphism  of $L$ given by $r\boldsymbol{1}_{L}$, and applying $\matdual{\Lambda}{}$ gives the push-out of $\mathsf{x}(\mu)$ along the endomorphism  $r\boldsymbol{1}_{\matdual{\Lambda}{}(L)}$ of $\matdual{\Lambda}{}(L)$. 
%     Altogether $\mathsf{x}(\mu r)=r\mathsf{x}(\mu)$ and so this map is $R$-linear. 
%     Exchanging  $\matdual{\Lambda}{}$ with   $\matdual{\Lambda^{\mathrm{op}}}{\mathrm{op}}\colon \rfl{\Lambda}{R}\to (\hspace{-0.5mm}\lfl{\Lambda}{R})^{\mathrm{op}}$ defines the inverse of $\mathsf{x}$. 
% \end{proof}

In \Cref{lem-symmetry-between-simples-ext} we check that simple left and right $\Lambda$-modules are in bijection under the Matlis dual.  

\begin{lem}
    \label{lem-symmetry-between-simples-ext} 
    If $\Lambda$ is basic and $e\in \boldsymbol{e}$ then $\matdual{\Lambda}{}(\Lambda e/\jacrad e)\cong e\Lambda /e\jacrad$ and $\matdual{\Lambda^{\mathrm{op}}}{\mathrm{op}}(e\Lambda /e\jacrad)\cong \Lambda e /\jacrad e$ as $\Lambda$-modules. 
\end{lem}

\begin{proof}
    Let $S_{g}=\Lambda g/\jacrad g$ and  $S_{g}'=g\Lambda /g\jacrad $ for any $g\in \boldsymbol{e}$. 
    We already saw $\dim_{k}(\Lambda/\jacrad)<\infty $ since $\Lambda$ is module-finite over $R$. 
    Hence $\dim_{k}(S_{e})<\infty$ and $\dim_{k}(S_{e}')<\infty $ for each $g$.  
    For any $\mathsf{f}\in \matdual{\Lambda}{}(S_{e})$ any $\lambda\in \Lambda$ we have 
    % $(\mathsf{f}\lambda)(\overline{x})=\mathsf{f}(\lambda\overline{x})$ where $\overline{x}=x+\jacrad e$ for $x\in \Lambda e$. 
    % Hence
    $\mathsf{f}\lambda=0$ when $\lambda\in \jacrad$ and so  $\matdual{\Lambda}{}(S_{e})$ is a right module over the semisimple ring  $\Lambda/\jacrad$. 
    Hence $\matdual{\Lambda}{}(S_{e})$ is a direct sum of modules of the form $S_{g}'$. 
    
    Now let $g\in \boldsymbol{e}$ and suppose $g\ncong e$. 
    Since $\Lambda$ is basic this means  $g\Lambda e\subseteq g\jacrad e$ by considering the Peirce decompositions of $\Lambda $ and $\jacrad$ from \S\ref{basic rings}. 
    Hence $\mathsf{f}\lambda=0$ for any $\lambda\in g\Lambda $  and $\mathsf{f}\in \matdual{\Lambda}{}(S_{e})$. 
    Altogether $\Hom{\rMod{\Lambda}}(g\Lambda,\matdual{\Lambda}{} (S_{e}))=0$, and so the direct sum $\matdual{\Lambda}{} (S_{e})$ only involves copies of $S_{e}'$. 
    As in \S\ref{sec-completeness-closure} the duality $\matdual{\Lambda}{}$ preserves length over $R$ of objects in $\lfl{\Lambda}{R}$. 
    Hence $\matdual{\Lambda}{} (S_{e})\cong S_{e}'$ and dually $\matdual{\Lambda^{\mathrm{op}}}{\mathrm{op}}(S_{e}')\cong S_{e}$.
\end{proof}

\begin{lem}
\label{lem-dimension-number-of-arrows}
\emph{(c.f \cite[\S III.1, Proposition 1.14]{AusReiSma1995}, \cite[p. 2618]{CarSal1987}).} 
    If $\Lambda$ is basic and if $e,f\in\boldsymbol{e}$ then 
    \[
    \dim_{k}\left(
    \Ext{1}{\lMod{\Lambda}}(\Lambda e/\jacrad e,\Lambda f/\jacrad f)
    \right)=
    \dim_{k}\left(f\jacrad e/f\jacrad^{2}e\right)=
    \dim_{k}\left(
    \Ext{1}{\rMod{\Lambda}}(f\Lambda /f\jacrad ,e\Lambda /e\jacrad )
    \right).
    \] 
\end{lem}
\begin{proof} 
As we saw in \S\ref{subsec-rad-topol-and-krull-intersect} the ring $\Gamma=\Lambda /\jacrad$ is a finite-dimensional algebra over the residue field $k$. 
Considering $\Gamma$ instead of $\Lambda$ and $k$ as a complete local ring in the notation of \S\ref{sec-completeness-closure}, the Matlis dual $\matdual{\Gamma}{}$ specialises to the $k$-dual $\Hom{k}(-,k)$ considered as a functor $\lfl{\Gamma}{k}\to \left(\lfl{\Gamma}{k}\right)^{\text{op}}$ where $\lfl{\Gamma}{k}$ is the category of finite-dimensional left $\Gamma$-modules. 
Let  $S=\Lambda f/\jacrad f$, $T=\Lambda e/\jacrad e$ and $U=\jacrad e$. Consider the $k$-vector spaces
\[
\begin{array}{ccc}
E=\Ext{1}{\lMod{\Lambda}}(T,S), 
     &
F = \Hom{\lMod{\Lambda}}(U,S)
     &
G= \Hom{\lMod{\Gamma}}(U/\jacrad U,S).   
\end{array}
\]  
By \Cref{lem-proj-covers-ext-between-simples} we have  $E\cong F$ as $e\Lambda e/e\jacrad e$-modules. 
Since $\rad(\jacrad f)=\jacrad^{2}f$ lies in the kernel of any map $\jacrad f\to \Lambda e/\jacrad e$ we have $F=G$. 
Using that $\matdual{\Gamma}{}$ is a duality, and using  \Cref{lem-symmetry-between-simples-ext}, as $k$-vector spaces we have 
\[
\begin{split}
    G\cong 
\Hom{\rMod{\Gamma}}(\matdual{\Gamma}{}\left(S\right),\matdual{\Gamma}{}\left(U/\jacrad U\right))
\cong 
\Hom{\rMod{\Gamma}}(e\Gamma,\matdual{\Gamma}{}\left(U/\jacrad U\right))
\cong  \matdual{\Gamma}{}\left(U/\jacrad U\right)e. 
\end{split}
\]
Let $V=e\Lambda f/e\jacrad^{2} f$ considered only as a $k$-vector space. 
Consider the map $\Hom{k}(V,k)\to \matdual{\Gamma}{}\left(U/\jacrad U\right)$ given by post-composing with the map $U/\jacrad U\to V$ defined by left-multiplication with $e$. 
It is straightforward to check this map is injective and its image is $\matdual{\Gamma}{}\left(U/\jacrad U\right)e$. 
Hence $\dim_{k}(V)=\dim_{k}(E)$. 
The final equality   follows immediately by symmetry 
% , or by combining \Cref{lem-duality-between-finlen-ext} 
and \Cref{lem-symmetry-between-simples-ext}.
\end{proof}

\subsubsection{Quivers of module-finite basic algebras over complete local rings}
\label{sec-quiv-complete-case}

Using what we have gathered so far we can see that the left and right quivers are dual to one another when $\Lambda$ is basic.

\begin{lem}
    \label{lem-constant-dim-lem}
    Let $\Lambda$ be basic. 
    Then $a_{l}(i,j)=a_{r}(j,i)$ for each $i,j=1,\dots,n$ and consequently the left quiver $Q_{l}(\Lambda)$ of $\Lambda$ is the opposite quiver of the right quiver $Q_{r}(\Lambda)$ of $\Lambda$. 
\end{lem}

\begin{proof}
   For each $h=1,\dots,n$  let $\Lambda(h)=e_{h}\Lambda e_{h}/e_{h}\jacrad e_{h}$, $S_{h}=\Lambda e_{h}/\jacrad e_{h}$ and  $S_{h}'=e_{h}\Lambda /e_{h}\jacrad $  giving $k$-algebra isomorphisms between $\End{\lMod{\Lambda}}(S_{h})$, $\Lambda(h)$ and $ \End{\rMod{\Lambda}}(S_{h}')$ for fixed $h$.  
    By \Cref{lem-counting-arrows-in-quiver},
    \[
    \begin{array}{c}
    a_{l}(i,j)=
    \dim_{\Lambda(j)}\left(
    \Ext{1}{\lMod{\Lambda}}(S_{i},S_{j})\right)=
    \dfrac{
    \dim_{k}\left(
    \Ext{1}{\lMod{\Lambda}}(S_{i},S_{j})
    \right)
    }
    {
    \dim_{k}\left(\Lambda(j)\right)
    }.
 %    ,
 % \\
 % \vspace{-4mm}
 % \\
 %    a_{r}(j,i)=\dim_{\Lambda(j)}\left(
 %    \Ext{1}{\rMod{\Lambda}}(S_{j}',S_{i}')
 %    \right)=
 %    \dfrac{
 %    \dim_{k}\left(
 %    \Ext{1}{\rMod{\Lambda}}(S_{j}',S_{i}')
 %    \right)
 %    }
 %    {\dim_{k}(\Lambda(j))
 %    }
    \end{array}
    \]
   Similarly $\dim_{k}\left(
    \Ext{1}{\rMod{\Lambda}}(S_{j}',S_{i}')
    \right)=a_{r}(j,i)\dim_{k}(\Lambda(j))
    $. 
    Thus the assertion follows by \Cref{lem-dimension-number-of-arrows}. 
\end{proof}

\begin{defn}
\label{def-gabriel-quiver}
    From \Cref{setup-completeness-of-R-assumtpion} we are assuming that $\Lambda$ is a module-finite algebra over a complete local noetherian ring $R$, and so $\Lambda$ is semiperfect. 
    If $\Lambda$ is basic we write   $Q(\Lambda)$ for the \emph{Gabriel quiver} of $\Lambda$, defined by taking the left\footnote{The left and right quivers are opposite one another by \Cref{lem-constant-dim-lem}. 
By our convention (see \Cref{arrow-direction-convention-remark}) representations of a given quiver are equivalent to left modules over the path algebra, hence our choice here.} quiver of $\Lambda$. 

So, $Q(\Lambda)$ is the (left) quiver defined in \S\ref{quiver-noetherian-semiperfect}, described as follows. 
There is a unique integer $n>0$ such that any complete and orthogonal set of local idempotents for $\Lambda$ has $n$ elements; see \S\ref{subsubsec-idempotents}. 
Furthermore, after choosing such a set $\{e_{1},\dots,e_{n}\}$ of idempotents, for any $i,j$ the multiplicity $a(i,j)$ of $\Lambda e_{j}$ as a summand of  a projective cover of $\jacrad e_{i}$ is uniquely defined. 
From this point $Q(\Lambda)$ is taken to have $n$ distinct vertices $v_{i}$ corresponding to the idempotents $e_{i}$, and taken to have $a(i,j)$ distinct arrows from $v_{i}$ to $v_{j}$. 
\end{defn}

\begin{cor}
\label{cor-basic-biserial-quiver-is-biserial-plus-dim-1-nice}
    Let $\Lambda$ be basic with a complete and orthogonal set $\{e_{1},\dots,e_{n}\}$ of local idempotents. 
    \begin{enumerate}
        \item If $\Lambda$ is  biserial then any vertex $v$ in $Q(\Lambda)$ then there are at most two arrows with head $v$ and there are at most two arrows with tail $v$.  
        \item If  $e_{i}\Lambda e_{i}/e_{i}\jacrad e_{i}\cong k$ for each $i$ then $a(i,j)=\dim_{k}(e_{j}(\jacrad/\jacrad^{2})e_{i})$ for each $i$ and each $j$. 
    \end{enumerate}
\end{cor}
\begin{proof}
  For (1) combine \Cref{quivers-of-biserials} and \Cref{lem-constant-dim-lem}, and for (2) combine  \Cref{lem-counting-arrows-in-quiver} and \Cref{lem-dimension-number-of-arrows}. 
\end{proof}
There are no non-trivial division algebras over an algebraically closed field. 
For $\Lambda$ basic and $k$ algebraically closed  $Q(\Lambda)$ gives the usual construction when $\maxideal=0$; see for example \cite[\S{}II.3, Definition 3.1]{Ass2006}.

\section{Path algebras over local rings}
\label{section-definition}

In \S\ref{section-definition} we forget the assumptions from \S\ref{section-background} and \S\ref{sec-module-finite-algebras}.  
Eventually we use results from these sections, but for now the algebras we consider are constructed from path algebras  and admissible ideals.

% \begin{notn}
% Let $\Lambda$ be a ring and $n\geq0$ be an integer. 
% For any $\Lambda$-$\Lambda$-bimodule $M$ we write $M^{\otimes n}$ for the tensor power, meaning the $\Lambda$-$\Lambda$-bimodule given by 
% \[
% \begin{array}{cccc}
%     M^{\otimes 0}=\Lambda, & M^{\otimes 1}=M, 
%      & M^{\otimes n}=M\otimes_{\Lambda} M^{\otimes (n-1)} & (n>1).
% \end{array}
% \]
% We write $T_{\Lambda}(M)=\bigoplus_{n\geq 0} M^{\otimes n}$ for the tensor ring of $M$. Note that when $\Lambda$ is an $R$-algebra then $T_{\Lambda}(M)$ is naturally an $R$-algebra.
% % %where the action of $r\in R$ on the left or right of $m_{1}\otimes \dots \otimes m_{n}\in M^{\otimes n}$ is defined to be $rm_{1}\otimes \dots \otimes m_{n}$. 
% \end{notn}

Below we recall language commonly used when dealing with path algebras. 

Let $Q=(Q_{0},Q_{1},h,t)$ be a finite quiver. 
Trivial  paths are denoted  $e_{v}$  and defined for any vertex $v$, and we set  $h(e_{v})=v=t(e_{v})$. 
Non-trivial paths of length $n\geq 1$ are denoted  $a_{n}\dots a_{1}$ and are defined for any arrows $a_{1},\dots,a_{n}$ with $t(a_{i+1})=h(a_{i})$ for $1\leq i<n$. 
The head and tail of a non-trivial path $p=a_{n}\dots a_{1}$ are given by $h(p)=h(a_{n})$ and $t(p)=t(a_{1})$. 
The composition of paths $p$ of length $m$ and $q$ of length $n$ with $t(p)=h(q)$, denoted $pq$, is defined by $pq=p$ if $n=0$, $pq=q$ if $m=0$, and arrow concatenation  if $m,n>0$. 

Next we explicitly state some less common terminology. 

By a \emph{subpath} of a path $p$ we mean a path $z$ satisfying $p=qzr$ for some paths $q,r$. 
A \emph{right} subpath of $p$ refers to such a subpath $z$ where $r$ is a trivial path, and \emph{left} subpaths are where $q$ is trivial. 
The \emph{right arrow} of a non-trivial path is its unique length-$1$ right subpath, and likewise the \emph{left arrow} is the unique length-$1$ left subpath. 
Note the language of first and left arrows was used in  \cite[Definition 1.1.13]{Ben2018}.

Recall the following setup and notation from the introduction. 

\begin{setup}
\label{setup-quiver}
Throughout the remainder of the article we use the following set up. 
Let $(R,\maxideal,k)$ be a commutative noetherian local ring. 
Let $Q$ be a finite quiver with path algebra  $RQ$. 
Explicitly $RQ=\bigoplus_{\ell\geq 0} A_{\ell}$  where each $A_{l}$ is $R$-span of length-$\ell$ paths $p$ in $Q$. 
Multiplication in $RQ$  
 is defined by $R$-bilinearly extending
 \[
 \begin{array}{ccc}
 (p,q)\mapsto  pq\quad(t(p)=h(q)),     &  & (p,q)\mapsto  0\quad(t(p)\neq h(q)).
 \end{array}
 \]   
Let $A$ be the ideal in $RQ$ generated by the non-trivial paths. 
Let $\maxideal Q$ be the $R$-submodule of $RQ$ consisting of $R$-linear combinations of paths with coefficients in $\maxideal$. 
We write elements of $RQ$ in the form $\sum r_{p}p$ where $r_{p}\in R$ is finitely supported over the paths $p$ in $Q$. For example, we have 
\[
\begin{array}{cc}
A=\bigoplus_{\ell\geq 1}A_{\ell}=\{\sum r_{p}p\in RQ\mid r_{e_{v}}=0\hspace{1.5mm}\forall v\in Q_{0}\},
&
\maxideal Q=\{\sum r_{p}p\in RQ\mid r_{p}\in\maxideal\hspace{1.5mm}\forall p\}.
\end{array}
\] 
Extending the ring map $R\to k$ defines an $R$-algebra map $RQ\to kQ$ with kernel $\maxideal Q$. 

In addition to the local ring $(R,\maxideal,k)$, the quiver $Q$ and the associated notation $A_{\ell}$, $RQ$, $A$ and $\maxideal Q$ introduced above; throughout \S\ref{section-definition} we fix a two-sided ideal $I\triangleleft RQ$,  let $\Lambda=RQ/I$ and let $\jacrad=\rad(\Lambda)$. 
\end{setup}
% \begin{notn}
% \label{notn-paths}
% % \begin{itemize}
% % \end{itemize}% The head and tail of each arrow are defined by $h$ and $t$. 
% \end{notn}
% 
% \begin{notn}
% We say that a path $p$ is \emph{inadmissible}, and write $Z\mid p$, given there exists a path $z\in Z$ that is a subpath of $p$. 
% Otherwise we say $p$ is \emph{admissible} and write $Z\nmid p$. 
% \end{notn}
% Every path in the given subset $Z$ is admissible by definition. 
% \begin{itemize}
%     \item $R$ is a ring that is  unital, commutative, noetherian and local.
%     \item $\maxideal=\rad(R)$ is the maximal ideal of $R$, necessarily unique and finitely generated.
%     \item $k=R/\maxideal$ is the residue field.
% \end{itemize}

Note that if $r\in R$  then its image under the ring map $R\to RQ$ is the sum $\sum_{v} re_{v}$ over the vertices $v$. 
So if $r$ is regular in $R$ then its image is regular in $RQ$ since the paths in $Q$ form an $R$-basis of $RQ$.

\subsection{Arrow-radical, arrow-distinct and permissible ideals}

We recall terms from the introduction.  
\label{subsection-path-algebras-over-local-rings}

\begin{defn}
\label{defn-arrow-radical}
We say the ideal $I$ is \emph{arrow}-\emph{radical} if (i)--(ii) hold for any vertex $v$. 
    \begin{enumerate}[label={\upshape(\roman*)}]
        \item $\sum_{t(a)=v\,} \Lambda a$ is the unique maximal left $\Lambda$-submodule of $\Lambda e_{v}$. 
        \item $\sum_{h(a)=v\,} a\Lambda$ is the unique maximal right $\Lambda$-submodule of $e_{v}\Lambda$.
    \end{enumerate} 
\end{defn}

Note that right multiplication by an arrow   $a$ with head $v$ and tail $u$ in $Q$ defines an epimorphism   $\Lambda e_{v}\to \Lambda a$ of left $\Lambda$-modules. 
 Hence when $I$ is arrow-radical this map is a projective cover in the category of left modules. 
    If $Q$ is a single loop $x$ 
 and if $R=\mathbb{R}$ then the ideal $I=\langle x^{2}+1\rangle$ is not arrow-radical since $RQ/I\cong \mathbb{C}$ . 
% , and likewise left multiplication defines a projective cover $e_{u}\Lambda \to a\Lambda$ of right modules. 

%CHECKED UP 

\begin{example}
\label{running-2-by-2-example-part-1}
Assume $R$ is a discrete valuation ring with uniformizer $\pi$. 
So  $\maxideal=\langle \pi\rangle$ and  $\pi$ is regular in $R$. 
Let $Q$ be the quiver consisting of a pair of loops  $a$ and $b$ incident at a single vertex $u$. 
Let $I=\langle a^{2},b^{2},ab+ba-\pi e\rangle$ where $e=e_{u}$. 
We claim that $I$ is arrow-radical.  
We check condition (i) from \Cref{defn-arrow-radical} holds, and note that checking (ii) holds is similar. 
To do so we begin by claiming that $\Lambda a+\Lambda b$ is a superfluous submodule of $\Lambda e$. 
So we let $\Lambda a + \Lambda b + M = \Lambda e$ and in order to prove our claim we must show $e\in M$. 

% Recall that the admissible paths with respect to $I$ are the alternating paths in $a$ and $b$;  see \Cref{running-2-by-2-example-part-2}. 
Note that any coset of $I$ can be represented by an $R$-linear combination of paths that does not contain $a^{2}$ (respectively, 
 $b^{2}$) as a subpath. 
Furthermore, the remaining paths of length at least $3$ can be rewritten modulo $I$. 
For example multiplying $ab+ba-\pi e$ on the right by $a$ gives  $aba-a\pi \in I$, meaning  $baba-ba\pi\in I$ and so on. 
Hence  and likewise we have $e-z+I\in M$ where $z=qa+rb+sab+tba$ for some $q,r,s,t\in R$. 

Observe firstly that $az-rab-t\pi a\in I$ and $bz-qba-s\pi b\in I$. 
This gives 
\[
(qa+rb)z-\pi (qre+qta+rsb)\in I.
\]
Using the aforementioned expressions again we have the new observations that $abz-q\pi a-s\pi ab\in I$ and that $baz-r\pi b-t\pi ba\in I$. 
Taking another $R$-linear combination, but of these new expressions, we have 
\[
(tab+sba)z-\pi (qta +rsb+st\pi e)\in I.
\]
Taking the difference of the centred expressions gives $xz-\pi ne\in I$ where $x=qa+rb-sba-tab$ and $n=qr+st\pi $. 
Note also that $x-z-(s+t)\pi e\in I$ and so 
\[
(1_{R}+\pi (s+t-n))e+I=(e+x)(e-z)+I\in M.
\]
Since $R$ is local, for any $m\in\maxideal$ the element $1_{R}+m\in R$ is a unit. 
Altogether we have $e\in M$. 
Finally we have that $\Lambda a+\Lambda b$ is superfluous in $\Lambda e$, and so $\Lambda a+\Lambda b\subseteq \rad(\Lambda e)$. 

Now consider the quotient $N=\Lambda e/(\Lambda a+\Lambda b)$ and note that 
 $\Lambda \pi =\Lambda (ab+ba)$ and so the kernel of the non-zero surjective  map $R\to N$ sending $r$ to $re+\Lambda a+\Lambda b$ contains $\maxideal$. 
 In other words, $\Lambda a+\Lambda b$ is maximal, and hence is the unique maximal submodule of $\Lambda e$. 
\end{example}

\begin{defn}
\label{defn-arrow-distinct}
We say $I\triangleleft RQ$ is   \emph{arrow}-\emph{distinct} if (iii)--(iv) hold for  any arrow $a$. 
\begin{enumerate}[label={\upshape(\roman*)}]
\setcounter{enumi}{2}
    \item $\Lambda a\cap \sum \Lambda b\subseteq \rad(\Lambda a)$ where the sum runs through the arrows $b\neq a$ with $t(b)=t(a)$. 
    \item $a\Lambda \cap \sum b\Lambda \subseteq \rad(a\Lambda )$ where the sum runs through the arrows $b\neq a$ with $h(b)=h(a)$. 
\end{enumerate}
\end{defn}

In \Cref{running-2-by-2-example-part-3} we show that the ideal from \Cref{running-2-by-2-example-part-1} is arrow-distinct.

\begin{defn}
\label{defn-$Z$-permissible}
With respect to  $I$ a path $p$ in  $Q$ is said to be \emph{inadmissible} if $p\in I$ and said to be \emph{admissible} if $p\notin I$. 
We say that $I$ is \emph{permissible} provided every arrow is admissible.    
 % By a $Z$-\emph{inadmissible} path we mean one that has a subpath in $Z$. Other paths are $Z$-\emph{admissible}. 
% If the distinction is not needed we refer to \emph{inadmissible} and \emph{admissible} paths. 
% We say $I\triangleleft RQ$ is  \emph{permissible for} $Z$ provided the $Z$-inadmissible paths are precisely the paths in $I$. 
% For brevity we refer to \emph{permissible} ideals when no confusion is caused. 
%  \begin{itemize}
%     % \item \emph{Subpaths} of a path $p$ are the paths $z$ satisfying $p=qzr$ for some paths $q,r$. 
%     % \item  
%     % \item $I\triangleleft RQ$ is  $Z$-\emph{permissible} if the $Z$-inadmissible paths are precisely the paths in $I$. 
%     % \item We say that paths $p$ and $q$ \emph{share} a subpath $z$ if $z$ is a subpath of both $p$ and $q$. 
% \end{itemize}
\end{defn}
 
If $d\geq 0$ is the minimum length of an inadmissible path in $Q$ then $\Lambda p\neq 0\neq p \Lambda$ for any path $p$ of length less than $d$, and $I$ is permissible if and only if $d\geq 2$. 
Suppose $d=1$. 
Then an arrow  $a$ lies in $I$ and we consider the subquiver $Q'$ of $Q$   found by deleting $a$. 
The composition of the inclusion $RQ'\to RQ$ with the projection $RQ\to RQ/I$ is a surjective ring homomorphism. 
Hence assuming $I$ is permissible is a mild condition. 
Note that if $I$ is permissible then each trivial path lies outside $I$.
% \begin{cor}
% \label{cor-ideals-generated-by-paths-spanned-by-factorising-paths}
% The ideal $\langle Z\rangle\triangleleft RQ$ generated by $Z$ is the free  $R$-submodule with an $R$-basis given by the inadmissible paths. 
% \end{cor}
% \begin{proof}
% Any element of $Z$ is by definition $Z$-inadmissible, and any $Z$-inadmissible path $p$ has a subpath $z\in Z$, and so $p=qzr\in \langle Z\rangle$ for paths of $q,r\in RQ$. 
% % In other words,  on the one hand the set of $Z$-inadmissible paths contains $Z$, and on the other hand it is a subset of $\langle Z\rangle$.
% \end{proof}

\begin{example}
\label{running-2-by-2-example-part-2}
We continue with \Cref{running-2-by-2-example-part-1} where $Q$ is a pair of loops $a,b$ at a vertex $u$, and we let $e=e_{u}$. 
We claim that $I=\langle a^{2},b^{2},ab+ba-\pi e\rangle$ is permissible. 
For a contradiction suppose $a\in I$. 
So there exists $x\in A^{2}=\langle a^{2},b^{2},ab,ba\rangle$ and $y\in \maxideal Q$ such that $a=x+y$. 
Let $A'_{\ell}$ be the $k$-space of the length-$\ell$ paths in $kQ$. 
Since $a-x$ is sent to $0$ under the ring map $RQ\to kQ$, the image of $a\in RQ$ lies in $A'_{1}\cap \bigoplus_{\ell\geq 2}A'_{\ell}=0$. 
Since $a$ defines a non-zero element of $kQ$ this is a contradiction. 
So $a\in I$ and  $b\notin I$ by symmetry. 

We now assert that the non-trivial admissible paths in $Q$ are defined by words that alternate in the letters $a$ and $b$. 
% For the duration of the example let us call them \emph{alternating} paths. 
Let $Z=\{a^{2},b^{2}\}$ and let $G=\langle Z\rangle$, the two-sided ideal in $RQ$ generated by $Z$. 
Hence $G$ is the free $R$-span of the paths that have either $a^{2}$ or $b^{2}$ as a subpath. 
Let $H$ be the $R$-span the alternating paths, and so $RQ=G\oplus H$ as $R$-modules. 
% Let $A_{\ell}$ denote the $R$-span in $RQ$ of length-$\ell$ paths in $Q$, and let $B_{\ell}$ denote the image of $A_{\ell}$ under the ring homomorphism $RQ\to kQ$.  
Since $Z\subseteq I$ we have $G\subseteq I$. 
Let $H'$ be the $R$-span of the non-trivial alternating paths with right arrow $a$.  
Likewise let $H''$ be the $R$-span of admissible non-trivial paths with right arrow $b$. 
This means $H\cap A=H'\oplus H''$ in $RQ$. 

We now state and prove a technical claim. 
We claim that for any $v\in A\cap I$ there exist elements $y\in H'$ and $z\in H''$ such that $\pi(y+z)+v=yba+zab$. 

Let $c=ab+ba$ and note that $ac-ca,bc-cb\in G$. 
Let $v\in A\cap I$. 
Since $I$ is generated by $c-\pi e$ together with the elements of $Z$ this means $v-u(c-\pi e)\in G$ for some $u\in RQ$. 
Since $RQ=G\oplus H$ we can write $u=w+x$ for some $w\in G$ and $x\in H$. 
Hence $v-x(c-\pi e)\in G$ where $x\in H$.  

Since $\pi x +v\in xc+G\subseteq A^{2}$ and $v\in A$ we must have $x\in A$. 
Since $x\in H\cap A$ there exist unique $y\in H'$ and $z\in H''$ such that $x=y+z$. 
This gives $ya,zb\in G$ and $yba,zab\in H$ and so
\[
H\ni v-yba-zab +\pi x =v-x(c-\pi e)+yab+zba\in G
\]
and since $H\cap G=0$ this means $yba+zab=\pi x +v$, giving our claim.  
For our assertion it is necessary and sufficient to prove that $I$ cannot contain alternating words in $a$ and $b$. 
For a contradiction assume it does. 

For simplicity we just show $a(ba)^{n}\notin I$ for all $n>0$. 
Assume otherwise for a contradiction. 
Note,
\[
a(ba)^{n}=(ab)^{n-1}a(ab+ba-\pi e)-(ab)^{n-1}a^{2}b+\pi a(ba)^{n-1}
\]
and so by induction we must have $\pi^{n}a \in I$. 

Let $v=\pi^{n}a$. 
Using our technical claim we choose $y\in H'$ and $z\in H''$ such that  $yba+zab=\pi x +v\in \maxideal Q$ where $x=y+z$. 
Since $yba\in H'$ and $zab\in H''$ this means $y,z\in \maxideal Q$. 
So $x=\pi x'$ for some $x'\in RQ$, and so  $x'\in H$. 
Hence $\pi(\pi^{n-1}a-x'(c-\pi e))\in \maxideal Q\cap G$. 
Written as an $R$-linear combination of non-alternating paths with coefficients in $\maxideal$, it follows that $a\pi^{n-1}-x'(c-\pi e)\in G$. 
This is because  $\pi e$ is regular in $RQ$.  

Iterating we find  $x''\in H$ with $a-x''(c-\pi)\in G$. 
So $a\in I$, but we have seen that  this is impossible. 
Thus the non-trivial  paths outside $I=\langle a^{2},b^{2},ab+ba-\pi e\rangle$ are precisely the alternating words

% Namely, writing $B_{\ell}$ for  the $k$-subspace of $kQ$ spanned by the length-$\ell$ paths, the image of $a\in RQ$ under this map is a $k$-basis element of $B_{1}$, and the image of $w+x''(c-\pi)$ lies in $\bigoplus_{\ell \geq 2}B_{2}$. 
\end{example}        

\begin{proof}[Proof of  \Cref{thm-characterising-semiperfect-mod-fin-alg-with-brick-simples}. ]

We are assuming that $\Gamma$ is a module-finite $R$-algebra that is semiperfect and such that $\Gamma/\mathcal{K}\cong k^{n}$ where $\mathcal{K}=\rad(\Gamma)$. 
Hence $\Gamma$ is basic and we can and do fix a complete and orthogonal set $\{e_{1},\dots,e_{n}\}$ of pair-wise non-isomorphic local idempotents. 
Define a quiver $Q$ as follows. 
We let $v_{1},\dots,v_{n}$ be a complete list of distinct vertices in $Q$. 
% By assumption $e_{i}\Lambda e_{i}/e_{i}\mathcal{K} e_{i}\cong k$ for all $i$. 
% By \Cref{cor-basic-biserial-quiver-is-biserial-plus-dim-1-nice} this means that there are 
Let $d(i,j)=\dim_{k}(e_{j}(\mathcal{K}/\mathcal{K}^{2})e_{i})$ for each $i$ and $j$, and in this case we let $x_{ij}(1),\dots,x_{ij}(d(i,j))$ be the distinct arrows in $Q$ from $v_{i}$ to $v_{j}$.  
% distinct arrows in $Q$ with tail $v_{i}$ and head $v_{j}$. 

Recall from \S\ref{subsec-rad-topol-and-krull-intersect} (see in particular \cite[Proposition 20.6]{Lam1991}) that, since $\Gamma$ is module-finite over the noetherian local ring $R$, we have $\mathcal{K}^{m}\subseteq \Gamma \maxideal \subseteq \mathcal{K}$ for some $m>0$. 
For the moment fix $i,j=1,\dots,n$.  
Define elements $\gamma_{ij}(\ell)\in e_{j}\mathcal{K} e_{i}\setminus e_{j}\mathcal{K}^{2}e_{i}$ for each $\ell=1,\dots,d(i,j)$ by lifting a $k$-basis 
% $\{\overline{\gamma}_{(i,j)}(\ell)\mid 1\leq \ell\leq d\}$ 
of $e_{j}(\mathcal{K}/\mathcal{K}^{2})e_{i}$. 
Let $X_{ij}=e_{j} \mathcal{K} e_{i}$. 
Let $c>0$ be an integer. 
Let $\boldsymbol{n}_{c}$ be the set of length-$(c+1)$ sequences $\boldsymbol{i}=(i_{0},\dots,i_{c})$ of integers with $1\leq i_{t}\leq n$ for all $t$. 
For each $\boldsymbol{i}\in \boldsymbol{n}_{c}$ let $\boldsymbol{i}_{c}$ be the set of length-$c$ sequences $\boldsymbol{\ell}_{c}=(\ell_{1},\dots,\ell_{c})$ of integers with $1\leq \ell_{t} \leq d(i_{t},i_{t+1})$ for all $t<c$. 
Now for each $c$, each $\boldsymbol{i}\in \boldsymbol{n}_{c}$ and each $\boldsymbol{\ell}\in \boldsymbol{i}_{c}$ as above let $\boldsymbol{\gamma}_{\boldsymbol{i}}(\boldsymbol{\ell})=\gamma_{i_{c-1}i_{c}}(\ell_{c})\dots \gamma_{i_{0}i_{1}}(\ell_{1})$. 

We claim that any element of $X_{ij}$ can be written as a sum of an element of $e_{j}\Gamma\maxideal e_{i}$ together with an $R$-linear combination of the elements $\boldsymbol{\gamma}_{\boldsymbol{i}}(\boldsymbol{\ell})$ where  $\boldsymbol{i}\in \bigsqcup_{t=1}^{m-1}\boldsymbol{n}_{m-1}$ and $\boldsymbol{\ell}\in \boldsymbol{i}_{t}$ whenever $\boldsymbol{i}\in \boldsymbol{n}_{t}$

In case $c=1$ any element of $e_{j}\mathcal{K} e_{i}$ is the sum of one from $e_{j}\mathcal{K}^{2}e_{i}$ together with an $R$-linear combination of the elements $\boldsymbol{\gamma}_{(i,j)}(\ell)=\gamma_{ij}(\ell)$ where  $1\leq \ell\leq d(i,j)$.  
Writing $e_{j}\mathcal{K}^{2}e_{i}=\bigoplus_{h=1}^{n} e_{j}\mathcal{K} e_{h}\mathcal{K} e_{i}$, any element of $e_{j}\mathcal{K}^{2}e_{i}$ is a sum of products of elements from $e_{j}\mathcal{K} e_{h}$ and $e_{h}\mathcal{K} e_{i}$ as $h$ varies between $1$ and $n$. 

By the case $c=1$ any element of $e_{j}\mathcal{K} e_{h}$ is a sum one from $e_{j}\mathcal{K}^{2}e_{h}$ together with an $R$-linear combination of the elements $\gamma_{hj}(\ell)$ where $1\leq \ell\leq d(h,j)$.  
Combining these cases, we have that any element of $e_{j}\mathcal{K}^{2}e_{i}$ is a sum of an element from $e_{j}\mathcal{K}^{3}e_{i}$ together with an $R$-linear combination of the elements $\boldsymbol{\gamma}_{(i,h,j)}(\ell_{1},\ell_{2})=\gamma_{hj}(\ell_{2})\gamma_{ih}(\ell_{1})$ with $1\leq \ell_{1}\leq d(i,h)$ and $1\leq \ell_{2}\leq d(h,j)$. 

By iterating the argument above we have that for any $c>0$ each element of $X_{ij}$ can be written as the sum of an element from $e_{j}\mathcal{K}^{c+1}e_{i}$ together with an $R$-linear combination of elements of the form $\boldsymbol{\gamma}_{\boldsymbol{i}}(\boldsymbol{\ell})$ where $\boldsymbol{i}\in \bigsqcup_{t=1}^{c}\boldsymbol{n}_{c}$ and $\boldsymbol{\ell}\in \boldsymbol{i}_{t}$ whenever $\boldsymbol{i}\in \boldsymbol{n}_{t}$. 
By taking $c=m-1$, our claim follows since $\mathcal{K}^{m}\subseteq \Gamma \maxideal$. 

Given that $\Gamma/\mathcal{K}\cong k^{\vert Q_{0}\vert}$ we have that $e_{i}\Gamma e_{i}/e_{i}\mathcal{K} e_{i}\cong k$ for each $i$, and so any element of $e_{i}\Gamma e_{i}$ is a sum of one from $e_{i}\mathcal{K}e_{i}$ and an element of the form $re_{i}$ with $r\in R$. 
Since $\Gamma$ is basic we have $e_{j}\Gamma e_{i}=X_{ij}$ whenever $i\neq j$. 
Let $P$ denote the finite subset of $\Gamma$ consisting of the elements $e_{i}$ as $i$ runs from $1$ to $n$, together with the elements $\boldsymbol{\gamma}_{\boldsymbol{i}}$ 
where  $\boldsymbol{i}\in \bigsqcup_{t=1}^{m-1}\boldsymbol{n}_{m-1}$ and $\boldsymbol{\ell}\in \boldsymbol{i}_{t}$ whenever $\boldsymbol{i}\in \boldsymbol{n}_{t}$.

So far we have shown that any element of $\Gamma$ is an $R$-linear combination of the elements from $P$ together with an element from $\Gamma \maxideal$. 
% Write $x_{ij}(1),\dots,x_{ij}(d(i,j))$ for the distinct arrows in $Q$ from $v_{i}$ to $v_{j}$. 
For each $i$ let $e'_{i}=e_{u}$ where $u=v_{i}$. 
Define a function $\varphi\colon RQ\to \Gamma$ by sending $e'_{i}$ to $e_{i}$ and sending $x_{ij}(\ell)$ to $\gamma_{ij}(\ell)$ for each $\ell$ with $1\leq \ell\leq d(i,j)$. 
It is straightforward to check that 
% \[
% \begin{array}{cccc}
%   \sum_{i} \varphi(e_{i}') =1, & \varphi(e_{j}')^{2}=\varphi(e_{j}'), & \varphi(e_{i}')\varphi(e_{j}')=0, & \varphi(e_{j}')\varphi(x_{ij}(\ell))=\varphi(x_{ij}(\ell)) =\varphi(x_{ij}(\ell)) \varphi(e_{i}') 
% \end{array}
% \]
% by construction, and so
$\varphi$ is an $R$-algebra map. 
Let $M$ be the $R$-submodule of $\Gamma$ generated by $P$. 
Recall $\Gamma$ is assumed to be a finitely generated $R$-module. 
Our summary of what was shown above gives $P+\Gamma \maxideal=\Gamma$ and so $P=\Gamma$ by Nakayama's lemma. 
By construction $P\subseteq \im(\varphi)$, and so altogether $\varphi$ is surjective. 

Let $I=\ker(\varphi)$. 
To complete the proof we show that $I$ is arrow-radical, arrow-distinct and permissible. 

For any vertex $v_{i}$ in $Q$ the radical $\mathcal{K}e_{i}$ of $\Gamma e_{i}$ is the $\Gamma$-submodule generated by the images of the arrows $x_{ij}^{\ell}$ as $\ell$ runs from $1$ to $d(i,j)$ as $j$ runs from $1$ to $n$. 
Hence, and likewise, we have that $I$ is arrow-radical. 

To show that $I$ is arrow-distinct we fix an arrow in $Q$. 
Hence fix $i,h$ and $\ell'$ where $1\leq i,h\leq n$ and  $1\leq \ell'\leq d(g,h)$. 
Let $\lambda\in \Gamma \gamma_{ih}(\ell') \cap \sum_{j,\ell}\Gamma \gamma_{ij}(\ell)$ where $1\leq j\leq n$ and $1\leq \ell\leq d(i,j)$ for each $j$, and where $(j,\ell)\neq (h,\ell')$. 
Note $\mathcal{K}^{2}$ is generated by images of paths in $Q$ of length at least $2$. 

Hence there exist $r_{ig\partial}\in R$ for each $g=1,\dots,n$ 
 and each $\partial=1,\dots,d(i,g)$ such that 
\[
\begin{array}{ccc}
\lambda-r_{ih\ell'}\gamma_{ih}(\ell'),\,\,
\lambda-\sum_{(j,\ell)\neq (h,\ell')}r_{ij\ell}\gamma_{ij}(\ell)\in \mathcal{K}^{2}e_{i},
&
\sum_{\partial=1}^{d(i,h)}r_{ih\partial}\gamma_{ih}(\partial)\in e_{h}\mathcal{K}^{2}e_{i}.
\end{array}
\]
The containment on the right follows from the containments on the left by multiplying both expressions by $e_{h}$ and taking the difference. 
Recall the cosets $\gamma_{ih}(\partial)+e_{h}\mathcal{K}^{2}e_{i}$ define a $k$-basis for $e_{h}\mathcal{K}e_{i}/e_{h}\mathcal{K}^{2}e_{i}$. 
In particular $r_{ih\ell'}\in \maxideal$ and so $\lambda\in \Gamma \maxideal\gamma_{ih}(\ell')$. 
Since $\Gamma\maxideal \subseteq \mathcal{K}$ we have $\lambda \in \rad(\Gamma \gamma_{(i,h)}(\ell'))$. 
Thus, and by symmetry, $I$ is arrow-distinct. 
Using again the fact we lifted a $k$-basis we must have $\gamma_{ih}(\partial)\notin e_{h}\mathcal{K}^{2}e_{i}$, and so in particular $x_{ih\partial}\notin I$, for any $i$, $h$ and $\partial$. 
Thus $I$ is permissible. 
\end{proof}

\subsection{Admissible ideals}
\label{admissible-ideals} 
In \S\ref{admissible-ideals} continue to keep the notation and assumptions from \Cref{setup-quiver}. 
% based on work by Raggi-C{\'a}rdenas and Salmer{\'o}n  \cite[p. 2619]{CarSal1987} and Pe{\~n}a and Raggi-C{\'a}rdenas \cite[p. 521, Definition]{PenRag1988}. 
% For an ideal $\mathfrak{a}\triangleleft R$ let $\mathfrak{a} Q$ be the $R$-submodule of $RQ$ consisting of $R$-linear combinations of paths with coefficients in $\mathfrak{a}$. 
% Let $A_{l}\subseteq RQ$ be the $R$-span of the paths of length $l \geq 0$. 
\begin{defn}
\label{defn-admissible}
We say that $I$ is \emph{admissible}  if it is \emph{bounded above} and \emph{bounded below} in the following sense.  
\begin{enumerate}[label={\upshape(\alph*)}]
        % \item $re_{v}\notin I$ whenever $0\neq r\in R$ and $v$ is a vertex. 
        \item $I$ is said to be \emph{bounded above} if $I\subseteq A+ \maxideal Q$ and $I\cap A \subseteq A^{2}+ \maxideal Q\cap A$.
        \item $I$ is said to be \emph{bounded below} if $A^{m}\subseteq I+ \maxideal Q$ and $\maxideal Q\subseteq I+\sum_{\ell=1}^{n} A_{\ell}$ for some integers $m,n>0$. 
        \end{enumerate} 
\end{defn}

In \cite[p. 521, Definition]{PenRag1988}, condition (a) is replaced by the requirement that  $I \subseteq A^{2}+ \maxideal Q\cap A$. 
Note that if $I \subseteq A^{2}+ \maxideal Q\cap A$ then (a) holds. 
In \cite[p. 2619]{CarSal1987} the authors require, in addition to conditions (a)--(b), that  
the canonical\footnote{This is the composition of the algebra map $R\to RQ$ and the quotient $RQ\to RQ/I$ that makes $RQ/I$ an $R$-algebra. 
In \cite{CarSal1987} these authors require that the preimage of $I$ under the map $R\to RQ$ is $0$. } ring homomorphism $R\to RQ/I$ is injective. 
We compare these notions in \Cref{comparing-example}. 

\begin{example}
    \label{comparing-example}
    Let $p>0$ be a prime integer and $R=\mathbb{Z}/p^{3}\mathbb{Z}$. 
    Let $Q$ consist of one loop $x$ at a vertex $v$ and let $e=e_{v}$. 
    Here $R Q\cong R[x]$. 
    Let $I=\langle pe-x,x^{2}\rangle$. 
    Here $A=\langle x\rangle$ and $A^{2}=\langle x^{2}\rangle\subset I$. 
    
For the moment suppose $x\in I$.  
Hence we have $r_{i},s_{i}\in R$ finitely supported over $i\in \mathbb{N}$ such that $pe=\sum_{i}(pe-x)r_{i}x^{i}+s_{i}x^{i+2}$. 
Comparing coefficients of powers of $x$ as elements of $R[x]$, this gives $1=pr_{1}-r_{0}$ and $pr_{0}=0$ inside $R$. 
But then we have $r_{0}\in p^{2}\mathbb{Z}/p^{3}\mathbb{Z}$ giving the contradiction $1\in \maxideal$. 
Hence $I$ is permissible.  
Since $A^{2}\subset I$ only the paths $e$ and $x$ are admissible. 
Note also that $p^{2}e=(pe-x)(p+ex)+x^{2}$ lies in $I$. 
So $0\neq p^{2}x\in I\cap A$ meaning $R\to RQ/I$ is not injective, and so $I$ is not admissible in the sense of  \cite{CarSal1987}. 
Furthermore since $pe-x\notin A$  we have $I\nsubseteq A^{2}+\maxideal Q\cap A$, and so $I$ is not admissible in the sense of  \cite{PenRag1988}. 
    
Any element of $A$ has the form  $rx+y$ with $y\in A^{2}$ and so assuming $rx+y\in I$ gives $rx\in I$. 
 Since $x\notin I$ we have $r\in \maxideal$. 
This means $I$ is bounded above. 
Since $A^{2}\subseteq I$ and $pe=pe-x+x\in I+A_{1}$ we have that $I$ is bounded below, and so admissible in our sense. 

Note that $RQ/I\cong SQ/H$ as rings where $S=\mathbb{Z}/p^{2}\mathbb{Z}$ and $H=\langle pe-x,x^{2}\rangle $, and $H$ is admissible in the sense of  \cite{CarSal1987}. 
    Of course, both $RQ/I$ and $SQ/H$ are isomorphic to $\mathbb{Z}/p^{2}\mathbb{Z}$, meaning that $\rad(RQ/I)$ is generated by $x+I$ as a two-sided ideal. 
\end{example}

% Now suppose $I$ is admissible in our sense. 
% In \Cref{prop-admissible-presentations-extra} we take the quotient $S$ of $R$ by an  appropriate ideal and give a surjective $S$-algebra map $SQ\to \Lambda$ whose kernel is  admissible as in \cite{CarSal1987}.  
% Both sets of authors in \cite{CarSal1987,PenRag1988} assume that $R$ is $\maxideal$-adically complete, and in this case any $R$-algebra that is finitely generated as an $R$-module is semiperfect; see for example \S\ref{sec-module-finite-algebras}. 

\begin{prop}
\label{prop-admissible-presentations-extra}
Let $S=R/\mathfrak{h}$ where $\mathfrak{h}$ is the preimage of $I$ under the algebra map $R\to RQ$, and let $H$ be the image of $I$ under the quotient map $RQ\to SQ$. 
Then $\Lambda \cong SQ/H$ and the following statements hold. 
\begin{enumerate}
    \item If $I$ is bounded above in $RQ$ then $H$ is bounded above in $SQ$. 
    \item If $I$ is bounded below in $RQ$ then $H$ is bounded below in $SQ$. 
    \item A path in $Q$ is admissible with respect to $I$ if and only if it is admissible with respect to $H$. 
\end{enumerate}
Hence if $I$ is admissible in the sense of \Cref{defn-admissible} then $H$ is admissible in the sense of \cite{CarSal1987}.
\end{prop}
\begin{proof}
Let $\mathfrak{n}=\maxideal+\mathfrak{h}/\mathfrak{h}$, the image of $\maxideal$ under the quotient map $R\to S$. 
By definition $S$ is local with maximal ideal  $\mathfrak{n}$. 
Let $\mathsf{t}\colon RQ\to SQ$ be the canonical ring homomorphism so that $H$ is the image of $I$ under $\mathsf{t}$. 
Since $R\to S$ is surjective, so is $\mathsf{t}$, and the image of $A_{l}$ under $\mathsf{t}$ is the $S$-submodule of $SQ$ generated by the length-$l$ paths. 
 Likewise $\mathfrak{n}Q$ is the image  of $\maxideal Q$ under $\mathsf{t}$. 
 This gives (1) and (2). 
 
For (3), clearly paths in $I$ are sent to $H$ under $\mathsf{t}$. 
Conversley assume $p\in H$ is a path, so $p=\mathsf{t}(\sum_{q}r_{q}q)$ where $\sum_{q}r_{q}q\in I$ as $q$ runs through the paths in $Q$. 
Since $SQ$ is a free $S$-module, $\mathsf{t}(r_{p})=1$ and $r_{q}\in \mathfrak{h}$ for $q\neq p$. 
Since $\mathsf{t}(r_{p})=1$ we have $1-r_{p}\in\maxideal
$ and so $r_{p}$ has an inverse $s\in R$. 
Altogether writing $r_{p}p$ as the difference of $r_{p}p+\sum_{q\neq p}r_{q}q$ and $\sum_{q\neq p}r_{q}q$ gives $p=sr_{p}p\in I$. 
Thus the paths lying in $H$ are those of the form $\mathsf{t} (p)$ with $p$ a path in $I$. 
So (3) holds. 

For the final claim let $ s\in S$, $e=\sum e_{v}\in RQ$, $f=\sum e_{v}\in SQ$  and suppose $sf\in H$. 
By definition this means there is some $x\in I$, say of the form $x=\sum t_{p}p$ with $t_{p}\in R$ of finite support over the paths $p$, where $sf=\sum s_{p}p$ where $s_{p}=t_{p}+\mathfrak{h}$ for each $p$. 
Since $SQ$ is a free $S$-module this means  $s=s_{p}$ when $p$ is trivial and $t_{q}\in \mathfrak{h}$ for any non-trivial path $q$. 
So $x-y\in I$ where $y=\sum_{q}t_{q}q$ and this means  $t_{p}e\in I$ for $p$ trivial. 
At least one trivial path $p$ exists, and by the above $t_{p}\in \mathfrak{h}$ and so $s_{p}=0$. 
We have shown that $sf\in H$ implies $s=0$, and therefore the algebra map $S\to SQ/H$ is injective. 
\end{proof}

\begin{lem}\emph{(c.f \cite[p. 522, Proposition]{PenRag1988}). }
\label{lem-admissible-means-R-mod-finite}
 If  $I$ is bounded below  then $\Lambda$ is module-finite over $R$.
\end{lem}

\begin{proof}
Let $\Lambda=RQ/I$. 
By assumption for some $m,n>0$ we have that $\sum_{l\geq m}A_{l}=A^{m}$ is contained in $\maxideal Q+I $, and this is contained in $I+\sum_{l=1}^{n} A_{l}$. 
Consequently $\Lambda=RQ/I$ is generated as an $R$-module by the paths of length at most $\max\{m,n\}$. 
To see this, note that $RQ=\sum_{\ell\geq 0}A_{\ell}$ and that if $p$ is a path of length $\ell\geq m$ then $p\in I+\sum_{l=1}^{n} A_{l}$.
% In \cite{PenRag1988} the ring $R$ is assumed to be $\maxideal$-adically complete, and the field $k$ is assumed to be  algebraically closed.  
\end{proof}

% Note that if $I\triangleleft RQ$ is a permissible for a set $Z$ of paths in $Q$ then, in the proof above, one need only consider paths without a subpath in $Z$.   

\begin{lem}
\label{lem-arrow-direct-radical-implies-semiperfect-basic}
Let  $I$ be arrow-radical and $\Lambda=RQ/I$. 
The following statements hold for a vertex $v$.
\begin{enumerate}
    \item The idempotent  $e_{v}+I\in \Lambda$ is local.
    \item For any other vertex $u$ the idempotents $e_{u}+I\in \Lambda$  and $e_{v}+I\in \Lambda$ are not isomorphic.
\end{enumerate}
Consequently $\Lambda$ is semiperfect and basic. 
\end{lem}
\begin{proof}
(1) 
Each $e_{v}+I\in \Lambda$ is a local idempotent since $I$ is arrow-radical; see \S\ref{subsection-semiperfect-background}. 

(2) 
Note $\rad(\Lambda e_{v})=\jacrad e_{v}$ for each $v$; see for example \cite[Theorem 24.7]{Lam1991}. 
For a contradiction suppose $e_{u}-ab,e_{v}-ba\in I$ for some $a\in e_{u}RQe_{v}$ and $b\in e_{v}RQe_{u}$. 
Since $u\neq v$ the element $a\in RQ$ must be an $R$-linear combination of paths of length at least $1$. 
Hence, and likewise, we  have $a,b\in A$. 
This means $e_{u}\in I+A$ and so $e_{u}+I\in \rad(\Lambda e_{u})$. 
This gives $\rad(\Lambda e_{u})=\Lambda e_{u}$ contradicting that $I$ is arrow-radical. 

% Hence $\jacrad e_{v}$ is the unique maximal submodule of $\Lambda e_{v}$ . 
% This means $e_{v}+I$ is a local idempotent; see for example \cite[Proposition 21.18]{Lam1991}. 
Consequently, the set $\{e_{v}+I\mid v\in Q_{0}\}$ is a complete and orthogonal set of pair-wise non-isomorphic local idempotents in the quotient  $\Lambda=RQ/I$. 
\end{proof}

% In \Cref{lem-arrow-direct-radical-implies-semiperfect-basic-part-2} we note some nice properties for arrow-radical admissible ideals. 

\begin{lem}
    \label{lem-arrow-direct-radical-implies-semiperfect-basic-part-2} 
    Let  $I$ be arrow-radical, $e=e_{v}$ and $f=e_{w}$ for some vertices $v$ and $w$. Let $\ell$ be the number of arrows in $Q$ from $v$ to $w$. The following statements hold. 
    \begin{enumerate}
        \item If  $\maxideal Q\subseteq A+I$ then  $e\Lambda e/e \jacrad e\cong  k$  as rings. 
        \item If $\maxideal Q\subseteq A+I$  and $I\cap A\subseteq A^{2}+\maxideal Q\cap A$ then $f\jacrad e/f\jacrad^{2} e \cong  k^{\ell}$  as $k$-vector spaces. 
    \end{enumerate}
    Consequently if $I$ is bounded above then $\Lambda/\jacrad \cong  k^{\vert Q_{0}\vert}$ as rings and  $\jacrad/\jacrad^{2} \cong  k^{\vert Q_{1}\vert}$ as vector spaces. 
\end{lem}

% up to here 

\begin{proof}
(1)
Define $\mathsf{h}\colon R\to e\Lambda e/e\jacrad e$ by $\mathsf{h}=\mathsf{j}\circ\mathsf{i}$ where $\mathsf{i}\colon R\to e\Lambda e$ and $\mathsf{j}\colon e\Lambda e\to e\Lambda e/e\jacrad e$ are the ring homomorphisms  defined by 
$
 \mathsf{i}(r)= ere+I, 
$ and $
 \mathsf{j}(\lambda 
)= \lambda +e\jacrad e$. 
Any $\lambda \in e\Lambda e$ is a coset $\lambda=\sum r_{p}p +I$ where $r_{p}\in R$ has finite support over the paths $p$ with $h(p)=v=t(p)$. 
In case $p\neq e$ we have  $p+I\in \jacrad $ since $p$ is non-trivial and $I$ is arrow-radical. 
Hence elements of $e\Lambda e/e\jacrad  e$ have the form $\lambda+e\jacrad  e$ where $\lambda =ere+I$ for some $r\in R$. 
Hence the map $\mathsf{h}$ is surjective. 

Since $R$ is a local ring we have $\ker(\mathsf{h})\subseteq \maxideal$. 
Since  $\maxideal Q\subseteq I+A$ and $a+I\in \jacrad$ for any $a\in A$ we have $\ker(\mathsf{h})\supseteq \maxideal$. 
Thus $e\Lambda e/e\jacrad e\cong k$.

(2) 
Without loss of generality we can assume there is at least $1$ arrow from $v$ to $w$. 
Let $a_{1},\dots,a_{\ell}$ be the distinct arrows from $v$ to $w$.  
Define $\mathsf{q}\colon R^{\ell}\to f\jacrad e/f\jacrad^{2}e$ by $\mathsf{q}=\mathsf{p}\circ\mathsf{n}$ where 
% the $R$-linear maps $\mathsf{n}$ and $\mathsf{p}$ are defined as follows. 
\[
\begin{array}{ccccc}
 \mathsf{n}\colon R^{\ell}\to f\jacrad e,
 &
 \mathsf{n}(r_{1},\dots,r_{\ell})=r_{1}a_{1}+\dots+r_{\ell}a_{\ell}+I, 
 & 
 &
 \mathsf{p}\colon f\jacrad e\to  f\jacrad e/f\jacrad^{2}e,
 &
 \mathsf{p}(\lambda 
)= \lambda +f\jacrad^{2} e.
\end{array}
\]
As $R$-submodules of $RQ$ we have $fAe=\bigoplus_{i=1}^{\ell}Ra_{i}\oplus fA^{2} e$. 
Since $I$ is arrow-radical we have  $Ae+I=\jacrad e$ and so $fAe+I=f\jacrad e$ and $fA^{2}e+I=f\jacrad ^{2} e$. 
From here it is straightforward to check $\mathsf{q}$ is surjective. 

Now suppose $(r_{1},\dots,r_{\ell})\in \ker(\mathsf{q})$. 
Consider the sum $b=\sum r_{i}a_{i}$ over $i$. 
By construction $b+I\in f\jacrad ^{2} e$ and so there exists $c\in I$ such that the difference $d=c-b$ lies in $fA^{2}e$. 
In particular $c\in f(I\cap A) e$. 

By assumption this means  $c\in f(A^{2}+\maxideal Q\cap A)e$. 
Let $c'\in \maxideal Q\cap A$ and  $c''\in A^{2}$  such that $c=f(c'+c'')e$ and so the difference $d'=b-fc'e$ lies in $fA^{2}e$. 
In particular $d'$ and $b$ have the same image under the canonical map $RQ\to kQ$ induced by the residue map $R\to k$. 
Hence $r_{i}\in \maxideal$ for each $i$. 
This argument shows that $\ker(\mathsf{q})\subseteq \maxideal^{\ell}$ and the inclusion $\maxideal Q\subseteq A+I$ gives that $\maxideal^{\ell}\subseteq \ker(\mathsf{q})$. 

% , by \Cref{lem-arrow-direct-radical-implies-semiperfect-basic} we have that $\Lambda$ is basic, and so the isomorphisms $\Lambda/\jacrad \cong  k^{\vert Q_{0}\vert}$ and $\jacrad/\jacrad^{2} \cong  k^{\vert Q_{1}\vert}$ that we assert to exist are obtained by 
By (1) and (2), considering Peirce decompositions gives $\Lambda/\jacrad \cong  k^{\vert Q_{0}\vert}$ and  $\jacrad/\jacrad^{2} \cong  k^{\vert Q_{1}\vert}$. 
See \S\ref{basic rings}.  
\end{proof}

\begin{lem}
\label{lem-admissibility-theorem-part-1}
If  $I$ is arrow-radical and $\Lambda$ is finitely generated as an $R$-module then $I$ is bounded below. 
\end{lem}

\begin{proof}
Let $\Lambda=RQ/I$. 
Since $\Lambda$ is module-finite over $R$, by \S\ref{sec-module-finite-algebras} the two-sided ideal $\Lambda\maxideal$ in $\Lambda$ generated by the image of $\maxideal$ under the composition of the canonical ring maps $R\to RQ \to \Lambda$ satisfies $\Lambda\maxideal\subseteq \jacrad $ and $\jacrad^{m}\subseteq \Lambda\maxideal$ for some $m>0$.   
% It follows that $A^{m}\subseteq \maxideal Q+I$. 
Let $x\in A^{m}$, and so since $I$ is arrow-radical we have that $x+I\in \jacrad^{m}$. 
By the above this means $x+I\in \Lambda \maxideal$, and so $x-y\in I$ for some $y\in \maxideal Q$. 
Thus $x\in \maxideal Q+I$ and so $A^{m}\subseteq \maxideal Q+I$.

% follows from the equations
% \[
% \frac{A^{m}+\maxideal Q + I}{\maxideal Q + I}=
% \frac{\frac{A^{m}+\maxideal Q + I}{I\vspace{1mm}}}{\frac{\maxideal Q + I}{I}}=
% \frac{
% \frac{A^{m}+ I}{I}
% +
% \frac{\maxideal Q + I}{I}
% }{\frac{\maxideal Q + I}{I}}=
% \frac{
% (\rad(\Lambda))^{m}
% +
% \Lambda \maxideal
% }{\Lambda \maxideal}
% =0.
% \]
To complete the proof of we now show $\maxideal Q\subseteq I+\sum_{l=1}^{n} A_{l}$ for some $n> 0$. 
Firstly, since $\Lambda$ is finitely generated as an $R$-module, and since $R$ is a noetherian ring, the $R$-submodule $\Lambda \maxideal$ of $\Lambda$ is finitely generated. 
Now we choose $R$-module generators, meaning that we choose $x_{1},\dots,x_{g}\in\maxideal Q$ such that $\Lambda \maxideal=\sum_{i=1}^{g}Rz_{i}$ where $z_{i}=x_{i}+I$.  
We have already seen above that $\Lambda\maxideal\subseteq \jacrad $, and so $\maxideal Q\subseteq A+I$. 
Recall $A$ is the sum over 
$\ell\geq 1$ of the $R$-submodules $A_{\ell}$. 
So for each $i=1,\dots,g$ we have an integer $d(i)>0$ and an element $a_{i}\in A_{1}+\dots +A_{d(i)}$ such that  $x_{i}-a_{i}\in I$. 
Taking $n=\max \{d(1),\dots,d(g)\}$, the second inclusion follows. 
\end{proof}

% \begin{proof}
%     By assumption $\Gamma$ is in particular semiperfect, noetherian and basic. 
%     Choose a complete and orthogonal set $\{e_{1},\dots,e_{n}\}$ of pair-wise non-isomorphic local idempotents for $\Gamma$. 
%     In the notation of \S\ref{def-gabriel-quiver} we have $Q=Q(\Gamma)$ that is equal to the left quiver by definition, or the opposite quiver of the right quiver by \Cref{lem-constant-dim-lem}. 
%     By \Cref{cor-basic-biserial-quiver-is-biserial-plus-dim-1-nice} the number $a(i,j)$ of arrows from $v_{i}$ to $v_{j}$ in $Q$ is given  by $\dim_{k}(e_{j}\mathcal{K}e_{i}/e_{j}\mathcal{K}^{2}e_{i})$ where $\mathcal{K}=\rad(\Gamma)$. 
%     Lifting a $k$-basis defines elements $\alpha_{ij}^{(\ell)}\in e_{j}\mathcal{K}e_{i}$ 
% \end{proof}

\begin{lem}
\label{lem-admissibility-theorem-part-2}
If $I$ is arrow-radical then $I\subseteq A+\maxideal Q$. 
\end{lem}

\begin{proof}
% Assuming $I$ is arrow-radical there cannot exist a path in $I$ of the form $e_{v}$, since (for example) otherwise $\Lambda e_{v}=0$ that does not have a maximal submodule. 
% This shows $d\geq 1$. 
% We now prove $I\subseteq A+\maxideal Q$.  
Let $x\in RQ$, say of the form of a sum $x=\sum_{p}r_{p}p$ over the paths $p$ where $r_{p}\in R$ and $r_{p}=0$ for all but finitely many $p$. 
For a contradiction suppose $x\in I$ and  $x\notin A+\maxideal Q $. 
Hence we have $r_{e}\notin \maxideal$ where $e=e_{v}$ for some $v$. 
Thus $r_{e}$ is a unit in $R$, say with inverse $s$. 
Since $x\in I$ we have $esxe\in I$. 
By construction $esxe=e+ \sum_{c} sr_{c}c$ 
% \[
% \begin{array}{c}
% I\ni esxe=esye+es(x-y)e=e+ \sum_{c} sr_{c}c
% \end{array}
% \]
where $c$ runs through the (non-trivial) cycles\footnote{\emph{Cycle} here means a non-trivial path $c$ where $h(c)=t(c)$, the vertex where $c$ is \emph{incident}.}  $c$  incident at $v$. 
Now we have shown that  $Re+I \subseteq  eAe+I$, and this means $\Lambda e\subseteq \jacrad e$ since we are assuming $I$ is arrow-radical. 
This means $\Lambda e=0$ by Nakayamas lemma, and this is impossible since $\Lambda e$ is a local module.  
\end{proof}

% One may consider \Cref{lem-admissibility-theorem-part-2} and \Cref{lem-admissibility-theorem-part-3} as motivation for \Cref{defn-arrow-distinct}. 
% That is, by combining these results one has that any arrow-radical arrow-distinct and permissible ideal must be bounded above. 

\begin{lem}
\label{lem-admissibility-theorem-part-3}
If $I$ is arrow-radical, arrow-distinct and permissible then $I\cap A\subseteq A^{2}+ \maxideal Q\cap  A$. 
\end{lem}

\begin{proof}
Let $d$ denote the minimum length of an inadmissible path in $Q$. 
Let $x\in RQ$ be given by $x=\sum_{p}r_{p}p$ where $r_{p}\in R$ is finitely supported over the paths $p$. 
Let $y=\sum_{q}r_{q}q$ where $q$ runs through the paths of length less than $d$. 
So $x-y\in A^{d}$. 
% We find a contradiction assuming $x$ lies on the left, and not on the right, of either inclusion. 
We now prove the inclusion $I\cap A\subseteq A^{2}+ \maxideal Q\cap  A$ holds by contradiction. 
So suppose $x\in I\cap A$ and $x\notin A^{2}+ \maxideal Q\cap  A$. 
Writing $x=x-y+y$ we have $y\notin \maxideal Q$ since $d\geq 2$ as $I$ is permissible.  
Since $x\in A$ we have $r_{p}=0$ for all trivial paths $p$. 
As $y\notin \maxideal Q$ and $d\geq 2$ we have $r_{a}\notin \maxideal$ for some arrow $a$. 
Below we deduce a contradiction, namely that $\Lambda a=0$. 

Let $u=h(a)$ and $w=t(a)$. 
Let $P$ be the set of all paths with head $u$, tail $w$ and of length $\ell\geq 1$. 
Let $P'$ be the set of $p\in P$ of length $\ell \geq 2$ and with right arrow $a$. 
Let $P''$ be the set of $p\in P$ with right arrow $b\neq a$. 
Let $s'$ be the inverse of $r_{a}$ in $R$ and let $s_{p}=s'r_{p}$ for each path $p$. 
So, 
\[
\begin{array}{c}
a+\sum_{p\in P'}s_{p}p +\sum_{q\in P''}s_{q}q 
=e_{u}(\sum_{p\in P}s_{p}p)e_{w}
=e_{u}s'xe_{w}\in I.
\end{array}
\]
If $p\in P'$ then 
% $p=p'a$ for some non-trivial path $p'$ and this means 
% $p'+I\in \jacrad $ and so
$p+I\in \jacrad a=\rad(\Lambda a)$. 
Since $I$ is arrow-distinct this means 
\[
\begin{array}{c}
a+I\in \Lambda a\cap \left(\rad(\Lambda a) +\sum \Lambda b\right)=\rad(\Lambda a)+\left(\Lambda a \cap \sum\Lambda b\right)\subseteq \rad(\Lambda a)
\end{array}
\]
where the sum runs over arrows $b\neq a$ with $t(b)=w$. 
So $\Lambda a=0$ by Nakayamas lemma, a contradiction. 
% By construction $a+\sum_{p\in P'}s_{p}p +I=- \sum_{q\in P''}s_{q}q + I$ that lies in $\Lambda a\cap \sum \Lambda b$ where the sum runs over the arrows $b\neq a$ with tail $w$.  
% Since $I$ is arrow-direct this means $a+\sum_{p\in P'}s_{p}p\in I$. 
% Now, since WE ARE ASSUMING that there are at most $2$ arrows with head $u$ and at most $2$ arrows with tail $w$, we have $\{a_{1},\dots,a_{n}\}=\{a,b\}$ where it is possible that $a=b$. 
% Write $Q^{n}(u\leftarrow \mid a)$ for the set of paths $p$ in $Q$ of length $n$  with head $u$ and right arrow $a$ (so that $p$ also has tail $w$). 
% Hence this set is empty if $n=0$. 
% Likewise write $Q^{n}(u\leftarrow \mid b)$ for the set of paths $q$ of length $n$ in $Q$ with head $u$ and right arrow $b$.  
% Combining this simplification and the corresponding notation, the equality above becomes 
% \[
% \begin{array}{c}
% a+\sum_{n>1}\sum_{p\in Q^{n}(u\leftarrow \mid a)}s_{p}p+I=-s_{b}b - \sum_{n>1}\sum_{q\in Q^{n}(u\leftarrow \mid b)}s_{q}q + I.
% \end{array}
% \]
% Thus the expression above defines a non-zero element of $\Lambda a\cap \Lambda b$, and since  we are ASSUMING THAT $\Lambda a\cap \Lambda b=0$, this means $a+\sum_{n>1}\sum_{p\in Q^{n}(u\leftarrow \mid a)}s_{p}p\in I$. 
\end{proof}

% \begin{lem}
% If $\Gamma = RQ$,  $I\triangleleft \Gamma$ and $\Lambda=\Gamma$ then  $\Gamma a\cap \sum \Gamma b \subseteq I$ if and only if $\Lambda a\cap \sum\Lambda b=0$. 
% \end{lem}
% \begin{proof}
% Suppose $\Gamma a\cap \sum \Gamma b \subseteq I$. 
% Let $\lambda \in \Lambda a\cap \sum\Lambda b$. 
% By definition we have 
% \[
% \begin{array}{c}
% \sum_{p}r_{p}pa+I=(\sum_{p}r_{p}p+I)a=\lambda=\sum_{b}(\sum_{q(b)}s_{q(b)}q(b)+I)b =\sum_{q(b)}s_{q(b)}q(b)b+I
% \end{array}
% \]
% \end{proof}

% \begin{rem}
% \label{admissibility-injectivity-remark}
% In the notation of \Cref{prop-admissible-presentations-extra}, consider that $J$ is constructed so that  the kernel of the surjective $R$-algebra homomorphism $RQ\to SQ/J$ is $I$. 
% Likewise, consider that $\mathfrak{h}\triangleleft R$ is constructed so that the composition of the canonical maps $S\to SQ$ and $SQ\to SQ/J$ is injective. 
% With these considerations in mind, by \Cref{thm-admissible-presentations} and \Cref{prop-admissible-presentations-extra} the ideal $J\triangleleft SQ$ is admissible in the stricter sense used by  \cite{CarSal1987} discussed in \Cref{rem-different-defs-of-admissibility}. 
% \end{rem}

% \subsection{Arrow-direct ideals}

When $R$ is $\maxideal$-adically complete  the \emph{Gabriel quiver} $Q(\Gamma)$ of a basic module-finite $R$-algebra $\Gamma$ was recalled in \Cref{def-gabriel-quiver}. 
In case $k$ is algebraically closed this yields the same definition from  \cite{CarSal1987}.
% as the (left\footnote{Our convention is that $ab$ is a length two path provided $t(a)=h(b)$. 
% By this convention we use the left quiver; see \Cref{arrow-direction-convention-remark}.} quiver $Q_{l}(\Gamma)$ is the opposite quiver of  the right quiver  $Q_{r}(\Gamma)$ of $\Gamma$ from \S\ref{quiver-noetherian-semiperfect}. 
% See \Cref{lem-constant-dim-lem}. 
% In  we relabeled the  by  $Q(\Gamma)$, that we refer to as the \footnote{}. 
% We now see a slight modification of \Cref{cor-raggi-card-salm}. 
% We do not assume $k$ is algebraically closed, but instead assume $I$ is arrow-radical. 

% \begin{cor}
% \label{prop-cool-gabriel-thing}
%     Let $R$ be $\maxideal$-adically complete and let $I$ be arrow-radical. 
%     If $I$ is admissible then $\Lambda$ is $R$-module finite, semiperfect and basic and $Q$ is the Gabriel quiver of $\Lambda$. 
% \end{cor}

% \begin{proof}
%     % By \Cref{lem-admissible-means-R-mod-finite} $\Lambda$ is module-finite over $R$. 
%     % By \Cref{lem-arrow-direct-radical-implies-semiperfect-basic} $\Lambda$ is semiperfect and basic. 
    
% \end{proof}

\begin{proof}[Proof of \Cref{thm-admissible-presentations}.]
We are assuming $I$ is arrow radical. 

(1) 
By \Cref{lem-arrow-direct-radical-implies-semiperfect-basic} $\Lambda$ is semiperfect, basic and the elements $e_{v}+I$ are local and pair-wise non-isomorphic. 
 
(2)
By \Cref{lem-admissible-means-R-mod-finite} if $I$ is bounded below then $\Lambda$ is module-finite over $R$. 
By \Cref{lem-admissibility-theorem-part-1} the converse holds since $I$ is arrow-radical. 

(3)
By \Cref{defn-$Z$-permissible} and \Cref{defn-admissible}  if $I$ is not permissible then $I$ is not bounded above. 
Here we are assuming additionally that $I$ is arrow-distinct. 
So the converse follows from  \Cref{lem-admissibility-theorem-part-2} and \Cref{lem-admissibility-theorem-part-3}.  

(4)
By \Cref{lem-arrow-direct-radical-implies-semiperfect-basic-part-2}, given that $I$ is admissible  we have $\Lambda /\jacrad\cong k^{\vert Q_{0}\vert}$ as rings and $\jacrad/\jacrad^{2}\cong k^{\vert Q_{1}\vert}$. 

(5) 
% By  \Cref{def-gabriel-quiver}, since $R$ is complete the Gabriel quiver $Q(\Lambda)$ is defined. 
    Recall $Q(\Lambda)$ has vertices $v_{1},\dots, v_{n}$ where $n$ is the size of set $\{e_{v}+I\mid v\in Q_{0}\}$ from (1).  
    So $Q(\Lambda)$ and $Q$ both have $n$ vertices.  
    By \Cref{lem-counting-arrows-in-quiver} the number of arrows in $Q(\Lambda)$ from $u$ to $v$ is equal to the rank of the left module  $\Ext{1}{\lMod{\Lambda}}(\Lambda e_{u}/\jacrad e_{u},\Lambda e_{v}/\jacrad e_{v})$  over the division ring $\End{\lMod{\Lambda}}(\Lambda e_{v}/\jacrad e_{v})\cong e_{v}\Lambda e_{v}/e_{v}\jacrad e_{v}$. 
     By part (4) the division ring $e_{v}\Lambda e_{v}/e_{v}\jacrad e_{v}$ is isomorphic to $k$. 
     By \Cref{lem-dimension-number-of-arrows} there are $\dim_{k}(e_{v}\jacrad e_{u}/e_{v}\jacrad^{2} e_{u})$ arrows in  $Q(\Lambda)$ from $u$ to $v$. 
     By part (4) this is the number of arrows in $Q$ from $u$ to $v$. 
     % Thus $Q$ and $Q(\Lambda)$ are isomorphic quivers. 
\end{proof}

% Our proof of \Cref{thm-characterising-semiperfect-mod-fin-alg-with-brick-simples} is self-contained. 

   % In case $k$ is algebraically closed we recover \Cref{cor-raggi-card-salm}, a result from \cite[p.2619]{CarSal1987}. 
% For context we recall results from \cite{CarSal1987} concerning  the Gabriel quiver; see \S\ref{sec-completeness-closure}. 

\begin{cor}\emph{(\cite[(5), Proposition]{CarSal1987}).} 
\label{cor-raggi-card-salm}
Let $R$ be $\maxideal$-adically complete and  $k$ be algebraically closed. 
\begin{enumerate}
    \item If  $\Gamma$ is a module-finite $R$-algebra that is also basic with Gabriel quiver $Q$ then there exists a surjective $R$-algebra homomorphism $RQ\to \Gamma$ whose kernel is admissible\footnote{In \Cref{cor-raggi-card-salm} we use the term \emph{admissible} to mean the definition used in \cite{CarSal1987}. 
Using \Cref{prop-admissible-presentations-extra}, in the proof of (2) we reduce to this case from \Cref{defn-admissible}.}. 
     \item  If $I$ is admissible then $\Lambda=RQ/I$ is $R$-module finite, semiperfect, basic and has Gabriel quiver $Q$. 
\end{enumerate}
\end{cor}
\begin{proof}
    Note that the construction of the Gabriel quiver from \cite{CarSal1987}  is equivalent to \Cref{def-gabriel-quiver} by \Cref{cor-basic-biserial-quiver-is-biserial-plus-dim-1-nice}. 
Since $k$ is algebraically closed any finite-dimensional division algebra over $k$ is isomorphic to $k$. 
Part (1) follows by \Cref{thm-characterising-semiperfect-mod-fin-alg-with-brick-simples} and \Cref{prop-admissible-presentations-extra} and part (2) follows from \Cref{thm-admissible-presentations}. 
\end{proof}

\section{String algebras over local rings}
\label{sec-stringalgebras-over-local}

In \S\ref{sec-stringalgebras-over-local} we continue to use the conventions and notation from \Cref{setup-quiver}.

\subsection{Special ideals}

% Eventually we use the discussions in \S\ref{section-background}, but for now the algebras we consider are constructed using path algebras over quivers. 

\Cref{defn-spec-quiver-and-relations} is motivated by work of  Skowro{\'n}ski and Waschb{\"u}sch \cite{SkoWas1983}.

\begin{defn}
\label{defn-spec-quiver-and-relations} 
We say that $I$ is \emph{special} if, for any arrow $b$ in $Q$, (SP1) and (SP2) below hold.
\begin{enumerate}
    \item[(SP1)] There exists at most $1$ arrow $a$ such that $ab$ is an  admissible length-$2$ path.
    \item[(SP2)] There exists at most $1$ arrow $c$ such that $bc$ is an  admissible length-$2$ path.
\end{enumerate}
\end{defn}

Note \cite[\S 1, (SP)]{SkoWas1983} includes a restriction on the number of arrows with a given head or tail.  
These restrictions arose when considering the quiver of a basic and biserial module-finite algebra over a complete local ring; see \Cref{cor-basic-biserial-quiver-is-biserial-plus-dim-1-nice}. 
We include these restrictions in \Cref{defn-string-pairs-and-algebras}. 

%UP TO HERE

\begin{setup}
\label{string-setup}
In \S\ref{sec-stringalgebras-over-local} we let $I$ be an ideal in $RQ$ and we let $\Lambda=RQ/I$. 
We assume that $I$ is arrow-radical, bounded below, permissible and special. 
Fix a non-trivial admissible path $p$. 
\end{setup}

\subsubsection{Uniserials} Note that applying \Cref{thm-admissible-presentations}  to \Cref{string-setup} shows, for example, that $I$ is bounded above and below. 
\Cref{lem-initial-termina-subpaths-special} can be found in the authors PhD thesis; see \cite[Lemma 1.1.14]{Ben2018}.

\begin{lem}
\label{lem-initial-termina-subpaths-special}
Let  $p'$ be a non-trivial admissible path. 
The following statements hold. 
\begin{enumerate}
    \item If $p$ and $p'$ have the same right arrow, then one is a right subpath of the other.
    \item If $p$ and $p'$ have the same left arrow, then one is a left subpath of the other.
\end{enumerate}
\end{lem}
\begin{proof}
We only show (1) holds, since (2) is dual. 
Let $b$ be the common right arrow, so that $p=qb$ and $p'=q'b$ for paths $q,q'$ in $Q$. 
If one of $q$ or $q'$ is trivial we are done, so assume otherwise, and let $a$ and $a'$ be the right arrows of $q$ and $q'$ respectively. 
It is impossible that some $z\in I$ occurs as a subpath in $p$, and this means $ab,a'b\notin I$ and so $a=a'$ by (SP1). 
Thus when all paths involved are non-trivial, if $p$ and $p'$ have the same right arrow then $q$ and $q'$ have the same right arrow. 
Repeating gives the claim.
\end{proof}

In Lemmas \ref{lem-left-uniserial-cyclic-modules-from-admissible-paths}, \ref{lem-right-uniserial-cyclic-modules-from-admissible-paths}, \ref{lem-left-uniserial-cyclic-radicals}, \ref{lem-right-uniserial-cyclic-radicals}, \ref{lem-left-uniserial-cyclic-modules-uniqueness} and \ref{lem-right-uniserial-cyclic-modules-uniqueness}  we combine a ring-theoretic observation with a combinatorial one, namely, we combine \S\ref{sec-module-finite-algebras} with \Cref{lem-initial-termina-subpaths-special}.  
These observations were made in the authors thesis \cite[Corollary 1.1.17]{Ben2018}, but we repeat the proof here for completeness in the generality we are working in.

\begin{lem}
\label{lem-left-uniserial-cyclic-modules-from-admissible-paths} 
Any non-trivial left $\Lambda$-submodule of $\Lambda p$ is of the form $\Lambda q$ where $q$ is an admissible path and such that $p$ is a right subpath of $q$. 
\end{lem}
\begin{proof}

By \Cref{lem-initial-termina-subpaths-special}(1) for any arrow $a$ there is a total order defined on the set of admissible paths whose right arrow is $a$ that is equivalent to  considering path length. 
Let $0\neq M$ be a $\Lambda$-submodule of $\Lambda p$. 
Let $P$ be the set of admissible paths $q$ such that $p$ is a right subpath of $q$ and such that $M\subseteq \Lambda q$. 
Since $p$ is non-trivial any pair of elements in $P$ share the same right arrow. 
Restrict the total order above to $P$. 

We are assuming $\Lambda$ is finitely generated as an $R$-module and so we may apply the \emph{Krull intersection property} for finitely generated $\Lambda$-modules recalled in  \S\ref{subsec-krull-intersect}. 
That is, for any finitely generated $\Lambda$-module $L$ we have that 
$\bigcap_{n>0}\jacrad^{n}L=0$. 
Observe now that $P$ must be finite, as  otherwise for any $n>0$ there exists  $q\in P$ of length $\ell>n$ and hence $M\subseteq \jacrad ^{n}e_{t(q)}$  contradicting the Krull-intersection property for $L=\Lambda e_{t(q)}$ since $M\neq 0$.  
Since $p\in P$,  $P\neq\emptyset$, and  we choose $q\in P$ of maximal length. 

By construction $M\subseteq \Lambda q$ and for a contradiction we suppose $M\neq \Lambda q$. 
Let $e$ be the trivial path at $h(q)$ and let $L$ be the set of $\lambda\in \Lambda e$ such that $\lambda q\in M$. 
Since $M\neq \Lambda q$ we must have $e+I\notin L$. 
Since $I$ is arrow-radical the left $\Lambda$-module $\Lambda e$ is local and hence $L$ is contained in $\rad(\Lambda e)=\sum_{t(a)=h(q)}\Lambda a$. 

Without loss of generality we can assume there exists at least $1$ arrow $c$ with tail $h(q)$ such that $cq$ is admissible. We are given that $p$ is non-trivial, and so $q$ is non-trivial. 
We are also assuming that $I$ is special,  and so for an arrow $b\neq c$ with tail $h(q)$ the path $bq$ is inadmissible, and so $bq\in I$ for any such $b$. 

For any $\lambda '\in M$ we have $\lambda'=\lambda q$ for some $\lambda \in L$ since $M\subseteq \Lambda q$. 
In this case, from the argument above there exists $\mu\in \Lambda$ such that $\lambda q=\mu cq$. 
Thus we have $M\subseteq \Lambda cq$ and this means $cq\in P$, contradicting the maximality of $q$. 
Hence $M=\Lambda q$.  
\end{proof}

A dual argument yields the following result concerning right modules.

\begin{lem}\label{lem-right-uniserial-cyclic-modules-from-admissible-paths} 
Any non-trivial right $\Lambda$-submodule of $p\Lambda $ is of the form $q\Lambda$ where $q$ is an admissible path and such that $p$ is a left subpath of $q$. 
\end{lem}

Recall uniserial noetherian modules are local. 
A direct consequence of \Cref{lem-left-uniserial-cyclic-modules-from-admissible-paths} is \Cref{lem-left-uniserial-cyclic-radicals} that describes the maximal submodules of the uniserial modules we are dealing with. 

\begin{lem}
\label{lem-left-uniserial-cyclic-radicals} 
There is a unique maximal $\Lambda$-submodule of $\Lambda p$ that is  either trivial or of the form $\Lambda xp$ for $x$ an arrow such that $t(x)=h(p)$ and such that $xp$ is admissible. 
\end{lem}

\begin{proof}
By \Cref{string-setup} we are assuming $I$ is bounded below and arrow radical. 
By \Cref{thm-admissible-presentations} this means $\Lambda$ is module-finite as an algebra over $R$. 
By the discussion in \S\ref{sec-module-finite-algebras}, since $R$ is local noetherian this means the ring $\Lambda$ is semilocal. 
Hence $\rad(\Lambda p)=\jacrad p$.  

By assuming this module is non-trivial, since $I$ is arrow-radical there must exist an arrow $c$ with tail $h(p)$ and such that $cp$ is admissible. 
By \Cref{lem-left-uniserial-cyclic-modules-from-admissible-paths} any submodule of $\Lambda p$ has the form $\Lambda q$ where $q$ is admissible and where $p$ is a right subpath of $q$. 
Thus $\Lambda cp$ is the unique maximal $\Lambda$-submodule of $\Lambda p$.
\end{proof}

As above, we state without proof the result concerning right modules. 

\begin{lem}
\label{lem-right-uniserial-cyclic-radicals} 
There is a unique maximal $\Lambda$-submodule of $p\Lambda$ that is  either trivial or of the form  $px\Lambda $ for $x$ an arrow such that $t(p)=h(x)$ and such that $px$ is admissible. 
\end{lem}

% \subsubsection{Covers} From Setup \ref{string-setup} we are already equipped with a special pair $(Q,I)$.  

\begin{notn}
\label{notn-proj-covers}
We use the following notation for the maps given by multiplication by $p$.
\[
\begin{array}{ccccc}
 \rightmult{p}\colon \Lambda e_{h(p)}\to \Lambda p,    
 & \lambda e_{h(p)}\mapsto\lambda p, &  \hspace{10mm}&
 \leftmult{p}\colon e_{t(p)}\Lambda \to p\Lambda,   
 &
e_{t(p)}\lambda \mapsto p\lambda.
\end{array}
\] 
The symbol $\mathsf{m}$ is intended as a mnemonic, standing for \emph{multiplication} by $p$. 
\end{notn}

\Cref{lem-left-uniserial-cyclic-modules-uniqueness} and \Cref{lem-right-uniserial-cyclic-modules-uniqueness} concern the uniqueness of the uniserials generated by admissible paths.

\begin{lem}
\label{lem-left-uniserial-cyclic-modules-uniqueness} 
If $q$ is a non-trivial admissible path with the same right arrow as $p$ and $\Lambda p=\Lambda q$  then $p=q$. 
\end{lem}

\begin{proof}
Using that $q$ is non-trivial and assumed to have the same right arrow as $p$, by \Cref{lem-initial-termina-subpaths-special} either $p$ is a right subpath of $q$ or $q$ is a right subpath of $p$. 
Without loss of generality we can assume $p=p'q$ for some path $p'$ with tail $h(q)$. 
Let $e$ be the trivial path at $h(q)$ and consider $\rightmult{q}\colon \Lambda e\to \Lambda q$.

Note that $\rightmult{q}\neq 0$ since $q$ is admissible, and since $I$ is arrow-radical $\ker(\rightmult{q})$ is contained in the unique maximal submodule $\jacrad e$ of $\Lambda e$ generated by the arrows with tail $h(q)$. 
Since $q\in \Lambda q=\Lambda p'q$ there exists $\lambda \in \Lambda$ such that $e-\lambda p'\in \ker(\rightmult{q})$. 
Now if one assumes $p'$ is a non-trivial path we have $\lambda p'\in \jacrad e$, that would mean that $e\in \jacrad e$ and hence that $\Lambda e=0$ by Nakayamas lemma, and this is impossible. 
% We have concluded that $p'$ is trivial, and since $p=p'q$ this means $p=q$. 
%%
%
%%
\end{proof}

\begin{lem}
\label{lem-right-uniserial-cyclic-modules-uniqueness} 
If $q$ is a non-trivial admissible path with the same left arrow as $p$ and $p\Lambda =q\Lambda $  then $p=q$. 
\end{lem}

\begin{proof}[Proof of \Cref{thm-main-string-biserial-multi}.]
By assumption $I$ is arrow-radical, bounded below and special. 
Furthermore $p$ is an admissible non-trivial path with left arrow $b$ and right arrow $a$. 

(1) 
Any left submodule of $\Lambda p$ has the form $\Lambda q$ for $q$ admissible with right subpath $p$ by \Cref{lem-left-uniserial-cyclic-modules-from-admissible-paths}.  

(2) 
Any right submodule of $p\Lambda $ has the form $q\Lambda$ for $q$ admissible with left subpath $p$ by \Cref{lem-right-uniserial-cyclic-modules-from-admissible-paths}.  

For (3) and (4) we fix an admissible path $q$. 
We only prove (3) since (4) is symmetric. 

(3) If $q$ has right arrow $a$ and $\Lambda p=\Lambda q$  then $p=q$ by \Cref{lem-right-uniserial-cyclic-modules-uniqueness}. 
\end{proof}

\subsection{Arrow-direct ideals} Our definition of a string algebra over a local ring uses the following notion.  

\begin{defn}
\label{defn-arrow-direct}
We say $I$ is   \emph{arrow}-\emph{direct} if (v) and (vi) below hold for  any arrow $a$. 
\begin{enumerate}[label={\upshape(\roman*)}]
\setcounter{enumi}{4}
    \item $\Lambda a\cap \sum \Lambda b=0$ where the sum runs through the arrows $b\neq a$ with $t(b)=t(a)$. 
    \item $a\Lambda \cap \sum b\Lambda =0$ where the sum runs through the arrows $b\neq a$ with $h(b)=h(a)$. 
\end{enumerate}
\end{defn}

Hence any arrow-direct ideal it is arrow-distinct. 
% Note that if $I$ is arrow-direct then  any sum of the form $\sum \Lambda a$ over a set of distinct arrows, without multiplicity, must be  direct sum. 

\begin{example}
\label{running-2-by-2-example-part-3}
We continue with \Cref{running-2-by-2-example-part-1},  where $Q$ was given by a pair of loops $a,b$ incident at $u$ and $I=\langle a^{2},b^{2},ab+ba-\pi e\rangle$ where $\maxideal=\langle \pi \rangle$ for $\pi \in R$ regular. 
We claim that $I$ is arrow-direct. 

Hence it suffices to show $\Lambda a\cap \Lambda b=0$. 
Recall $G=\langle  a^{2},b^{2}\rangle$ is the $R$-span of the inadmissible paths (so those with $a^{2}$ or $b^{2}$ as a subpath) and recall: $H$ is the $R$-span of the (other paths, that is, the) admissible paths, so that $RQ=G\oplus H$;  $H'$ is the $R$-span of the non-trivial admissible paths that have right arrow $a$; and  $H''$ is the $R$-span of the non-trivial admissible  paths with right arrow $b$. 
See \Cref{running-2-by-2-example-part-2}. 

Suppose $d\in RQ$ such that $d+I\in \Lambda a\cap \Lambda b$. 
Note that $d\in A$, and written as an element of $H\oplus G$, we can assume $d\in H$ and so there exist $s\in H'$ and $t\in H''$ such that $s-d,d-t\in I$. 
In particular $v\in I$ where $v=s+t\in H$. 
The technical claim from \Cref{running-2-by-2-example-part-2} gives the existence of $y\in H'$ and $z\in H''$ where
\[
H'\ni y\pi +s-yba=-z\pi -t+zab\in H''. 
\]
By construction $H'\cap H''=0$, so $s=y(ab+ba-\pi e)-yab\in I$, and so $d\in I$. 
\end{example}

% \Cref{uniserial-module-left-right-compatability-one-side} and \Cref{uniserial-module-left-right-compatability-other-side} give   \cite[Corollary 1.2.14]{Ben2018}. 

\begin{lem}
    \label{uniserial-module-left-right-compatability-one-side}
If $I$ is arrow-direct and if $p$ has left arrow $a$ then  $\jacrad p\cap a\Lambda\subseteq p\jacrad $ where $\jacrad=\rad(\Lambda)$.
\end{lem}

\begin{proof}
Fix an arbitrary element of $\jacrad p\cap a\Lambda$, that must have the form $\lambda p$ with $\lambda\in\jacrad $. 
Since $I$ is arrow-radical we have $\lambda=\sum_{q}r_{q} q +I$ for  $r_{q}\in R$ finitely supported over the non-trivial admissible paths  $q$. 
Note that $r_{q}=0$ when $h(q)\neq h(a)$ since $\lambda p\in a\Lambda\subset e_{h(a)}\Lambda$. 

Let $P$ be the set of admissible paths $q$ with left arrow $a$, $t(q)=h(p)$ and with $qp$ admissible. 
Let $\mu=\sum_{q\notin P}r_{q}q + I$. 
Hence $(\lambda-\mu)p\in a\Lambda$ and  $\mu\in \sum b\Lambda$ where the sum  runs over arrows $b\neq a$ with $h(b)=h(a)$. 
Since $\lambda p \in a\Lambda $ we have $\mu p\in a\Lambda $. 
Since  $\mu p\in \sum b\Lambda$ and $I$ is arrow-direct this gives  $\lambda p=\sum_{q\in P}r_{q}qp+I$. 

Now, by construction, for any $q\in P$ the path $qp$ is  admissible, has left arrow $a$ and is longer than $p$. 
By \Cref{lem-initial-termina-subpaths-special} this means any  $q\in P$ satisfies $qp=pq'$ for some admissible path $q'$ that depends on $q$. 
Note than any such $q'$ must be non-trivial since $q$ is non-trivial, and since they must have the same length. 
Furthermore, there exists an arrow $x$ such that any such $q'$ has left arrow $x$. 

Combining this observation with the equation $\lambda p=\sum_{q\in P}r_{q}qp+I$, we have that $\lambda p\in px\Lambda$. 
From here it suffices to recall that  $px\Lambda=p\jacrad $ by \Cref{lem-right-uniserial-cyclic-radicals}. 
\end{proof}

\Cref{uniserial-module-left-right-compatability-one-side} follows from a dual argument to that used in the proof of \Cref{uniserial-module-left-right-compatability-other-side}. 

\begin{lem}
    \label{uniserial-module-left-right-compatability-other-side}
If $I$ is arrow-direct and $p$ has right arrow $a$ then $p\jacrad \cap \Lambda a\subseteq \jacrad p$  where $\jacrad=\rad(\Lambda)$.
\end{lem}

\begin{rem}
    \label{rem-projective-covers-by-path-multiplication} 
    Recall that by \Cref{string-setup} we are assuming that $I$ is arrow-radical. 
    
For any admissible path $p$ the maps $\rightmult{p}$ and $\leftmult{p}$ respectively give projective covers of the left module $\Lambda p$ and the right module $p\Lambda$.  
That these maps are projective covers follows from the fact that any idempotent of the form  $e_{v}+I\in \Lambda $  is (local, and hence)  primitive by \Cref{lem-arrow-direct-radical-implies-semiperfect-basic}; see \Cref{noetherian-semiperfect-lemma}. 
Note also that if $ap$ is inadmissible for each arrow $a$ with tail $v$ then $\ker(\rightmult{p})=\sum_{t(a)=v}\Lambda a=\jacrad e_{v}$. 
Dually if $pa$ is inadmissible for each arrow $a$ with head $v$ then $\ker(\leftmult{p})=\sum_{h(a)=v}a \Lambda =e_{v}\jacrad $. 
\end{rem}

% Together with \Cref{rem-projective-covers-by-path-multiplication}, \Cref{lem-projective-covers-and-their-kernels-for-uniserial-left-modules} and \Cref{lem-projective-covers-and-their-kernels-for-uniserial-right-modules} complete a description of the kernel of the projective covers from \Cref{notn-proj-covers}. 

\begin{lem}
    \label{lem-projective-covers-and-their-kernels-for-uniserial-left-modules}
    Let $I$ be arrow-direct, $v=h(p)$ and $a$ be an arrow with tail $v$ such that $ap$ is admissible. 
    The following statements hold for the sum $K_{a}=\bigoplus\Lambda b$ over arrows $b\neq a$ with $t(b)=v$. 
        \begin{enumerate}
            \item If $qp$ is admissible for every path $q$ with right arrow $a$ then $\ker(\rightmult{p})=K_{a}$. 
            \item If $q$ is the shortest path with right arrow $a$ and such that $qp$ is inadmissible then $\ker(\rightmult{p})=K_{a}\oplus \Lambda q$. 
        \end{enumerate}
\end{lem}

\begin{proof}
Let $e$ be the trivial path at $v$ and let  $\lambda=\sum_{q\in P}r_{q}q+I$ be some element of $\Lambda e$ where the elements  $r_{q}\in R$ have finite support over the set $P$ of paths $q$ with tail $v$. 
Let $P'$ be the set of $q\in P$ with right arrow $a$ and let $P''=P\setminus P'$. 
So $e\in P''$. 

Since $p$ is admissible we cannot have $\rightmult{p}=0$, and so $\ker(\rightmult{p})\subseteq K_{a}+\Lambda a$ since $I$ is arrow-radical. 
Note also this sum is direct, since $I$ is arrow-direct. 
Assume now $\lambda\in \ker(\rightmult{p})$. 
By definition this means $\sum_{q\in P}r_{q}qp\in I$, and by our previous discussion  $r_{e}=0$. 
By \Cref{string-setup} we have that $I$ is special,  meaning $qp\in I$ for all non-trivial paths $q\in P''$. 
So far we have $r_{q}qp\in I$ for all $q\in P''$. 
Thus $\sum_{q\in P'}r_{q}qp\in I$.

Let $M_{a}=\Lambda a\cap \ker(\rightmult{p})$, and so $\ker(\rightmult{p})=K_{a}+M_{a}$ by what we have seen above. 
Since we are assuming $I$ to be arrow-direct, this sum is direct, so $\ker(\rightmult{p})=K_{a}\oplus M_{a}$. 

If $M_{a}\neq 0$ then it is a non-trivial submodule of $\Lambda a$, and so by 
 \Cref{lem-left-uniserial-cyclic-modules-from-admissible-paths} we have $M_{a}=\Lambda q'$ where $q'\in P'$. 
In this case, the inclusion $M_{a}\subseteq  \ker(\rightmult{p})$ gives $q'p\in I$ and so $q'p$ is inadmissible. 
Note $M_{a}=0$ if no such $q'$ exists, completing the proof of (2). 
Assuming such a path $q'$ does exist, it must be unique and of minimal length  by \Cref{lem-initial-termina-subpaths-special} and  \Cref{lem-left-uniserial-cyclic-modules-uniqueness}, as required. 
\end{proof}

\Cref{lem-projective-covers-and-their-kernels-for-uniserial-right-modules} follows from a dual argument to that used in the proof of \Cref{lem-projective-covers-and-their-kernels-for-uniserial-left-modules}. 

\begin{lem}
    \label{lem-projective-covers-and-their-kernels-for-uniserial-right-modules}
    Let $I$ be arrow-direct $v=t(p)$ and $a$ be an arrow with head $v$ such that $pa$ is admissible. 
    The following statements hold for the sum ${}_{a}K=\bigoplus b \Lambda$ over arrows $b\neq a$ with $h(b)=v$. 
        \begin{enumerate}
            \item If $pq$ is admissible for every path $q$ with left arrow $a$ then $\ker(\leftmult{p})={}_{a}K$. 
            \item If $q$ is the shortest path with left arrow $a$ and such that $pq$ is inadmissible then $\ker(\leftmult{p})={}_{a}K\oplus q\Lambda $. 
        \end{enumerate}
\end{lem}

\begin{proof}[Proof of \Cref{thm-main-string-biserial}. ]
By assumption $\Lambda=RQ/I$ is a string algebra  over $R$. 
See \Cref{defn-string-pairs-and-algebras}. 
Since  $I$ is arrow-radical the radical of $\Lambda e_{v}$ (respectively, $e_{v}\Lambda$) is the sum of the modules $\Lambda a$ (respectively, $a\Lambda$) where $a$ runs through the arrows with tail (respectively, head) equal to $v$. 
Since every such $v$ is the tail (respectively, head) of at most $2$ such $a$, this means the radical of $\Lambda e_{v}$ (respectively, $e_{v}\Lambda$) is the sum of at most $2$ modules of the form $\Lambda a$ (respectively $a\Lambda$). 
Since $I$ is bounded below and special, by \Cref{thm-main-string-biserial-multi} each of these summands are uniserial. 
Since $I$ is arrow-direct any distinct pair of such summands has a trivial intersection. 
 Thus $\Lambda$ is biserial as a semiperfect noetherian ring; see \S\ref{subsec-biserial}. 

For the remainder of the proof fix an admissible non-trivial path $p$ with left arrow $b$ and  right arrow $a$. 

 (1) Let $\jacrad=\rad(\Lambda)$. 
 By \Cref{uniserial-module-left-right-compatability-one-side} we have $\jacrad p\cap b\Lambda\subseteq p\jacrad$ and by \Cref{uniserial-module-left-right-compatability-other-side} we have $p\jacrad\cap \Lambda a\subseteq \jacrad p$. 

(2) Suppose $p$ has tail $u$ and head $v$. 
By \Cref{rem-projective-covers-by-path-multiplication} the map $\rightmult{p}$ defines a projective cover of $\jacrad e_{u}$.  
Furthermore without loss of generality we can assume $ap$ is admissible for some arrow $a$. 
Let $K_{a}$ be  $0$ if $a$ is the only arrow with tail $v$. 
Otherwise let $K_{a}=\Lambda b$ where $b$ is the unique arrow distinct from $a$ with tail $v$. 

By \Cref{lem-projective-covers-and-their-kernels-for-uniserial-left-modules} the kernel $\ker(\rightmult{p})$ is a direct sum of $K_{b}$ and a module of the form $\Lambda q$ where $q$ is an admissible path of  minimal length such that $qp$ is inadmissible and such that $a$ is the right arrow of $q$. 
 
 (3) 
 Follows a similar argument to the proof of (2), using instead  \Cref{lem-projective-covers-and-their-kernels-for-uniserial-right-modules}.     
\end{proof}

% One can view \Cref{cor-basic-biserial-quiver-is-biserial-plus-dim-1-nice} as motivation for \Cref{biserial-quiver-defn}. 

% \begin{defn}
%     \label{biserial-quiver-defn}

%     DEFINE SPECIAL BISERIAL QUIVERS?
    
% \end{defn}

% SAY SOMETHING ABOUT QUIVER

\section{Examples}
\label{section-examples}
% In \S\ref{section-examples} we recover the examples discussed in the introduction. 

\subsection{String algebras of Butler and Ringel}
\label{sec-string-algebras-butrin}

In \S\ref{sec-string-algebras-butrin} we recover the algebras from \cite{ButRin1987}.

\begin{lem}
\label{lem-special-conditions-and-arrow-radical-over-field-means-monomial}   
Let $k$ be a field and let  $I\triangleleft kQ$ be arrow-radical, arrow-direct and special. 
If $I\subseteq A^{2}$ then $I=\langle Z\rangle$ where $Z$ is the set of inadmissible paths. 
\end{lem}

\begin{proof}
The inclusion $\langle Z\rangle\subseteq I$ follows by construction. 
For a contradiction we suppose $x\in e_{v}kQe_{u}\cap  (I\setminus \langle Z\rangle)$ for some vertices $u$ and $v$. 
Since $\langle Z\rangle\subseteq I\subseteq A^{2}$ we write $x=\sum r_{p}p$ with 
$r_{p}\in k$ supported over a finite subset $P$ of the set of admissible paths $p$ of  length at least $2$ such that $t(p)=u$ and $h(p)=v$. 
% Let $P$ denote the finite set of such paths $p$ with $r_{p}\neq 0$, and so $x=\sum_{p\in P} r_{p}p$. 
For each arrow $a$ with tail $u$ let $x(a)=\sum_{p\in P(a)} r_{p}p$ where $P(a)$ is the subset of paths $p\in P$ with right arrow   $a$. 

Hence $x=\sum x(a)$ where the sum runs over the arrows $a$ with tail $u$. 
Since $x\in I$ and $I$ is arrow-direct we have $x(a)\in I$ for each $a$. 
Since $x\notin \langle Z\rangle$ there is some arrow $b$ with $t(b)=u$ and  $x(b)\notin \langle Z\rangle$. 
% Let $x'=x(b)$, and so $\sum_{p\in P(b)} r_{p}p=x'\notin I$. 

Choose a path $q\in P(b)$ of maximal length such that every element in $P(b)$ has  $q$ as a right subpath. 
By construction there is a finite non-empty set $P'$ of paths such that $P(b)=\{pq\mid p\in P'\}$ and  $x=(\sum_{p\in P'}r_{pq}p)q$. 
Since all the elements of $P'$ have tail $h(q)$ they all have $e_{h(q)}$ as a right subpath. 
Hence as above there exists a path $e$ of maximal length that is a right subpath of every element in $P'$ and this means $eq$ is a right subpath of every element in $P(b)$. 
By maximality of the length of $q$ means $e=e_{h(q)}$.

Consider a pair of non-trivial paths in $P'$ whose right arrows are $c$ and $d$. 
Since $q$ is non-trivial it has a left arrow  $a'$ that gives $ca',da'\notin I$ since the paths $cq,dq$ must be admissible, that then means $c=d$ since $I$ is special. 
Hence $P'$ cannot consist entirely of non-trivial paths, for otherwise they all have the same right arrow $c$, contradicting the maximality of the length of $e$.  
In other words, by the argument above we have $P'\ni e$. 
Let $z=r_{e}^{-1}y$ where $y=\sum_{p\in P'}r_{pq}p$. 
Note $x(b)=yq$. 
This gives $\Lambda z+\sum_{t(a)=v}\Lambda a=\Lambda e$. 
Since $I$ is arrow-radical  $\sum_{t(a)=v}\Lambda a$ is superfluous and so $\Lambda z=\Lambda e$. 

So $e+I=z'z$ for some $z'\in \Lambda$. 
Combining everything so far we have $q\in z'r_{e}^{-1}x(b)+I$.  
% \[
% q\in q+I=eq+I=(e+I)(q+I)=z'zq+I=z'r_{e}^{-1}yq+I=z'r_{e}^{-1}x+I=I.
% \]
Since $x(b)\in I$ this contradicts that $q$ is an admissible path. 
\end{proof}
\Cref{prop:butler-ringel-string-for-a-field} recovers the notion of a string algebra from \cite{ButRin1987}. 

\begin{prop}
\label{prop:butler-ringel-string-for-a-field}
If $k$ is a field then $\Lambda$ is a string algebra over $k$ (in the sense we are using) if and only if $\Lambda$ is a string algebra in the sense of Butler and Ringel \cite{ButRin1987} where the quiver is taken to be finite. 
\end{prop}
\begin{proof}
By \Cref{lem-special-conditions-and-arrow-radical-over-field-means-monomial} any string algebra over $k$ is defined by a monomial ideal in $I\triangleleft kQ$. 
By \Cref{thm-admissible-presentations} $I$ is admissible. 
Thus, assuming $I$ is special and assuming every vertex is the head (respectively, tail) of at most $2$ arrows, any such an algebra is a string algebra in the sense of \cite{ButRin1987}.

Conversley suppose $\Lambda=kQ/I$ is a string algebra in the sense of \cite{ButRin1987}, and take the quiver $Q$ to be finite. 
By conditions $(3)$ and $(3^{*})$ from \cite[p. 157]{ButRin1987} we have that $\Lambda$ is finite-dimensional over $k$. 
For such algebras $I$ is monomial and generated by a set of paths of length at least $2$, meaning the set of inadmissible paths form a $k$-basis for $I$, and hence the set of admissible paths defines a $k$-basis for $\Lambda$. 
In particular $I$ is permissible and arrow-direct. 
See the discussion after the Lemmas in \cite[p. 169]{ButRin1987}. 
Since $I$ is admissible, meaning that $A^{m}\subseteq I\subseteq A^{2}$, we have that $I$ is arrow-radical by \cite[\S II.2, Lemma 2.10]{Ass2006}. 
By $(1)$ (respectively, $(1^{*})$) from \cite[p. 157]{ButRin1987}   each vertex is the head (respectively, tail) of at most $2$ arrows.  
By $(2)$ and $(2^{*})$ we have that $I$ is special. 
This completes the proof that $\Lambda$ is a string algebra over $k$. 
\end{proof}

\subsection{B\"{a}ckstr\"{o}m orders} 
\label{sec-backstrom}

In \S\ref{sec-backstrom} we recall the notion of an   \emph{order} to provide examples of path algebras over local rings in the sequel. 
We begin with a simplifying assumption. 
\begin{setup}
    In \S\ref{sec-backstrom} assume that $R$ is an $\maxideal$-adically complete discrete valuation ring, say where $\maxideal=\langle \pi\rangle$. 
Note $\pi\in R$ is regular and  $R$ is not a field. 
Let $F=\mathrm{Frac}(R)$,  the field of fractions of $R$. 
\end{setup}
Recall that an $R$-\emph{order} is a subring $\Gamma$ of a finite-dimensional $F$-algebra $\Omega$ such that $\Gamma$ is a full $R$-lattice\footnote{Hence $\Gamma$ spans $\Omega$ as a vector-space over $F$, and $\Gamma$ is a finitely-generated $R$-submodule of $\Omega$. Here the $R$-action on $\Gamma$ is the restriction of the $F$-action on $\Omega$ via the embeddings $R\to F$ and $\Gamma\to \Omega$.   } in $\Omega$. 
Recall, say following Ringel and Roggenkamp \cite{RinRog1979}, that a \emph{B\"{a}ckstr\"{o}m order} is a subring $\Lambda$ of a hereditary $R$-order $\Gamma$ such that $\rad(\Lambda)=\rad(\Gamma)$. 
Such rings have a typical form. 

% With the rings $\Gamma,\Omega$ above in mind, we fix some notation. 
Let 
$\boldsymbol{\mathrm{H}}_{n}(R)$ be the subring of $\boldsymbol{\mathrm{M}}_{n}(R)$ consisting of matrices $(r_{ij})$ satisfying $r_{ij}\in \maxideal$ whenever $i<j$. 
% Hence for $n=2$ we have that $\boldsymbol{\mathrm{H}}_{2}(R)=\begin{psmallmatrix}
%     R & \maxideal\\
% R & R
% \end{psmallmatrix}$ is the set of $\begin{psmallmatrix}
%     r_{11} & r_{12}\\
% r_{21} & r_{22}
% \end{psmallmatrix}$ with $r_{11},r_{21},r_{22}\in R$ and $r_{12}\in \maxideal$. 
% Note that $\boldsymbol{\mathrm{H}}_{n}(R)$ is a hereditary $R$-order in the semisimple finite-dimensional $F$-algebra  $\boldsymbol{\mathrm{M}}_{n}(F)$. 
% since the image of $0\neq m\in R$ under the embedding  $R\to F$ must (be non-zero and) have an inverse in $F$. 
% The $R$-orders $\Gamma$ from \S\ref{section-examples} are products of rings of the form $\boldsymbol{\mathrm{H}}_{n}(R)$.
% in a  seperable $F$-algebra $\Omega$. 
% So $\Omega$ is semisimple and finite-dimensional over $F$ by definition. 
% Moreover, in the cases we consider $\Omega$ is a product of $F$-algebras of the form $\boldsymbol{\mathrm{M}}_{n}(F)$). 
The ring $\boldsymbol{\mathrm{H}}_{n}(R)$ is an example of a hereditary $R$-order, and one can check  $\rad(\boldsymbol{\mathrm{H}}_{n}(R))$ is equal to the set of matrices $(r_{ij})\in \boldsymbol{\mathrm{H}}_{n}(R)$ satisfying $r_{ij}\in \maxideal$ whenever $i\leq j$. 
% The rings $\boldsymbol{\mathrm{H}}_{n}(R)$ define prototypical examples of well-studied hereditary $R$-orders.  
See for example work of Roggenkamp \cite[p. 523]{Roggenkamp-rep-theory-blocks} and Roggenkamp and Wiedemann \cite[p. 2528, \S1.4]{roggwied1984auslander}. 
\begin{example}
\label{Backstrom-order-example}
Let $\Lambda$ and $\Gamma$ be the subsets of $\boldsymbol{\mathrm{M}}_{3}(R)$ defined by 
\[
\begin{array}{cc}
\Lambda=
\left\{
\medmatrix{
r_{11}& r_{12} & r_{13}\\
r_{21} & r_{22}  &  r_{23}\\
r_{31} & r_{32}  &  r_{33}
}
% (r_{ij})
\colon r_{ij}\in R,\,r_{11}-r_{33},r_{12},r_{23},r_{13}\in \maxideal
\right\},
&
\Gamma=\boldsymbol{\mathrm{H}}_{3}(R)=\medmatrix{
R & \maxideal & \maxideal\\
R & R  &  \maxideal\\
R & R  &  R
}.
% \\
% \vspace{-2mm}
% \\
% =\left\{
% \medmatrix{
% s_{11}& s_{12} & s_{13}\\
% s_{21} & s_{22}  &  s_{23}\\
% s_{31} & s_{32}  &  s_{33}
% }
% % (r_{ij})
% \colon  s_{ij}\in R,\,
% s_{12},s_{23},s_{13}\in \maxideal
% \right\}.
\end{array}
\]
Let $\boldsymbol{1}$ be the identity matrix in $\boldsymbol{\mathrm{M}}_{3}(R)$. 
It is straightforward to check that the rings $\Lambda$ and $ \Gamma $ are both $R$-subalgebras of $\boldsymbol{\mathrm{M}}_{3}(R)$ under the map $\mathsf{a}\colon R\to \Lambda$ given by $r\mapsto r\boldsymbol{1}$. 
Define ideals  $W$, $X$ and $Y$ of $\Lambda$ by 
\[
\begin{array}{ccc}
W=\medmatrix{
\maxideal & \maxideal & \maxideal\\
R & \maxideal &  \maxideal\\
R & R  &  \maxideal
},
&
X=\Lambda\cap \medmatrix{
R & \maxideal & \maxideal\\
R & \maxideal &  \maxideal\\
R & R  &  R
},
&
Y=\medmatrix{
\maxideal & \maxideal & \maxideal\\
R & R &  \maxideal\\
R & R  &  \maxideal
}.
\end{array}
\]
For example  $X$ is the set of matrices $(r_{ij})\in \Lambda$  such that $r_{22}\in \maxideal$.  
Hence $\rad(\Gamma)=W$ and  $X\cap Y=W$. 
% It is straightforward to check $W$, $X$ and $Y$ are all two-sided ideals of $\Lambda$.  

We claim that $\rad(\Lambda)=W$. 
Let  $\mathsf{g}\colon \Lambda\to \Lambda/X$ and $\mathsf{h}\colon \Lambda\to \Lambda/Y$ be the quotient maps. 
Since $\im(\mathsf{a})\ni \boldsymbol{1}\notin X$ we have $\mathsf{g}\mathsf{a}\neq 0$ and so $\ker(\mathsf{g}\mathsf{a})\subseteq\maxideal$. 
On the other hand $m\boldsymbol{1}\in X$ and so   $\ker(\mathsf{g}\mathsf{a})=\maxideal$. 
Similarly $\ker(\mathsf{h}\mathsf{a})=\maxideal$. 
It is straightforward to check $\mathsf{h}\mathsf{a}$ is onto. 
That $\mathsf{g}\mathsf{a}$ is onto follows from the fact that $r_{11}+\maxideal=r_{33}+\maxideal$ whenever $(r_{ij})\in \Lambda$. 
% Thus $\Lambda/X$ and $\Lambda/Y$ are $1$-dimensional $k$-vector spaces, and so $G$ and $H$ are both maximal as one-sided ideals. 
Hence $\rad(\Lambda)\subseteq X\cap Y=W$. 
Conversley let $(r_{ij})\in W$. 
% Since $W$ is an ideal, to show $(r_{ij})\in \rad(\Lambda)$ it suffices to prove $\boldsymbol{1}+(r_{ij})$ is a unit in $\Lambda$. 
% Since $G=\rad(\Gamma)$ is a subset of $S$ that is closed under the left and right action of $S$ it is a left and right ideal of $S$. 
% Suppose $1-y'\in U(S)$ for all $y'=(r_{ij}) \in J$. Let $y\in J$. Then for any $x\in S$ we have that the element $y'=-xy$ lies in $J$, and so $1+xy$ is a unit in $S$, meaning $y\in \rad(S)$. See https://en.wikipedia.org/wiki/Jacobson_radical. 
Note that 
% $\mathrm{det}(
% \boldsymbol{1}+(r_{ij})
% )$ is equal to
% \[
% % \medmatrix{
% % 1+r_{11}& r_{12} & r_{13}\\
% % r_{21} & 1+r_{22}  &  r_{23}\\
% % r_{31} & r_{32}  &  1+r_{33}
% % }=
% (1+r_{11})(1+r_{22})(1+r_{33})+r_{12}r_{23}r_{31}+r_{13}r_{21}r_{32}
% -(r_{31}(1+r_{22})r_{13}+r_{32}r_{23}(1+r_{11})+(1+r_{33})r_{21}r_{12}).
% \]
% Thus 
$\mathrm{det}(
\boldsymbol{1}+(r_{ij})
)\notin \maxideal$, and so 
% must be a unit in the local ring $R$. 
% Hence the matrix
$\boldsymbol{1}+(r_{ij})$ has an inverse $(t_{ij})$ in $\boldsymbol{\mathrm{M}}_{3}(R)$. 
Now observe that $(t_{ij})\in \Lambda$ using that $(t_{ij})(\boldsymbol{1}-(r_{ij}))=\boldsymbol{1}$. 
% we have 
% \[
% \begin{array}{cc}
% \begin{array}{c}
% 0=t_{11}r_{13}+t_{12}r_{23}+t_{13}(1+r_{33}),
% \\
% 0=t_{11}r_{12}+t_{12}(1+r_{22})+t_{13}r_{32},
% \\
% 0=t_{21}r_{13}+t_{22}r_{23}+t_{23}(1+r_{33}).
% \end{array}
% & 
% \begin{array}{c}
% 1=t_{11}(1+r_{11})+
% t_{12}r_{21}+t_{13}r_{31},
% \\
% 1=t_{31}r_{13}+
% t_{32}r_{23}+t_{33}(1+r_{33}).
% \end{array}
% \end{array}
% \]
% The first equation on the left gives $t_{13}\in \maxideal$ since $r_{13},r_{23},r_{33}\in \maxideal$. 
% Together with the second equation on the left this gives $t_{12}\in \maxideal$ since $r_{12}\in \maxideal$. 
% The third equation on the left gives $t_{23}\in\maxideal$.  
% Combining these observations with the two equations on the right, we have $t_{11}-t_{33}\in \maxideal$. 
% \[
% t_{11}-t_{33}=1-t_{11}r_{11}-
% t_{12}r_{21}-t_{13}r_{31}-1+t_{31}r_{13}+
% t_{32}r_{23}+t_{33}r_{33}\in \maxideal.
% \]
% Thus $\rad(\Lambda)=\rad(\Gamma)$. 
% A similar argument was presumably used in \cite[(b), p. 130]{Fields-Examples-1969}. 
% We revisit the ring $\Lambda$ from  \Cref{Backstrom-order-example} in \S\ref{sec-fields}. 
\end{example}

\subsection{Non-classical examples}
\label{sec-non-classical-examples}

We now give examples outside the scope of \Cref{prop:butler-ringel-string-for-a-field}. 
In \S\ref{sec-fields}--\ref{sec-drozd} we choose different values for $(Q,I)$. 
We exhibit the $R$-algebras  $\Lambda=RQ/I$ as string algebras over $R$ that are found in work of other authors \cite{BurDro2006,Dro2001,Fields-Examples-1969,Gab1987,GelPon1968,Rin1975}. 
In each example we follow steps $1$ to $5$  below. 
\begin{enumerate}
    \item We will make choices for $Q$, $I$,  an order $\Gamma$ and a ring map $\theta\colon RQ\to \Gamma$.  
    \item We will express any element of $\Lambda=RQ/I$ as a coset $y+I$ with $y$ in some canonical form. 
    \item We will prove that $I=\ker(\theta)$ and describe $\im (\theta)$.
    \item We will explain why $\Lambda$ is a B\"{a}ckstr\"{o}m $R$-order.
    \item We will explain why $\Lambda$ is a string algebra over $R$; recall \Cref{defn-string-pairs-and-algebras}. 
\end{enumerate}
% We follow these steps with diligence in \S\ref{sec-fields}. 
% In the remaining examples we give less detail for the sake of brevity, noting that the reader may adapt the given detail in \S\ref{sec-fields}. 
For (1), to define an $R$-algebra homomorphism $\theta\colon RQ\to \Gamma$ it suffices to multiplicatively and $R$-bilinearly extend an assignment on trivial paths $e_{v}$ and arrows $a$, so long as the assignments satisfy 
\[
\begin{array}{ccccc}
  \sum_{v} \theta(e_{v}) =1, & \theta(e_{u})^{2}=\theta(e_{u}), & \theta(e_{u})\theta(e_{w})=0, & \theta(e_{h(a)})\theta(a)=\theta(a) =\theta(a) \theta(e_{t(a)}) & (*)
\end{array}
\]
where the sum runs over vertices and the other equations hold for all vertices $u\neq w$ and arrows $a$. 
% In each example we define the ideal $I$ by means of a generating set. 

For (2) we will find the canonical form $y$ by using generators of $I$ to rewrite cosets of arbitrarily long paths. 
For (3), $I\subseteq \ker(\theta)$ will be clear, and with (2) established the reverse inclusion will be straightforward.   

For (4) we compute jacobson radicals using results in \cite{BurDro2006,Dro2001,Fields-Examples-1969,Gab1987,GelPon1968,Rin1975}.  
For (5), with (3) established we compute the set $P$ of admissible paths by showing $\theta(p)\neq 0$ for any $p\in P$.  
It will then be straightforward by (4) to see that $I$ is arrow-radical, arrow-direct and permissible. 
% We refer to these examples as \emph{non}-\emph{classical} merely to reflect the fact that $R$ is not a field. 
% Recall from \S\ref{sec-backstrom} that we write  $\boldsymbol{\mathrm{H}}_{n}(R)$ for the subring of the matrix ring $\boldsymbol{\mathrm{M}}_{n}(R)$ given by matrices whose upper-triangular entries lie in $\maxideal$. 
% For example, schematically we have $\boldsymbol{\mathrm{H}}_{2}(R)=
% \medmatrix{
% R & \maxideal\\
% R & R
% }$, the $R$-subalgebra of $\boldsymbol{\mathrm{M}}_{2}(R)=\medmatrix{
% R & R\\
% R & R
% }$.  
\subsubsection{Answer of Fields to a question of Kaplansky}
\label{sec-fields}
(1) 
In \S\ref{sec-fields} take $Q$ and $I$ to be
\[
\begin{array}{cc}
Q=
\begin{tikzcd}[column sep=0.7cm, row sep=0.3cm]
1\arrow[out=150,in=210,loop,  "", "a"',distance=0.6cm]
\arrow[r, bend right, swap,"b"]
&
2\arrow[l,"c", swap, bend right]
\end{tikzcd}\,
&
I=\langle a^{2},\,bc,\,\pi e_{1}-acb-cba,\,\pi e_{2}-bac\rangle.
\end{array}
\]
We will establish a presentation of $\Lambda$ so that, when  $R=k\llsquare t \rrsquare $ and $\pi =t$, this  $R$-algebra specifies to an example discussed in work of Fields \cite{Fields-Examples-1969} that provided an answer to a question of  Kaplansky. 
Let $\Gamma= \boldsymbol{\mathrm{H}}_{3}(R)$
% , an $R$-order in the $F$-algebra $\Omega =\boldsymbol{\mathrm{M}}_{3}(F)$. 
% Let 
and
\[
\begin{array}{cc}
\theta(e_{1})=
\medmatrix{
1 & 0 & 0\\
0 & 0 & 0\\
0 & 0 & 1
}
,
\,
\theta(e_{2})= 
\medmatrix{
0 & 0 & 0\\
0 & 1 & 0\\
0 & 0 & 0
},
&
\theta(a)= 
\medmatrix{
0 & 0 & \pi \\
0 & 0 & 0\\
0 & 0 & 0
},
\,
\theta(b)=
\medmatrix{
0 & 0 & 0\\
1 & 0 & 0\\
0 & 0 & 0
},
\,
\theta(c)=
\medmatrix{
0 & 0 & 0\\
0 & 0 & 0\\
0 & 1 & 0
}.
\end{array}
\]
It may be observed that the equations labelled $(*)$ from the start of \S\ref{sec-non-classical-examples} hold. 

(2) 
Let $g_{1}=\pi e_{1}-acb-cba$ and $g_{2}=\pi e_{2}-bac$. 
% Note that $bacb=\pi b-g_{2}b$, that $cbac=\pi c-g_{2}c$ and that $acba=\pi a-cba^{2}-g_{1}a$. 
Note that $bacb=\pi b-g_{2}b$, $cbac=\pi c-g_{2}c$ and $acba=\pi a-cba^{2}-g_{1}a$. 
By induction, for any $n>0$ the following expressions all define elements of $I$: 
\[
\begin{array}{ccc}
b(acb)^{n}-\pi^{n}b,
&
c(bac)^{n}-\pi^{n}c,
&
a(cba)^{n}-\pi^{n}a,
\\
cb(acb)^{n}-\pi^{n}cb,
&
ac(bac)^{n}-\pi^{n}ac,
&
ba(cba)^{n}-\pi^{n}ba,
\\
(acb)^{n+1}-\pi^{n}acb,
&
(bac)^{n+1}-\pi^{n}bac,
&
(cba)^{n+1}-\pi^{n}cba.
\end{array}
\]
So if $\lambda\in \Lambda$ then  
% $\lambda=x+I$ where $x\in RQ$ is an $R$-span of admissible paths. 
% See \S\ref{sec-non-classical-examples}. 
% Combined with the observations above,  any such $\lambda$ has the form  
$\lambda=y+I$ where for some $r_{1},r_{2},r_{a},r_{b},r_{c},r_{ac},r_{cb},r_{ba},r_{acb},r_{bac},r_{cba}\in R$ we have
% that are not of the form
% \[
% b(acb)^{n},
% \,cb(acb)^{n},
% \,(acb)^{n+1},
% \,c(bac)^{n},
% \,ac(bac)^{n},
% \,(bac)^{n+1},
% \,a(cba)^{n}, 
% \,ba(cba)^{n},
% \,(cba)^{n+1}, 
% \]
% for some integer $n>0$. 
% In other words,  $y$ must be an $R$-span of paths 
\[
y=r_{1}e_{1}+r_{2}e_{2}+r_{a}a+r_{b}b+r_{c}c+r_{ac}ac+r_{cb}cb+r_{ba}ba+r_{acb}acb+r_{bac}bac+r_{cba}cba.
\]

(3) 
It is straightforward to check that the generators of $I$ all lie in the kernel $K$ of $\theta$. 
We claim the inclusion $I\subseteq K$ is an equality. 
Applying $\theta$ to the expression $y$ found in (2) we obtain
\[
\theta(y)=
\medmatrix{
r_{1}+\pi r_{acb} & \pi r_{ac} & \pi r_{a}\\
r_{b} & r_{2}+\pi r_{bac} & \pi r_{ba}\\
r_{cb} & r_{c} & r_{1}+\pi r_{cba}
}.
\]
Let $x\in K$. 
Writing $x+I=y+I$ for $y$ above, it follows that $y\in K$ since 
% $y=x+y-x$ and $y-x\in$ 
$I\subseteq K$. 
% We now combine the fact that $y\in K$ and the calculation of the image of $y$ under $\theta$ above. 
From $\theta(y)=0$ we know  $r_{b}=r_{c}=r_{cb}=0$. 
Since $\pi\in R$ is regular we also have that $r_{a}=r_{ac}=r_{ba}=0$ and that $r_{acb}=r_{cba}$. 
It then follows that $y=r_{cba}g_{1}+r_{bac}g_{2}\in I$, and so $I=K$. 
Note $\im(\theta)$ lies inside $S=
\left\{
% \medmatrix{
% r_{11}& r_{12} & r_{13}\\
% r_{21} & r_{22}  &  r_{23}\\
% r_{31} & r_{32}  &  r_{33}
% }
(r_{ij})
\colon r_{11}-r_{33},r_{12},r_{23},r_{13}\in \maxideal
\right\}$.  
% We claim $\im(\theta)=S$ where we let
% define the $R$-subalgebra $S$ of $\Gamma$ by 
Let  
$r_{11}=r+r'\pi$, $r_{33}=r+r''\pi$, $r_{12}=s\pi $, $r_{23}=t\pi $ and $r_{13}=u\pi $ for $r,r',r'',s,t,u\in R$. 
Then
\[
% \medmatrix{
% r_{11}& r_{12} & r_{13}\\
% r_{21} & r_{22}  &  r_{23}\\
% r_{31} & r_{32}  &  r_{33}
% }
% =
% \medmatrix{
% r+r'\pi& s\pi  & u\pi \\
% r_{21} & r_{22}  &  t\pi \\
% r_{31} & r_{32}  &  r+r''\pi
% }
(r_{ij})
=
\theta
(
re_{1}+r_{22}e_{2}+sac+tba+ua+r'acb+r''cba
).
\]
As claimed, when  $R=k\llsquare t\rrsquare$ and $\pi=t$, $\Lambda$ coincides with the ring $S$ from work of Fields \cite[p.129, \S4]{Fields-Examples-1969}.

(4) By \Cref{Backstrom-order-example} we have that $\rad(S)=\rad(\Gamma)$.

(5) 
Let $P$ be the set of paths in $Q$ that do not contain either of $a^{2}$ or $bc$ as a subpath. 
We claim $P$ is the set of admissible paths (in $Q$ with respect to $I$). 
The table below describes the image of $P$ under $\theta$. 

\begin{table}[H]
\begin{tabular}{|c|c|c|c|c|c|}
\hline
Path $p$ 
& 
$\theta(p)$
&
Path $p'$ 
& 
$\theta(p')$
&
Path $p''$ 
& 
$\theta(p'')$\\
\hline
$(acb)^{n}$             & 
$\medmatrix{
\pi^{n} &&& 0 &&& 0\\
0 &&& 0 &&& 0\\
0 &&& 0 &&& 0
}$
&
$b(acb)^{n}$             
&      
$\medmatrix{
0 &&& 0 &&& 0\\
\pi^{n} &&& 0 &&& 0\\
0 &&& 0 &&& 0
}$
&
$cb(acb)^{n}$             &  
$\medmatrix{
0 &&& 0 &&& 0\\
0 &&& 0 &&& 0\\
\pi^{n} &&& 0 &&& 0
}$
\\
\hline
$(bac)^{n}$             & 
$\medmatrix{
0 &&& 0 &&& 0\\
0 &&& \pi^{n} &&& 0\\
0 &&& 0 &&& 0
}$
&
$c(bac)^{n}$             
&          
$\medmatrix{
0 &&& 0 &&& 0\\
0 &&& 0 &&& 0\\
0 &&& \pi^{n} &&& 0
}$
&
$ac(bac)^{n}$             &  
$\medmatrix{
0 &&& \pi^{n+1} & 0\\
0 &&& 0 &&& 0\\
0 &&& 0 &&& 0
}$
\\
\hline
$(cba)^{n}$             & 
$\medmatrix{
0 &&& 0 &&& 0\\
0 &&& 0 &&& 0\\
0 &&& 0 &&& \pi^{n}
}$
&
$a(cba)^{n}$             &   
$\medmatrix{
0 &&& 0 &&& \pi^{n+1}\\
0 &&& 0 &&& 0\\
0 &&& 0 &&& 0
}$
&
$ba(cba)^{n}$             &   
$\medmatrix{
0 &&& 0 &&& 0\\
0 &&& 0 &&& \pi^{n+1}\\
0 &&& 0 &&& 0
}$
\\
\hline
\end{tabular}
\end{table}

It can be observed that $I$ is permissible since $\pi\neq 0$. 
Furthermore, $\pi$ is not nilpotent (since it is regular) and so $\theta(p)\neq 0$ and thus $p\notin I$ for all $p\in P$. 
Likewise $I$ is arrow-direct 
% . Combining these observations with (4),  $I$ is
and arrow-radical. 

\subsubsection{Completion of the local ring of a simple node}
% \subsubsection{An example considered by Gabriel}
\label{sec-gabriel-example}
(1) 
In \S\ref{sec-gabriel-example} take $Q$ and $I$ to be
\[
\begin{array}{cc}
Q=
\begin{tikzcd}[column sep=0.7cm, row sep=0.3cm]
1\arrow[out=150,in=210,loop,  "", "a"',distance=0.6cm]
\arrow[out=330,in=30,loop,  "", "b"',distance=0.6cm]
\end{tikzcd}\,
&
I=\langle ab,\,ba,\,\pi e_{1}-a-b\rangle.
\end{array}
\]
When $R=k\llsquare t \rrsquare $ and $\pi=t$ we shall see below that there is an $R$-algebra isomorphism $\Lambda\cong k\llsquare a,b \rrsquare /\langle ab\rangle$. 
Finite-dimensional $\Lambda$-modules were classified by Gel`fand and Ponomarev \cite{GelPon1968}. 
%a classification problem was solved that gave description of indecomposable finite-dimensional\footnote{That is, modules over $k\llsquare a,b \rrsquare /\langle ab\rangle$  that are finite-dimensional when considered as a $k$-vector spaces.} modules over $k\llsquare a,b \rrsquare /\langle ab\rangle$. 
%See notes from a talk by Gabriel \cite{Gab1987}. 
Let $\Gamma=R\times R$ and define $\theta
\colon 
RQ\to  
\Gamma$ by extending $e_{1}
\mapsto 
(1,1)$, $a
\mapsto 
(\pi,0)$ and $b
\mapsto 
(0,\pi)$. 
% \[
% \begin{array}{ccc}
% ,
% &
% &
% .
% \end{array}
% \] 
It is immediate that $(*)$ from \S\ref{sec-non-classical-examples} holds.

(2) 
Let $g_{1}=\pi e_{1}-a-b$ so that $a^{2}=\pi a-ba-g_{1}a$ and $b^{2}=\pi b-ab-g_{1}b$ in $RQ$. 
So if $\lambda\in \Lambda$ then $\lambda=y+I$ for $y\in RQ$ of the form $y=re_{1}+r_{a}a+r_{b}b$ for some $r_{1}, r_{a}, r_{b}\in R$. 

(3) 
It is straightforward to check that $I\subseteq K$, and we claim $I\supseteq K$. 
We proceed as we did for \S\ref{sec-fields}(3). 
Let $z\in K$. 
Then $z+I\in \Lambda$ has the form  $y+I$ for $y$ as in \S\ref{sec-gabriel-example}(2) above. 
This gives $y\in K$ meaning $r+\pi r_{a}=r+\pi r_{b}=0$. 
Since $\pi$ is regular this gives $r_{a}=r_{b}$ and so $y=-r_{a}g_{1}\in I$. 
Thus $z\in I$ and so  $I=K$. 
It is clear that $\im(\theta)=\{(r,r')\in R\times R\mid r-r'\in \maxideal\}$. 

(4) 
We now specify $R=k\llsquare t\rrsquare$ and $\pi=t$. . 
The condition $r-r'\in \maxideal$ from (3) becomes equivalent to $r(0)=r'(0)$. 
Burban and Drozd \cite[Example 2.2(1)]{BurDro2006} present $\Lambda \cong \im(\theta)\cong k\llsquare a,b \rrsquare /\langle ab\rangle$ as a \emph{nodal algebra}, and this means $\Lambda$ is a B\"{a}ckstr\"{o}m $k\llsquare t\rrsquare$-order, and we have $\rad(\Lambda)=\rad(\Gamma)$. 

(5) 
Cosets of the powers $a^{n}$ and $b^{n}$ ($n\geq 0$) form a basis of the algebra $k[a,b]/\langle ab\rangle$ and there is a $k$-algebra embedding $k[a,b]/\langle ab\rangle\to \Lambda$. 
In particular $I$ is permissible. 
% The admissible paths in $Q$ are clearly these powers together with the trivial path. 
That $I$ is arrow-direct and arrow-radical are straightforward consequences of  $\rad(\im(\theta))=\maxideal\times \maxideal$. 
% This ideal is generated by $(1,0)$ and $(0,1)$ that correspond to the cosets of $a$ and $b$ in $k\llsquare a,b \rrsquare /\langle ab\rangle$. 
Thus $\Lambda \cong k\llsquare a,b \rrsquare /\langle ab\rangle$ is a string algebra over $k\llsquare t \rrsquare$. 

\subsubsection{Completion of the  dihedral algebra}
% \subsubsection{An example considered by Ringel}
\label{sec-running-example}
(1) 
In \S\ref{sec-running-example} take $Q$ and $I$ to be
\[
\begin{array}{cc}
Q=
\begin{tikzcd}[column sep=0.7cm, row sep=0.3cm]
1\arrow[out=150,in=210,loop,  "", "a"',distance=0.6cm]
\arrow[out=330,in=30,loop,  "", "b"',distance=0.6cm]
\end{tikzcd}\,
&
I=\langle a^{2},\,b^{2},\,\pi e_{1}-ab-ba\rangle.
\end{array}
\]
% Recall the quiver $Q$ consisting of a pair of loops  $a$ and $b$ incident at a single vertex $u$. 
% So here the ideal $A\triangleleft RQ$ generated by the arrows is given by $A=\langle a,b\rangle$.
% Recall the (two-sided) ideal $I=\langle a^{2},b^{2},\pi-ab-ba\rangle$ in $RQ$. 
When $R=k\llsquare t\rrsquare $ and $\pi=t$ we shall see below that there is an $R$-algebra isomorphism $\Lambda\cong k\llangle a,b\rrangle/\langle a^{2},b^{2}\rangle $.  
Ringel \cite{Rin1975} solved a classification problem
% \footnote{The proof in \cite{Rin1975} followed the so-called \emph{functorial filtrations} method that was  developed and used in  \cite{GelPon1968}. } 
describing  indecomposable finite-dimensional $\Lambda$-modules. Let
\[
\begin{array}{cccc}
\Gamma=\boldsymbol{\mathrm{H}}_{2}(R), 
& 
\theta(e_{1}
)= 
\medmatrix{
1 & 0\\
0 & 1
},     &  
\theta( a
)= 
\medmatrix{
0 & 1\\
0 & 0
},    &
\theta(b)= 
\medmatrix{
0 & 0\\
\pi & 0
}.     
\end{array}
\]
As in \S\ref{sec-gabriel-example} the equations $(*)$ from \S\ref{sec-non-classical-examples} hold.

(2) 
Let $g_{1}=\pi e_{1}-ab-ba$, so  $aba=\pi a-ba^{2}-g_{1}a$ and $bab=\pi b-ab^{2}-g_{1}b$ and so elements of $\Lambda$ have the form $y+I$ where $y=re_{1}+r_{a}a+r_{b}b+r_{ab}ab+r_{ba}ba$ for some scalars $r,r_{a},r_{b},r_{ab},r_{ba}\in R$. 
The admissible paths are the  sequences alternating in $a$ and $b$.

(3) 
It is straightforward to check that $I\subset K$ where $K=\ker(\theta)$ and so $I\subseteq K$, and we claim $I\supseteq K$. 
Let $z\in K$ and write $z+I=y+I$ for $y$ as in (2).  
% so that $\theta(y)=
% \medmatrix{
% r+\pi r_{ab} & r_{a}\\
% r_{b} & r+\pi r_{ba}
% }
% $. 
As in \S\ref{sec-gabriel-example} we have $y\in K$ so $r_{a}=r_{b}=0$. 
Since $\pi$ is regular we also have $r_{ab}=r_{ba}$ so $y=-r_{ab}g_{1}\in I$ giving $z\in I$ and so $K=I$. 
It is straightforward to check that the image of $\theta$ consists of matrices $(r_{ij})\in \boldsymbol{\mathrm{H}}_{2}(R)$ such that $r_{11}-r_{22}\in \maxideal$. 
See \S\ref{sec-gabriel-example}(3). 

(4) 
Just as we saw in \S\ref{sec-gabriel-example}(4), after  specifying by setting $R=k\llsquare t\rrsquare$ and $\pi=t$, the ring $\Lambda\cong \im(\theta)$ is another example of a nodal ring as exhibited by Burban and Drozd \cite[Example 2.2(2)]{BurDro2006}. 
% We claim  \S\ref{sec-gabriel-example}(4) may be adapted (and simplified) to show $\rad(\im(\theta))=\rad(\Gamma)$ where $R$ is any discrete valuation ring. 
% Recall $\rad(\Gamma)$ consists of $(r_{ij})\in \Gamma$ with $r_{11},r_{22}\in\maxideal$. 

(5) 
In  \Cref{running-2-by-2-example-part-1} we checked that $I$ is arrow-radical. 
In \Cref{running-2-by-2-example-part-2} we checked that $I$ is permissible and calculated the admissible paths. 
In  \Cref{running-2-by-2-example-part-3} we checked that $I$ is arrow-direct. 
It is clear that $I$ is special and bounded below and so the quotient $\Lambda=RQ/I$ is a string algebra over $R$. 
% By \Cref{thm-admissible-presentations} we have that $I$ is admissible. 
% To see this directly,  observe that 
% \[
% \begin{array}{c}
% \begin{array}{cc}
% a^{2},b^{2},\pi-ab-ba\in A+\maxideal Q,  
% &
% \pi a-aba,\pi b-bab\in A^{2}+\maxideal Q\cap A
% \end{array}
% \\
% \begin{array}{ccc}
% aba\in \pi a+I,
% &
% bab\in \pi b+I,
% &
% \pi\in ab+ba+I.
% \end{array}
% \end{array}
% \]
% We now specify, taking $R=\compactdhat_{p}$, the ring of $\pi=p$-adic integers. 
% GET THE PROTOTYPICAL EXAMPLE FROM THESIS

\subsubsection{A B\"{a}ckstr\"{o}m order for a transposition}
\label{sec-roggenkamp}
(1) 
In \S\ref{sec-roggenkamp} take $Q$ and $I$ to be
\[
\begin{array}{ccc}
Q=
\begin{tikzcd}[column sep=0.7cm, row sep=0.3cm]
1
\arrow[r, shift left = 1mm, bend right, swap,"a"]\arrow[r,bend right=90,looseness=1, swap, "b"]
&
2\arrow[l,"c", shift left = 1mm, swap, bend right]
\arrow[l,bend right=90,looseness=1, swap, "d"]
\end{tikzcd}\,
&
I=\langle ad,\,da,\,bc,\,cb,\,\pi e_{1}-ca-db,\,\pi e_{2}-ac-bd\rangle.
\end{array}
\]
We shall show that  $\Lambda$ is an example\footnote{Any B\"{a}ckstr\"{o}m order $\Lambda$ in a  hereditary order $\Gamma$ satisfies that the radical $\rad(P)$ of a projective indecomposable $\Gamma$-module $P$ is again projective and indecomposable, and hence taking radicals gives a permutation of the projective indecomposables; see \cite[\S 1.2]{Roggenkamp-AR-Backstrom-orders}. 
The permutation for \Cref{sec-roggenkamp} turns out to be a transposition.} considered by Roggenkamp \cite{Roggenkamp-AR-Backstrom-orders}.  
Let $\Gamma=\boldsymbol{\mathrm{H}}_{2}(R)\times \boldsymbol{\mathrm{H}}_{2}(R)$   
% , 
% $\Omega =\boldsymbol{\mathrm{M}}_{2}(F)\times \boldsymbol{\mathrm{M}}_{2}(F)$ 
and    
\[
\begin{array}{ccc}
\begin{array}{c}
\theta(e_{1}
) =
\left(
\medmatrix{
1 & 0\\
0 & 0
},
\medmatrix{
1 & 0\\
0 & 0
}
\right),
\\
\vspace{-3mm}
\\
\theta(e_{2}
)=
\left(
\medmatrix{
0 & 0\\
0 & 1
}
,

\medmatrix{
0 & 0\\
0 & 1
}
\right),
\end{array}
&
\begin{array}{c}
\theta(a
)=
\left(
\medmatrix{
0 & 0\\
1 & 0
},
\medmatrix{
0 & 0\\
0 & 0
}
\right),
\\
\vspace{-3mm}
\\
\theta(b
)=
\left(
\medmatrix{
0 & 0\\
0 & 0
},
\medmatrix{
0 & 0\\
1 & 0
}
\right),
\end{array}
&
\begin{array}{c}
\theta(c
)=
\left(
\medmatrix{
0 & \pi\\
0 & 0
}
,
\medmatrix{
0 & 0\\
0 & 0
}
\right),
\\
\vspace{-3mm}
\\
\theta(d
)=
\left(
\medmatrix{
0 & 0\\
0 & 0
}
,
\medmatrix{
0 & \pi\\
0 & 0
}
\right).
\end{array}
\end{array}
\]
As in \S\ref{sec-running-example} it is clear that the equations $(*)$ from \S\ref{sec-non-classical-examples} are satisfied. 

(2) 
Let $g_{1}=\pi e_{1}-ca-db$ and $g_{2}=\pi e_{2}-ac-bd$. 
Hence for any $n>0$ we have
\[
\begin{array}{cccc}
aca = \pi a-adb-ag_{1},     & 
cac = \pi c-cbd-cg_{2},  &
  
bdb = \pi b-bca-ag_{1},    & 
dbd = \pi d-dac-dg_{1}. 
\end{array}
\]
Hence if  $\lambda\in \Lambda$ then for some $r_{1},r_{2}r_{a},r_{b},r_{c},r_{d},r_{ac},r_{ca},r_{bd},r_{db}\in R$ we have 
\[
\begin{array}{cc}
\lambda=y+I, 
& 
y=r_{1}e_{1}+r_{2}e_{2}+r_{a}a+r_{b}b+r_{c}c+r_{d}d+r_{ac}ac+r_{ca}ca+r_{bd}bd+r_{db}db.
\end{array}
\]

(3)
It is straightforward that $I\subseteq K$. 
We claim $I\supseteq K$. 
Let $z\in I$ and write $z+I=y+I$ for $y$ above. 
%, so
%\[
%0=\theta(y)=\left(
%\medmatrix{
%r_{1}+mr_{ca} & mr_{c}\\
%r_{a} & r_{2}+mr_{ac}
%},
%\medmatrix{
%r_{1}+mr_{db} & mr_{d}\\
%r_{b} & r_{2}+mr_{bd}
%}
%\right). 
%\]
Adapting the argument from \S\ref{sec-running-example}(3) gives $y=-r_{ac}g_{1}-r_{bd}g_{2}\in I$ since $\pi$ is regular. 
It is straightforward to show that  $\im(\theta)=\{((r_{ij}),(s_{ij}))\in\boldsymbol{\mathrm{H}}_{2}(R)\times \boldsymbol{\mathrm{H}}_{2}(R)\mid  r_{11}-r_{22},s_{11}-s_{22}\in\maxideal\}$. 

(4) 
By (3) we have that $\Lambda$ arises as an example of a B\"{a}ckstr\"{o}m order considered by  Roggenkamp \cite[Examples 1.5(I)]{Roggenkamp-AR-Backstrom-orders}.  
In particular this means $\rad(\im(\theta))=\rad(\boldsymbol{\mathrm{H}}_{2}(R))\times \rad(\boldsymbol{\mathrm{H}}_{2}(R))$. 
% \[
% \rad(\im(\theta))=\rad(\Gamma)=\rad(\boldsymbol{\mathrm{H}}_{2}(R))\times \rad(\boldsymbol{\mathrm{H}}_{2}(R))=
% \medmatrix{
% \maxideal & \maxideal \\
% R & \maxideal 
% }
% \times 
% \medmatrix{
% \maxideal & \maxideal \\
% R & \maxideal 
% }.
% \]

(5) 
Let $P$ be the set of paths in $Q$ that do not contain any of $ad$, $da$,  $bc$ or $cb$ as a subpath. One can check $P$ is the set of admissible paths by considering the image of $P$. 
As in \S\ref{sec-fields}, likewise (5) follows.  

%\begin{table}[H]
%\begin{tabular}{|c|c|c|c|}
%\hline
%Path $p$ 
%& 
%$\theta(p)$
%&
%Path $p'$ 
%& 
%$\theta(p')$\\
%\hline
%$(ac)^{n}$             
%&      
%$
%\Big(
%\medmatrix{
%0 && 0\\
% 0 && \pi^{n} 
%%}
%,
%%%\medmatrix{
%0 && 0\\
% 0 && 0
%}
%\Big)
%$
%&
%$c(ac)^{n}$             &  
%$
%\Big(
%\medmatrix{
%0 && \pi^{n+1}\\
% 0 && 0 
%}
%,
%%\medmatrix{
%0 && 0\\
% 0 && 0
%}
%\Big)$
%\\
%\hline
%%$(ca)^{n}$             
%&          
%$
%\Big(
%\medmatrix{
%\pi^{n} && 0\\
%% 0 && 0
%}
%,
%\medmatrix{
%0 && 0\\
% 0 && 0
%%}
%\Big)
%$
%&
%$a(ca)^{n}$             &  
%$
%%\Big(
%\medmatrix{
%0 && 0\\
%\pi^{n} && 0 
%}
%,
%\medmatrix{
%%0 && 0\\
% 0 && 0
%}
%\Big)$
%\\
%\hline
%%$(bd)^{n}$             & 
%%$
%\Big(
%\medmatrix{
%0 && 0\\
% 0 && 0 
%}
%,
%\medmatrix{
% 0 && 0\\
%% 0 && \pi^{n} 
%}
%%\Big)$
%&
%$d(bd)^{n}$             & 
%%$
%%\Big(
%\medmatrix{
%%0 && 0\\
% 0 && 0
%}
%,
%\medmatrix{
%0 && \pi^{n+1}\\
% 0 && 0 
%}
%%\Big)$
%\\
%\hline
%$(db)^{n}$             & 
%$
%%\Big(
%\medmatrix{
% 0 && 0\\
% 0 && 0
%},
%\medmatrix{
%\pi^{n} && 0\\
% 0 && 0
%%}
%\Big)$
%&
%$b(db)^{n}$             & 
%$
%\Big(
%\medmatrix{
% 0 && 0\\
% 0 && 0
%%%},
%\medmatrix{
%0 && 0\\
% \pi^{n} && 0
%%}
%\Big)$
%\\
%%\hline
%%\end{tabular}
%%\end{table}
%
%

\subsubsection{Ring of quadratic modules}
\label{sec-drozd}
(1) 
In \S\ref{sec-drozd} take $Q$ and $I$ to be
\[
\begin{array}{ccc}
Q=
\begin{tikzcd}[column sep=0.7cm, row sep=0.3cm]
1\arrow[out=150,in=210,loop,  "", "a"',distance=0.6cm]
\arrow[r, bend right, swap,"b"]
&
2\arrow[l,"c", swap, bend right]
\arrow[out=330,in=30,loop,  "", "d"',distance=0.6cm]
\end{tikzcd}\,
&
I=\langle ac,\,ba,\,db,\,cd,\,\pi e_{1}-a-cb,\,\pi e_{2}-d-bc\rangle.
\end{array}
\]
When $R\cong \compactdhat_{p}$ and $\pi=p$ we shall exhibit $\Lambda$ as a ring considered by Drozd \cite[\S 4]{Dro2001} (for $p=2$), such that  finitely generated modules correspond to certain \emph{quadratic functors} arising in topology. 
See also work of Baues \cite[Proposition 2.2(3)]{Baues-quadratic-functors} and Burban and Drozd \cite[Example 2.2(4)]{BurDro2006}. 

(2) 
Let $\Gamma=R\times 
R\times \boldsymbol{\mathrm{H}}_{2}(R)$  and let $\Omega =F\times 
F\times 
\boldsymbol{\mathrm{M}}_{2}(F)$. 
Let 
\[
\begin{array}{ccc}
\begin{array}{c}
\theta(e_{1}
)=
\left(
1,
0,
\medmatrix{
1 & 0\\
0 & 0
}
\right),
\\
\vspace{-3mm}
\\
\theta(e_{2}
)=
\left(
0,
1,
\medmatrix{
0 & 0\\
0 & 1
}
\right),
\end{array}
&
\begin{array}{c}
\theta(a
)=
\left(
\pi,
0,
\medmatrix{
0 & 0\\
0 & 0
}
\right),
\\
\vspace{-3mm}
\\
\theta(b
)=
\left(
0,
0,
\medmatrix{
0 & 0\\
1 & 0
}
\right),
\end{array}
&
\begin{array}{c}
\theta(c
)=
\left(
0,
0,
\medmatrix{
0 & \pi\\
0 & 0
}
\right),
\\
\vspace{-3mm}
\\
\theta(d
)=
\left(
0,
\pi,
\medmatrix{
0 & 0\\
0 & 0
}
\right).
\end{array}
\end{array}
\]
(3) 
It is straightforward that $I\subseteq K$ where $K=\ker(\theta)$.  
Let $g_{1}=\pi e_{1}-a-cb$ and $g_{2}=\pi e_{2}-d-bc$. 
Note 
\[
\begin{array}{cccc}
bcb=\pi b-db-g_{2}b,
&
cbc=\pi c-ca-g_{1}c,
&
a^{2}=\pi a-cba-g_{1}a,
&
d^{2}=\pi d-bcd-g_{2}d.
\end{array}
\]
So $b(cb)^{n}-\pi^{n}b$, $c(bc)^{n}-\pi^{n}c$, $a^{n+1}-\pi^{n}a$ and $d^{n+1}-\pi^{n}d$ all lie in $I$ for any $n>0$. 
One can see $I=K$ by writing any $\lambda\in \Lambda$ as $\lambda=y+I$ where 
%$y\in RQ$ is an $R$-span of admissible paths that do not have a subpath in $\{bcb,cbc,a^{2},d^{2}\}$. 
%Hence 
for some  $r_{1},r_{2},r_{a},r_{b},r_{c},r_{d},r_{bc},r_{cb}\in R$ we have 
\[
y=r_{1}e_{1}+r_{2}e_{2}+r_{a}a+r_{b}b+r_{c}c+r_{d}d+r_{bc}bc+r_{cb}cb.
\]
%As we have seen in the previous examples, this expression can be used to show $I=K$. 
% We claim $I = K$. 
% Let $z\in K$. 
% We can use the previous discussion above to write $z+I=y+I$ for $y\in RQ$ of the form above. 
% Since $z\in K$ and $y-z\in I\subseteq K$ we have $y\in K$. 
% Applying $\theta$ gives 
% \[\theta(y)=\left(r_{1}+\pi r_{a},r_{2}+\pi r_{d},\medmatrix{
% r_{1}+\pi r_{cb} & \pi r_{c}\\
% r_{b} & r_{2}+\pi r_{bc}
% }\right).
% \] 
% Since $\theta(y)=0$ it is immediate that  $r_{b}=0$. 
% Since also $\pi r_{c}=0$ we have $r_{c}=0$ since $\pi$ is regular. 
% Since $\pi(r_{a}-r_{bc})=-r_{1}+r_{1}=0$ we likewise have $r_{a}=r_{cb}$. 
% Similarly $r_{d}=r_{bc}$. 
% Thus $y=-r_{cb}g_{1}-r_{bc}g_{2}\in I$.

(4) Hence $\theta$ defines a presentation of the $R$-algebra $\Lambda$. 
Namely, it is straightforward to check that 
\[
\im (\theta)=
\left\{
\left(
r,
s,
\medmatrix{
t & \pi u\\
v & w
}
\right)\in 
R\times 
R\times 
\medmatrix{
R & \maxideal\\
R & R
}
\mid
r-t,s-w\in\maxideal
\right\}. 
\]
% To see this equality, writing $r-t=r'\pi$ and $s-w=s'\pi$ for $r',s'\in R$ we have
% \[
% \left(
% t+r'\pi,
% w+s'\pi,
% \medmatrix{
% t & \pi u\\
% v & w
% }
% \right)
% =
% \theta
% (
% te_{1}+we_{2}+r'a+vb+uc+s'd
% ).
% \]

(5) 
The admissible paths here are those that do not have $ac$, $ba$, $db$ or $cd$ as a subpath.  
%This can be seen by constructing a table for the image of $P$ under $\theta$ as we have seen in previous examples. 
% For the remaining paths use that $\pi$ is not nilpotent and consider the images of $a^{n}$,  $d^{n}$, $(bc)^{n}$, $(cb)^{n}$, $(bc)^{n}b$ and  $(bc)^{n}b$. 
% \[
% \begin{array}{ccc}
% \begin{array}{c}
% \theta(a^{n})
% =
% \left(
% \pi^{n},
% 0,
% \medmatrix{
% 0 & 0\\
% 0 & 0
% }
% \right),
% \\
% \vspace{-3mm}
% \\
% \theta(d^{n})=
% \left(
% 0,
% \pi^{n},
% \medmatrix{
% 0 & 0\\
% 0 & 0
% }
% \right),
% \end{array}
% &
% \begin{array}{c}
% \theta((bc)^{n})= 
% \left(
% 0,
% 0,
% \medmatrix{
% 0 & 0\\
% 0 & \pi^{n}
% }
% \right),
% \\
% \vspace{-3mm}
% \\
% \theta((cb)^{n})=
% \left(
% 0,
% 0,
% \medmatrix{
% \pi^{n} & 0\\
% 0 & 0
% }
% \right),
% \end{array}
% &
% \begin{array}{c}
% \theta((bc)^{n}b)=
% \left(
% 0,
% 0,
% \medmatrix{
% 0 & 0\\
% \pi^{n} & 0
% }
% \right),
% \\
% \vspace{-3mm}
% \\
% \theta((cb)^{n}c)= 
% \left(
% 0,
% 0,
% \medmatrix{
% 0 & \pi^{n+1}\\
% 0 & 0
% }
% \right).
% \end{array}
% \end{array}
% \]
Hence and likewise one can also see that $I$ is arrow-direct. 
%For example, to see that \Cref{defn-arrow-direct}(iii) holds, consider the following identifications under $\Lambda \cong \im (\theta)$
%\[
%%\begin{array}{cccc}
%\Lambda a= 
%\left(
%\maxideal,
%0,
%\medmatrix{
%0 & 0\\
%0 & 0
%}
%\right),    
%&
%\Lambda b= 
%\left(
%0,
%0,
%\medmatrix{
%\maxideal & 0\\
%R & 0
%}
%\right), 
%&
%\Lambda c= 
%\left(
%0,
%0,
%\medmatrix{
%0 & \maxideal\\
%0 & \maxideal
%}
%\right), 
%&
%%\Lambda d=    
%\left(
%0,
%\maxideal,
%\medmatrix{
%0 & 0\\
%0 & 0
%}
%\right). 
%\end{array}
%\]
Now let $R=\compactdhat_{2}$, the $2$-adic integers. 
In \cite{BurDro2006} the authors present $\Lambda$ as a \emph{nodal ring} by noting  $\rad(\im(\theta))=\maxideal\times \maxideal\times\rad(\boldsymbol{\mathrm{H}}_{2}(R))$.
% Such rings are exhibited as a class of B\"{a}ckstr\"{o}m orders, in this case taken inside the hereditary order $\Gamma$, the codomain of $\theta$. 
% \[
% \rad(\im(\theta))=\rad(\Gamma)=\rad
% \left(R\times R\times \medmatrix{
% R & \maxideal\\
% R & R
% }\right)
% =\maxideal\times \maxideal\times \medmatrix{
% \maxideal & \maxideal\\
% R & \maxideal
% }.
% \]
% This is precisely the image under $\theta$ of the two-sided ideal in  $\Lambda$ generated by the arrows in $Q$. 
Thus $I$ is arrow-radical.
% and so $\Lambda$ is a string algebra over $\compactdhat_{2}$. 

\begin{acknowledgements}
The author is grateful for  support by the Danish National Research Foundation (DNRF156); 
 the Independent Research Fund Denmark (1026-00050B); 
 and the Aarhus University Research Foundation (AUFF-F-2020-7-16). 
\end{acknowledgements}

\bibliographystyle{abbrv}
\bibliography{biblio}

\end{document}